\documentclass{psbis}
\def\I{{\rm{I\!\!I}}}
\def\leftB{[\![}
\def\rightB{]\!]}

\def\F{{\mathcal{ F}}}

\def\R{{\mathbb{R}} }

\def\N{{\mathbb{N}} }
\def\E{{\mathbb{E}}  }

\def\P{{\mathbb{P}}  }

\def\I{{\mathbb{I}}}
\def\det{{\rm{det}}}
\def\Tr{{\rm{Tr}}}
\def\bint#1^#2{\displaystyle{\int_{#1}^{#2}}}
\def\bsum#1^#2{\displaystyle{\sum_{#1}^{#2}}}

\def\supp{{\rm{supp}}}
\def\xdt_#1{X_#1(\Delta t)}

\newtheorem{THM}{Theorem}
\newtheorem{PROP}{Proposition}
\newtheorem{LEMME}{Lemma}

\newtheorem{REM}{Remark}

\newcommand{\mysection}{\setcounter{equation}{0} \section}

\usepackage{color}
\newcommand \A[1]{{\bf (#1)}}

\begin{document}
\title{Weak Error for the Euler Scheme Approximation of Diffusions with non-smooth coefficients}\thanks{The article was prepared within the framework of a subsidy granted to the HSE by the Government of the Russian Federation for the implementation of the Global Competitiveness Program.
}
\author{V. Konakov}\address{Higher School of Economics, Shabolovka 28, Building 1, Moscow, Russian Federation. vkonakov@hse.ru}
\author{S. Menozzi}\address{Higher School of Economics, Shabolovka 28, Building 1, Moscow, Russian Federation and  LaMME, UMR CNRS 8070, Universit\'e d'Evry Val d'Essonne, 23 Boulevard de France, 91037 Evry, France. stephane.menozzi@univ-evry.fr}
\date{\today}
\begin{abstract} 
We study the weak error associated with the Euler scheme of  non degenerate diffusion processes with non smooth bounded coefficients. Namely, we consider the cases of H\"older continuous coefficients as well as piecewise smooth drifts with smooth diffusion matrices.
\end{abstract}
%
%
\maketitle

\mysection{Introduction}
\subsection{Setting.}
Let $T>0$ be a fixed given deterministic final horizon and $x\in \R^d$ be an initial starting point. We consider the following multidimensional SDE:
\begin{align}
& X_t=x+\int_0^t b(s,X_s)ds+\int_0^t\sigma (s,X_s)dW_s,\ t\in [ 0,T],
\label{1} 
\end{align}
where the coefficients $b:[0,T]\times \mathbb{R}^{d}\rightarrow \mathbb{R}^{d},\  \sigma:[0,T]\times \mathbb{R}^{d}\rightarrow \R^d\otimes \R^d$ are bounded  measurable in time and  space and 
$W $ is a Brownian motion on some filtered probability space $(\Omega,\F,(\F_t)_{t\ge 0},\P) $. 
We assume that the diffusion matrix $a(t,x):=\sigma\sigma^*(t,x)$ is uniformly elliptic and at least H\"older continuous in time and space. We will consider two kinds of assumptions for the drift coefficient $b$: either H\"older continuous in time and space (as for the diffusion matrix), or piecewise smooth and having at most a finite set of spatial discontinuities.
These assumptions guarantee that \eqref{1} admits a unique weak solution, see e.g. Bass and Perkins \cite{bass:perk:09}, \cite{meno:10} from which the uniqueness to the martingale problem for the associated generator can be derived under the current assumptions.

Define now for a given $N\in \N^* $ the time step $h:=T/N$ and set for all $i\in \leftB 1,N \rightB,\ t_i:=ih $ where from now on the notation $\leftB \cdot, \cdot\rightB $ is used to denote an interval of integers.
Consider the continuous Euler scheme associated with \eqref{1} whose dynamics writes $X_0^h=x $ and for all $t\in [0,T] $:
\begin{equation}
\label{EULER}
X_t^h=x+\int_0 ^t b(\phi(u),X_{\phi(u)}^h)du+\int_0 ^t \sigma(\phi(u),X_{\phi(u)}^h)dW_u,
\end{equation}
where we set $\phi(u)=\inf\{ (t_i)_{i\in \leftB 0,N-1\rightB}: t_i\le u<t_{i+1}\}$.

A useful quantity to study, arising in many applicative fields from physics to finance, is the so-called weak error which for a suitable real valued test function $f$ writes:
\begin{equation}
\label{WEAK}
d(f,x,T,h):=\E[f(X_T^{h,0,x})]-\E[f(X_T^{0,x})], 
\end{equation}
using the usual Markovian notations, i.e. $X_{T}^{h,0,x}, X_{T}^{0,x}$ respectively stand for the Euler scheme and the diffusion at time $T$ which start at point $x$ at $0$.

There is a huge literature concerning the weak error for smooth and/or non-degenerate coefficients, from the seminal paper of Talay and Tubaro \cite{tala:tuba:90}, to the extensions to the hypoelliptic framework \cite{ball:tala:96:1}. Under those conditions, the quantity $d(f,x,T,h)$ is of order $h$ corresponding to the magnitude of the time step.
In the non degenerate framework (under some uniform ellipticity or hypoellipticity conditions) it is even possible to take $f$ to be a Dirac mass in the above expression \eqref{WEAK}. The associated convergence rate remains of order $h$ for the Euler scheme, see \cite{kona:mamm:02}
\cite{ball:tala:96:2} and  $h^{1/2} $ in the more general case of 
  Markov Chain approximations, see e.g \cite{kona:mamm:00} in which the Brownian increments appearing in \eqref{EULER} are replaced by i.i.d. sequences $(\xi_i)_{i\ge 1} $ that are not necessarily Gaussian. In the framework of Lipschitz coefficients we can also mention, in the scalar case, the recent work of Alfonsi \textit{et al.} \cite{alfo:jour:koha:14}, who obtained bounds on the Wasserstein distances between the laws of the paths of the diffusion and its Euler scheme. 
Anyhow, the case of non smooth coefficients, H\"older continuous or less, has rarely been considered. Such cases might anyhow appear very naturally in many applications, when the drifts have for instance discontinuities at some given interfaces or when the diffusion coefficients are very irregular (random media).

In the framework of bounded non degenerate and H\"older continuous coefficients, let us mention 
the work of Mikulevi\v{c}ius and Platen \cite{miku:plat:91} who obtained bounds for the weak error in \eqref{WEAK} at rate $h^{\gamma/2} $ where $\gamma\in (0,1) $ is the H\"older exponent of the coefficients $b,\sigma$ in \eqref{1} provided $f\in C_b^{2+\gamma}(\R^d,\R)$ (space of bounded functions with bounded derivatives up to order two and $\gamma $-H\"older continuous second derivatives). This regularity is essential in that work to apply It\^o's formula. 
Our approach permits to establish that this bound holds true, up to an additional slowly varying factor in the exponent, for the difference of the densities itself,
 which again corresponds to the weak error \eqref{WEAK} for a $\delta $-function. We also mention the recent work of Mikulevi\v{c}ius \textit{et al.} \cite{miku:12}, \cite{miku:zhan:15}, concerning some extensions of \cite{miku:plat:91} to jump-driven SDEs with H\"older coefficients.

Finally, concerning numerical schemes for diffusions with non-regular coefficients, we refer to the recent work of Kohatsu-Higa \textit{et al.} \cite{koha:leja:yasu:15} who investigate the weak error for possibly discontinuous drifts and diffusion coefficients that are just continuous. We are able to extend some of their controls to densities. Indeed, in the quoted work, the authors investigate \eqref{WEAK} for functions $f$ that are at least continuous. We again have an additional slowly varying factor in the exponent which is due to our smoothing approach.

Our strategy is the following. Under the previous assumptions (stated after \eqref{1}), both processes $(X_t)_{t\in (0,T]}$ in \eqref{1} and $(X_{t_i}^h)_{i\in \leftB 1,N\rightB} $ in \eqref{EULER} have densities, see e.g. \cite{kona:kozh:meno:15:1} for the continuous process and Lemaire and Menozzi \cite{lema:meno:10} for the scheme. Let us denote them respectively for $x\in \R^d$, $0\le i<j\le N $, by $p(t_i,t_j,x,.) $ and $p^h(t_i,t_j,x,.) $ for the processes starting at time $t_i$ from point $x$ and considered at time $t_j$. To study the error $(p-p^h)(t_i,t_j,x,y) $ we introduce perturbed dynamics associated with \eqref{1} and \eqref{EULER} respectively. Namely, for a small parameter $\varepsilon$, we mollify suitably the coefficients, the mollification procedure is described in its whole generality in Section \ref{SEC_MOL} and depends on the two considered sets of assumptions indicated above, and consider two additional processes with dynamics:
\begin{equation}
\label{PERT_DYN}
\begin{split}
X_t^{(\varepsilon)}=x+\int_0^t b_\varepsilon(s,X_s^{(\varepsilon)})ds+\int_0^t \sigma_\varepsilon(s,X_s^{(\varepsilon)})dW_s,\\
X_0^{h,(\varepsilon)}=x,\ X_{t_{i+1}}^{h,(\varepsilon)}=X_{t_i}^{h,(\varepsilon)}+b_\varepsilon(t_i,X_{t_i}^{h,(\varepsilon)})h+
\sigma_\varepsilon(t_i,X_{t_i}^{h,(\varepsilon)})(W_{t_{i+1}}-W_{t_i}),\end{split}
\end{equation}
where $b_\varepsilon,\sigma_\varepsilon $ are mollified versions of $b,\sigma$. It is clear that both $(X_t^{(\varepsilon)})_{t\in (0,T]} $ and $(X_{t_i}^{h,(\varepsilon)})_{i\in \leftB 1,N\rightB}  $ have densities. The mollified coefficients indeed satisfy uniformly in the mollification parameter the previous assumptions. Let us denote those densities for $x\in \R^d, 0\le t_i<t_j\le T $ by $p_{\varepsilon}(t_i,t_j,x,.),\ p_{\varepsilon}^h(t_i,t_j,x,.)  $ respectively. 

The idea is now to decompose the global error as:
\begin{equation}
\label{DECOMP_ERROR}
\begin{split}
(p-p^h)(t_i,t_j,x,y)=(p-p_\varepsilon)(t_i,t_j,x,y)+(p_\varepsilon-p_{\varepsilon}^h)(t_i,t_j,x,y)+(p_\varepsilon^{h}-p^h)(t_i,t_j,x,y).
\end{split}
\end{equation}
The key point is that the stability of the densities with respect to a perturbation has been thoroughly investigated for diffusions and Markov Chains in Konakov \textit{et al.} \cite{kona:kozh:meno:15:1}. The results of that work allow to control the differences $p-p_\varepsilon $, $p_\varepsilon^h-p^h $. On the other hand, since the coefficients $b_\varepsilon,\sigma_\varepsilon $ of $(X_t^{(\varepsilon)})_{t\in [0,T]} $, $(X_{t_i}^{h,(\varepsilon)})_{i\in \leftB 0,N \rightB}  $ are smooth the central term $p_\varepsilon-p_{\varepsilon}^h $ in \eqref{DECOMP_ERROR} can be investigated thanks to the work of Konakov and Mammen \cite{kona:mamm:02} giving the error expansion at order $h$ on the densities for the weak error. The key point is that the coefficients in the expansion depend on the derivatives of $b_\varepsilon,\sigma_\varepsilon $ which explode when $\varepsilon $ goes to zero. This last condition is natural in order to control $p-p_\varepsilon,p_\varepsilon^h-p^h $. Thus, two contributions need to be equilibrated to derive the global error bounds. This will be done through a careful analysis of the densities (heat kernel) of the processes with dynamics described in  \eqref{1}, \eqref{EULER}, \eqref{PERT_DYN}. The estimates required for the error analysis will lead us to refine some bounds previously established by Il'in \textit{et al} \cite{ilin:kala:olei:62}. Let us indicate that this perturbative approach had also been considered by Kohatsu-Higa \textit{et al.} \cite{koha:leja:yasu:15} but for the weak error \eqref{WEAK} involving at least a continuous function. Our approach, based on parametrix techniques, allows to handle directly the difference of the densities, and gives, up to an additional factor going to zero with the time step, the \textit{expected} convergence rates.


\subsection{Assumptions and Main Results.}

Let us introduce the following assumptions. 
\begin{trivlist}
\item[(\textbf{A1)}] \textbf{(Boundedness of the coefficients)}. The components of the vector-valued function $b(t,x)$ and the
matrix-valued function $\sigma(t,x)$ are bounded measurable. Specifically, there exist constants $K_1,K_2>0$  s.t.
\begin{eqnarray*}
\sup_{(t,x)\in [0,T]\times \R^d}|b(t,x)|\le K_1,\
\sup_{(t,x)\in [0,T]\times \R^d}|\sigma(t,x)|\le K_2.
\end{eqnarray*}
\item[(\textbf{A2})] \textbf{(Uniform Ellipticity)}. The diffusion matrix $a:=\sigma\sigma^*$ is uniformly elliptic, i.e. there exists $\Lambda \ge 1,\ \forall (t,x,\xi)\in [0,T]\times (\R^d)^2$,
\begin{eqnarray*}
\Lambda^{-1} |\xi|^2 
\le \langle a(t,x)\xi,\xi\rangle \le \Lambda |\xi|^2.
\end{eqnarray*}
\end{trivlist}
We consider two types of \textit{smoothness} assumptions for the coefficients $b,\sigma$ in \eqref{1}.
\begin{trivlist}
\item[\textbf{(H)}] \textbf{(H\"older drift and diffusion coefficient)}. The drift $b$ and the diffusion coefficient $\sigma $ are  time-space H\"older continuous in the following sense: for some $ \gamma \in (0,1]$ , $\kappa<+\infty $, for all $(s,t)\in [0,T]^2, (x,y)\in (\R^d)^2$,
\begin{eqnarray*}
\left\vert \sigma(s,x)-\sigma(t,y)\right\vert+\left\vert b(s,x)-b(t,y)\right\vert 
\leq 
\kappa\{ |s-t|^{\gamma/2}+\left\vert x-y\right\vert ^{\gamma }\}.
\end{eqnarray*}
Observe that the last condition also readily gives, thanks to the boundedness of $\sigma$, that the diffusion matrix $a=\sigma\sigma^*$ enjoys the same H\"older property.\\

\item[\textbf{(PS)}] \textbf{(Piecewise smooth drift and Smooth diffusion coefficient)}. The drift $b$ is piecewise smooth with bounded derivatives outside of the discontinuity sets. Precisely,  $b\in C_b^{2,4}([0,T]\times (\R^d\backslash {\mathcal I}),\R^d)$ where the set of possible discontinuities ${\mathcal I}$ writes as $ {\mathcal I}:=\cup_{i=1}^{m}{\mathcal S}_i$,  $m\in \N $. Here, for all $i\in \leftB 1,m\rightB $, ${\mathcal S}_i$ is a smooth bounded submanifold of $\R^d$ (at least $C^4$) of dimension lower or equal to $d-1$, i.e.
${\mathcal S}_i:=\{ x\in \R^d: g_i(x)=0\} $ for a corresponding smooth function $g_i$. We also assume that the $({\mathcal S}_i)_{i\in \leftB 1,m\rightB} $ do not intersect: for all $1\le i<j\le m,\ {\mathcal S}_j\cap {\mathcal S}_i=\emptyset $. 

On the other hand we assume that the diffusion coefficient $\sigma $ is \textbf{globally} $C_b^{2,4}([0,T]\times \R^d,\R^d\otimes \R^d) $.
\end{trivlist}

We emphasize that, with the above definition, the discontinuity set of $b$ only depends on the spatial variable. 
A time-dependent discontinuity set could \textit{a priori} also be considered provided each of it components is the boundary of a smooth time-space domain. Namely, considering for $i\in\leftB 1,m\rightB, t\in [0,T], {\mathcal S}_i(t):=\{x\in \R^d: g_i(t,x)=0\} $, the smooth spatial submanifolds $ {\mathcal S}_i(t)$ should as well evolve smoothly in time.
We consider the case introduced in \textbf{(A${}_{PS}$)} for simplicity.

From now on, we always assume conditions \textbf{(A1)}-\textbf{(A2)} to be in force. We say that assumption \textbf{(A${}_{H}$)} (resp. \textbf{(A${}_{PS}$)}) holds if additionally the coefficients satisfy \textbf{(H)} (resp. \A{PS}).  We will write that \textbf{(A)} holds whenever \textbf{(A${}_{H} $)} or \textbf{(A${}_{PS} $)} is satisfied.

We will denote, from now on, by $C$ a constant depending on the parameters appearing in \textbf{(A)} and $T$. We reserve the notation $c$ for constants that only depend on \textbf{(A)} but not on $T$. The values of $C,c$ may change from line to line. Other possible dependencies will be explicitly specified.

\begin{THM}[Error for the Euler scheme of a diffusion with H\"older coefficients]
\label{THM_HOLDER_ES}
Let $T>0$ be fixed and consider a given time step $h:=T/N$, for $N\in\N^*  $. Set for $i\in \N,\ t_i:=ih $.
Under (\textbf{A${}_{H} $}), there exist $C\ge 1,c\in (0,1]$ s.t. for all $0\le t_i<t_j\le T$ s.t. $(t_j-t_i)\ge h^{1/(2-\gamma)}$and $(x,y)\in (\R^d)^2$:
\begin{equation}
\label{ERR_EUL}
p_c(t_j-t_i,y-x)^{-1}|(p-p^h)(t_i,t_j,x,y)|\le \frac{C}{(t_j-t_i)^{(1-\gamma)\gamma/2}}h^{\frac{\gamma}{2}-C\psi(h)},
\end{equation}
where $p,p^h $ respectively stand for the densities of the diffusion $X$ and its Euler approximation $X^h $ with time step $h$, for all $(t,z)\in \R_*^{+}\times \R^d$, $ p_c(t,z):=\frac{c^{d/2}}{(2\pi t)^{d/2}}\exp(-c\frac{|z|^2}{2t}) $ and $\psi(h)=\frac{\log_3(h^{-1})}{\log_2(h^{-1})} $ where $\log_k $ denotes for $k\in \N^* $ the $k^{{\rm th}} $ iterated logarithm. Let us observe that $\psi(h)\underset{h\rightarrow 0}{\longrightarrow }0 $.
If we are now interested in the \textit{weak error} in the sense of \eqref{WEAK}, 
for a function $f\in C^\beta(\R^d,\R)$ (uniformly $\beta $-H\"older continuous functions), $ \beta\in (0,1]$:
\begin{equation}
\label{WEAK_HOLDER}
|\E[f(X_{t_j}^{h,t_i,x})]-\E[f(X_{t_j}^{t_i,x})]|\le C_f h^{\gamma/2},
\end{equation}
using again the usual Markovian notations, i.e. $X_{t_j}^{h,t_i,x}, X_{t_j}^{t_i,x}$ respectively stand for the Euler scheme and the diffusion at time $t_j$ which start at point $x$ at $t_i$.

Eventually, if we consider a smooth domain $A\subset \R^d$ (i.e. a connected open set at least $C^2$) with bounded boundary and non zero Lebesgue measure, 
 we also get that for all $x\in\R^d$ s.t. $d(x,\partial A)\ge (t_j-t_i)^{1/2}h^{\gamma/2}$:
\begin{equation}
\label{HOLDER_BOREL} 
|\E[\I_{X_{t_j}^{h,t_i,x}\in A}]-\E[\I_{X_{t_j}^{t_i,x}\in A}]|\le C\Big\{\frac{1}{\gamma d(x,\partial A)^\gamma}\I_{d(x,\partial A)\ge \exp(-\frac 1\gamma)}+ |\ln (d(x,\partial A))|\I_{d(x,\partial A)< \exp(-\frac 1\gamma)}+1\Big\} h^{\gamma/2},
\end{equation}
where $d(.,\partial A)$ stands for the distance to the boundary of $A$.
\end{THM}

\begin{REM}\label{REM_INT_HOLDER}
We point out that this result is to be compared with the one
obtained by
Mikulevi\v{c}ius and Platen \cite{miku:plat:91} for the weak error. The framework they considered is similar to ours, and their main results consists in controlling at rate $h^{\gamma/2} $ the \textit{weak} error $d(f,x,T,h)=\E[f(X_T^{h,0,x})]-\E[f(X_T^{0,x})] $ for a \textit{smooth} function $f\in C_b^{2+\gamma}(\R^d,\R)$ (space of bounded functions, with bounded derivatives up to order two and $\gamma $-H\"older continuous second derivatives). The above theorem establishes that $|d(f,x,T,h)|\le Ch^{\gamma/2-C\psi(h)} $ as soon as $f$ is measurable and satisfies the growth condition 
\begin{equation}
\label{Growth}
\exists c_0<c/(2T),C_0>0,\forall x\in \R^d,\ |f(x)|\le C_0
\exp(c_0|x|^2)
.
\end{equation}
This control can be useful for specific and relevant applications, like for instance quantile estimation (that would involve functions of the form $f(x)=\I_{|x|\le K}$ or $f(x)=\I_{|x|\le K}\exp(c|x|) $) that appear in many applications: default probabilities in mathematical finance, fatigue of structures in random mechanics. We are able to find the expected convergence rate up to a vanishing contribution. The rate $h^{\gamma/2} $ again holds, without the additional term, as soon as $f\in C^{\beta}(\R^d,\R),\ \beta \in (0,1]$. Some extensions to unbounded functions $f$ satisfying the growth condition \eqref{Growth} are described in Remark \ref{RQ_UNBOUNDED} of Section \ref{ERR_FAIBLE_INT}.  

The contribution in $ \psi(h)$ appearing in \eqref{ERR_EUL}, which slightly deteriorates the convergence, seems to be, with our approach, the price to pay to get rid of any smoothness on $f$. Observe anyhow that for indicator functions of smooth Borel sets, equation \eqref{HOLDER_BOREL} provides a better result than \eqref{ERR_EUL} as soon as the initial distance to the boundary satisfies $d(x,\partial A)\ge (t_j-t_i)^{1/2}h^{\gamma/2} $ (see Section \ref{SUBSEC_HODLER_BOREL} for details). Observe that this control improves in that case what could be derived from \cite{koha:leja:yasu:15} in which continuous test functions are considered.
\end{REM}

\begin{REM}[About the Convergence Rate]
We also emphasize that the convergence rate in $h^{\gamma/2} $ is closer to a rate associated with a \textit{strong error}. It indeed corresponds to the typical magnitude of the quantity
$\E[|W_h|^\gamma]\le c_\gamma h^{\gamma/2} $, which reflects the variations, on one time-step of length $h$, of the Euler scheme with H\"older coefficients. Indeed, under \A{A${}_H $}, for all $i\in \leftB 0,N-1 \rightB:$
\begin{eqnarray}
\E[\sup_{u\in [t_i,t_{i+1}]} |b(u,X_u^h)-b(t_i,X_{t_i}^h)|]+\E[\sup_{u\in [t_i,t_{i+1}]} |\sigma(u,X_u^h)-\sigma(t_i,X_{t_i}^h)|]\le \kappa\left\{h^{\gamma/2} +\E[\sup_{u\in [t_i,t_{i+1}]} |X_u^h-X_{t_i}^h|^\gamma]\right\}\nonumber\\
\le \kappa\left\{h^{\gamma/2} +\E[\{\sup_{u\in [t_i,t_{i+1}]} |\sigma(t_i,X_{t_i}^h) (W_u-W_{t_i})|+K_1 h \}^2 ]^{\gamma/2} \right\}\le ch^{\gamma/2}.\nonumber\\
\label{FLAVOUR_STRONG}
\end{eqnarray}
These terms typically appear in the error analysis when there is low regularity of the coefficients or of the value function $v(t,x):=\E[f(X_{T}^{t,x})] $. Under the previous assumptions, if the function $f$ belongs to $C_b^{2+\gamma}(\R^d,\R),\ \gamma\in (0,1)$ it is then well known, see e.g. Friedman \cite{frie:64} or Ladyzhenskaya \textit{et al.} \cite{lady:solo:ural:68} that $v\in C_b^{1+\gamma/2,2+\gamma}([0,T]\times \R^d,\R)$. Also $v$ satisfies the parabolic PDE $(\partial_t v+L_tv)(t,x)=0,\ (t,x)\in [0,T)\times \R^d $, where $L_t$ stands for the generator of \eqref{1} at time $t$, i.e. 
for all $\varphi\in C_0^2(\R^d,\R), x\in \R^d $,
$$L_t\varphi(x)=b(t,x)\cdot \nabla_x \varphi(x)+\frac 12 \Tr(a(t,x)D_x^2 \varphi(x)).$$
Recalling that $t_0=0, t_N=T$, we decompose the error as:
\begin{equation}
\begin{split}
d(f,x,T,h)&:=\E[f(X_T^{h,0,x})]-\E[f(X_T^{0,x})]=\sum_{i=0}^{N-1} \E[v(t_{i+1},X_{t_{i+1}}^{h,0,x})-v(t_{i},X_{t_{i}}^{h,0,x})]\\
       &=\sum_{i=0}^{N-1} \E\Big[\int_{t_i}^{t_{i+1}}\Big\{\partial_s v(s,X_{s}^{h,0,x})+\nabla_x v(s,X_{s}^{h,0,x}) \cdot b(t_i,X_{t_i}^{h,0,x})+\frac{1}{2}{\rm Tr}(D_x^2v(s,X_s^{h,0,x})a(t_i,X_{t_i}^{h,0,x}))\Big\} ds\Big]\\
       &=\sum_{i=0}^{N-1} \E\Big[\int_{t_i}^{t_{i+1}}\Big\{\partial_s v+L_s v
       \Big\} (s,X_{s}^{h,0,x})ds\Big]
       +\E\Big[\int_{t_i}^{t_{i+1}} \Big\{\nabla_x v(s,X_{s}^{h,0,x}) \cdot(b(t_i,X_{t_i}^{h,0,x})- b(s,X_{s}^{h,0,x}))\\
       &\hspace*{15pt}+
       \frac{1}{2}{\rm Tr}(D_x^2v(s,X_s^{h,0,x})(a(t_i,X_{t_i}^{h,0,x})-a(s,X_{s}^{h,0,x})))\Big\}ds\Big]\\
       &=\sum_{i=0}^{N-1} \E\Big[\int_{t_i}^{t_{i+1}} \Big\{\nabla_x v(s,X_{s}^{h,0,x}) \cdot(b(t_i,X_{t_i}^{h,0,x})- b(s,X_{s}^{h,0,x}))\\
       &\hspace*{15pt}+\frac{1}{2}{\rm Tr}(D_x^2v(s,X_s^{h,0,x})(a(t_i,X_{t_i}^{h,0,x})-a(s,X_{s}^{h,0,x})))\Big\}ds\Big],
 \end{split}
\label{DECOUP_ERR_EDP}
\end{equation}
exploiting the PDE satisfied by $v$ for the last equality. For a function $f$ in $C_b^{2+\gamma}(\R^d,\R)$, the spatial derivatives of $v$ up to order two are globally bounded on $[0,T] $. Indeed, the classical Schauder estimates hold (see e.g. Theorem 5.2, p. 361 in \cite{lady:solo:ural:68}). We are thus led to control in \eqref{DECOUP_ERR_EDP} quantities similar to those appearing in \eqref{FLAVOUR_STRONG}. The associated bound then precisely gives the convergence rate. The analysis extends if $f $ is simply $C^{\beta}(\R^d,\R),\ \beta \in (0,1]$ and therefore possibly unbounded. In that case the second derivatives yield an integrable singularity in time for the second order partial derivatives. We refer to Proposition \ref{THE_FINAL_PROP}, which holds under the sole assumption \A{A${}_H$} for multi-indices $\alpha,\ |\alpha|\le 2 $, and to the proof of Theorem \ref{THM_HOLDER_ES} in Section \ref{SEC_PROOF_THM_1}. Extensions to locally $\beta $-H\"older functions $f$ satisfying the growth condition \eqref{Growth} are discussed in Remark \ref{RQ_UNBOUNDED}.
\end{REM}

\begin{REM}
Even though we have considered $\gamma \in (0,1] $, our analysis should extend to the framework of H\"older spaces to $\gamma\in (1,2] $. 
 On the other hand, Theorem \ref{THM_HOLDER_ES} specifies the time-singularity in small time. 

\end{REM}

\begin{REM}
\label{TRUE_HOLDER}
We feel that the bounds of Theorem \ref{THM_HOLDER_ES} are \textit{relevant} for functions which are \textbf{truly} H\"older continuous, that is for coefficients that would involve some simple transformations of the Weierstrass functions, see e.g. \cite{zygm:36}, or of an independent Brownian sample path in order that \A{A${}_H$} is fulfilled. Indeed, for functions which are just \textit{locally} H\"older continuous, like the mapping $x\mapsto 1+ |x|^\alpha\wedge K $, $\alpha \in (0,1] $, we think that it would be more appropriate to study some local regularizations, close to the neighborhoods of \textbf{real} H\"older continuity (0 and $K^{1/\alpha}$ for the indicated example) and to exploit that, outside of these neighborhoods, the \textbf{usual} sufficient smoothness is available. For such coefficients we think that the convergence rates might be definitely better. 
\end{REM}

\begin{THM}[Error for the Euler Scheme with Smooth Diffusion Coefficients and Piecewise Smooth Drift]
\label{THM_HOLDER_ES_PSD}
Let $T>0$ be fixed and (\textbf{A${}_{PS} $}) be in force. 
With the notations of Theorem \ref{THM_HOLDER_ES} we have that:
\begin{trivlist}
\item[-] there exist $C\ge 1,c\in (0,1]$ s.t. for all $0\le t_i<t_j\le T$ s.t. $(t_j-t_i)\ge h^{1/2}
$ and $(x,y)\in (\R^d)^2$:

\begin{equation}
\label{ERR_EUL_PS_SD_GLOB}  
p_c(t_j-t_i,y-x)^{-1}|(p-p^h)(t_i,t_j,x,y)|\le 
C
h
^{1/(2d)-C\psi(h)}.
\end{equation}

\item[-] If  $d(y,{\mathcal I})$ (distance of the final point $y$ to the spatial discontinuity set ${\mathcal I} $) satisfies $d(y,{\mathcal I})\ge h^{1/2-\epsilon}$ for a fixed given $\epsilon\in (0,1/2] $, then:
\begin{equation}
\label{ERR_EUL_PS_SD}  
p_c(t_j-t_i,y-x)^{-1}|(p-p^h)(t_i,t_j,x,y)|\le 
C\Big[
h
^{1/(d+1)-C\psi(h)}+\frac{h^{1-C\psi(h)}}{d(y,{\mathcal I})}\Big].
\end{equation}
\item[-] 
In the special case $\sigma(t,x)=\sigma$, i.e. constant diffusion coefficient\footnote{the case of an inhomogeneous diffusion coefficient independent of $x$, i.e. $\sigma(t,x)=\sigma(t) $ could also be handled provided the Gaussian part is simulated exactly in a modified Euler scheme.}, the previous bound improves to:
\begin{equation}
\label{ERR_EUL_PS_SD_PART}
p_c(t_j-t_i,y-x)^{-1}|(p-p^h)(t_i,t_j,x,y)|\le C\Big[
h 
^{1/d-C\psi(h)}+\frac{h^{1-C\psi(h)}}{d(y,{\mathcal I})} \Big].
\end{equation}
\end{trivlist}
\end{THM}
\begin{REM}
This result emphasizes that, as soon as the drift is irregular, a \textit{true} diffusion coefficient deteriorates the convergence rate. This is clear since, in that case, the difference of the densities $p_\varepsilon-p_\varepsilon^h$ in \eqref{DECOMP_ERROR}  involves higher derivatives of densities of processes with mollified coefficients  which are more explosive (see Section \ref{SEC_EUL_ANA}). 

We also mention that the distance of the final point to the discontinuity set plays an important role. The global control \eqref{ERR_EUL_PS_SD_GLOB} improves to \eqref{ERR_EUL_PS_SD} as soon as $h^{1-1/(2d)}\le d(y,{\mathcal I}) $.

Eventually, if the diffusion coefficient does not depend on space, we find, up to the additional term in $\psi(h) $, the usual convergence rate for the weak error if $d=1$ as soon as $c_0 \le d(y,{\mathcal I}) $ for any  given $c_0>0$. 

However, our regularization approach clearly feels the dimension, when doing e.g. H\"older inequalities on neighborhoods of the discontinuity sets, and the convergence rates decrease with the dimension.

Let us carefully mention that considering the weak error 
$d(f,x,t_i,t_j,h):=\E[f(X_{t_j}^{h,t_i,x})]-\E[f(X_{t_j}^{t_i,x})]$ for \textit{smooth} functions $f$ and not Dirac masses as we do, should improve the convergence rates and in particular allow to get rid of the terms in $\psi(h) $ through a careful investigation of the derivatives of the associated heat kernels. We refer to the estimates of Proposition \ref{THE_FINAL_PROP} that could be refined when considering an additional integration w.r.t. to the final variable. 
\end{REM}

%
%
%

\subsection{On Some Related Applications.}

\subsubsection{Some Approximating Dynamics for Interest Rates.}
A very popular model for interest rates in the financial literature is the Cox-Ingersoll-Ross process with dynamics:
\begin{equation}
\label{CIR}
dX_t=(a-k X_t)dt+\sigma |X_t|^{1/2}dW_t,
\end{equation}
for given parameters $\sigma,k,a >0$.
From the numerical viewpoint, the behavior of the Euler scheme is not standard. For a given time-step $h$, the strong error was indeed proved to be, as in the usual Lipschitz case, of order $h^{1/2} $ in  Berkaoui \textit{et al.} \cite{berk:boss:diop:08} provided $a$ is not \textit{too small}. On the other hand, numerical experiments in Alfonsi \cite{alfo:05} emphasized very slow convergence, of order $(-\ln h)^{-1} $, for small values of $a$. This convergence order has been established by Gy\"ongy and R\'asonyi \cite{gyon:raso:11}.
 
Of course the dynamics in \eqref{CIR} does not enter our framework, since it is closer to the dynamics of a Bessel-like process whose density does not satisfy Gaussian bounds. However, we could introduce for positive parameters $\eta,K $, which are respectively meant to be small and large enough, the dynamics:
\begin{equation}
\label{APP_CIR}
 dX_t=(a-k X_t)dt+(\eta +\sigma |X_t|^{1/2}\wedge K )dW_t.
 \end{equation}
The diffusion coefficient $\tilde \sigma(x)=(\eta+\sigma|x|^{1/2}\wedge K ) $ is then uniformly elliptic, 1/2 H\"older continuous and bounded. On the other hand the drift is not bounded but the analysis of Theorem \ref{THM_HOLDER_ES} would still hold true thanks to the work of Konakov and Markova \cite{kona:mark:15} that allows to get rid of the linear drift through a suitable transformation. We would then derive a convergence of order $h^{1/4-C\psi(h)} $ at least for the associated Euler scheme on the densities (see also Remark \ref{TRUE_HOLDER}). 
Even though the marginals in \eqref{APP_CIR} enjoy Gaussian bounds, see e.g. \cite{dela:meno:10}, the expected properties for an interest rate dynamics, mean reverting and positivity, should hold with some high probability. Also, the difference between the approximate dynamics in \eqref{APP_CIR} and the original one in \eqref{CIR} might be investigated through stochastic analysis tools (occupation times). 

\subsubsection{Extension to some Kinetic Models}
The results of Theorems \ref{THM_HOLDER_ES} and \ref{THM_HOLDER_ES_PSD} should extend without additional difficulties to the case of degenerate diffusions of the form:
\begin{equation}
\label{EQ_DEG}
\begin{split}
dX_t^1&=b(t,X_t) dt+\sigma(t,X_t) dW_t,\\
dX_t^2&= X_t^1 dt,
\end{split}
\end{equation}
denoting $X_t=(X_t^1,X_t^2) $, under the same previous assumptions \A{A${}_H$} or \A{A${}_{PS}$} on $b,\sigma$. 
The sensitivity analysis when we consider perturbations of the non-degenerate components, i.e. for a given $\varepsilon>0$:
\begin{equation}
\label{EQ_DEG_PERT}
\begin{split}
dX_t^{1,(\varepsilon)}&=b_{\varepsilon}(t,X_t^{(\varepsilon)}) dt+\sigma_\varepsilon(t,X_t^{(\varepsilon)}) dW_t,\\
dX_t^{2,(\varepsilon)}&= X_t^{1,(\varepsilon)} dt,
\end{split}
\end{equation}
has been performed by Kozhina \cite{kozh:16} following \cite{kona:kozh:meno:15:1}. The key point is that 
 under \textbf{(A)}, the required parametrix expansions of the densities associated with the solutions of equation \eqref{EQ_DEG}, \eqref{EQ_DEG_PERT} were established in \cite{kona:meno:molc:10}. The analysis of the derivatives of the heat kernel, that would require to extend the results of Section \ref{HK_ANA} to the considered degenerate setting will concern further research.\\ 

The paper is organized as follows. We first introduce a suitable mollification procedure of the coefficients in Section \ref{SEC_MOL} and derive from the stability results of Konakov \textit{et al.} \cite{kona:kozh:meno:15:1} how the error of the mollifying procedure is then reflected on the densities. This allows to control the terms $p-p_\varepsilon $ and $p_\varepsilon^h-p^h $ in \eqref{DECOMP_ERROR}. We then give in Section \ref{HK_ANA} some pointwise bounds on the derivatives of the heat-kernels with mollified coefficients. From these controls and the previous error expansion obtained for the Euler scheme with smooth coefficients by Konakov and Mammen \cite{kona:mamm:02}, we are able to control the remaining term $p_\varepsilon-p_\varepsilon^h $ in \eqref{DECOMP_ERROR}. We then establish our main estimates equilibrating the two errors.
Eventually, Section \ref{PROOF_HK} is dedicated to the proof of the controls stated in Section \ref{HK_ANA}. These proofs are based on the parametrix expansions of the underlying densities following the Mc-Kean and Singer approach \cite{mcke:sing:67}.

\mysection{Mollification of the Coefficients and Stability Results}
\label{SEC_MOL}
For the error analysis, in order to apply the strategy described in the introduction,  we first need to \textit{regularize} in an appropriate manner the coefficients. The mollifying procedures differ under our two sets of assumptions.

\subsection{Mollification under \textbf{(A${}_H $) (H\"older continuous coefficients)}}
\label{SEC_HOLDER_MOL}
In this case both coefficients $b,\sigma$ need to be globally regularized in time and space. We introduce the mollified coefficients defined for all $(t,x)\in [0,T]\times \R^{d} $ and $\varepsilon>0$  by 
\begin{equation}
\label{DEF_CONV_SPAT}
b_{\varepsilon,S}(t,x):=b(t,\cdot)*\rho_{\varepsilon}(x), \sigma_{\varepsilon,S}(t,x):=\sigma(t,\cdot)*\rho_{\varepsilon}(x), 
\end{equation}
where $*$ stands for the spatial convolution and for $\varepsilon>0 $, $\rho_{\varepsilon} $ is a spatial mollifier, i.e. for all  $x\in \R^d$, 
$$\rho_{\varepsilon}(x):=\varepsilon^{-d}\rho(x/\varepsilon),\ \rho \in C^\infty(\R^d,\R^+), \int_{\R^d}\rho(y)dy=1, |\supp(\rho)|\subset K,$$ for some compact set $K\subset \R^d$. The subscript $S$ in $b_{\varepsilon,S}, \sigma_{\varepsilon,S}$ appears to emphasize that the \textit{spatial} convolution is considered. We will also need a mollification in time when the coefficients are inhomogeneous.  Up to a symmetrization in time of the coefficients  $b,\sigma $, i.e. we set for all $(t,x)\in [0,T]\times \R^d, b(-t,x)=b(t,x), \ \sigma(-t,x)=\sigma(t,x) $ we can define:
\begin{eqnarray}
\label{DEF_CONV}
b_\varepsilon(t,x)=b_{\varepsilon,S}(.,x)\star \zeta_{\varepsilon^2}(t), \sigma_\varepsilon(t,x)=\sigma_{\varepsilon,S}(.,x)\star \zeta_{\varepsilon^2}(t),
\end{eqnarray}
where $\star $ stands for the time convolution and  for $s\in \R, \ \zeta_{\varepsilon^2} (s):=\varepsilon^{-2}\zeta(s/\varepsilon^{2})$, $\zeta $ being a scalar mollifier with compact support in $[-T,T] $. The \textit{complete} regularization in the spatial and time variable reflects the usual parabolic scaling. This feature will be crucial to balance the singularities appearing in our analysis (see Propositions \ref{THE_FINAL_PROP}, \ref{PROP_C_D} and their proofs below).
We have the following controls.
\begin{PROP}[First Controls on the Mollified Coefficients]
\label{PROP_FCMC}
Assume that \textbf{(A${}_H $)} is in force. Then, there exists $C\ge 1$ s.t. for all $\varepsilon>0$,
\begin{equation}
\label{THE_FIRST_CTR_MC}
\begin{split}
\Delta_{\varepsilon,b}:=\sup_{(t,x)\in [0,T]\times \R^d}|b(t,x)-b_\varepsilon(t,x)|\le C\varepsilon^\gamma,\ \Delta_{\varepsilon,\sigma}:=\sup_{(t,x)\in [0,T]\times \R^d}|\sigma(t,x)-\sigma_\varepsilon(t,x)|\le C\varepsilon^\gamma,\\
\forall \eta \in (0,\gamma),\ \Delta_{\varepsilon,\sigma,\eta}:=\Delta_{\varepsilon,\sigma}+\sup_{t\in [0,T]}|(\sigma-\sigma_\varepsilon)(t,.)|_\eta
\le C(\varepsilon^\gamma+\varepsilon^{\gamma-\eta}),
\end{split}
\end{equation}
where for a given function $f:\R^d\rightarrow \R $, we denote for $\eta\in (0,1),\ |f|_\eta:=\sup_{(x,y)\in (\R^{d})^2,x\neq y}\frac{|f(x)-f(y)|}{|x-y|^\eta} $.

\end{PROP}

\textit{Proof.}
Write first for all $(t,x)\in [0,T]\times \R^d$:
\begin{eqnarray*}
b(t,x)-b_{\varepsilon,S}(t,x)&:=&\int_{\R^d}\{b(t,x)-b(t,y)\} \rho_{\varepsilon}(x-y)dy
=\int_{\R^d}\{b(t,x)-b(t,x-z\varepsilon )\} \rho(z)dz.
\end{eqnarray*}
From the H\"older continuity of $b$ assumed in \A{H} and the above equation, we deduce that $b_{\varepsilon,S}$ satisfies \A{H} as well and that:
\begin{eqnarray}
\label{CTR_HOLDER_B}
\sup_{(t,x)\in [0,T]\times \R^d}|(b-b_{\varepsilon,S})(t,x)| \le C_\rho\varepsilon^{\gamma},\ C_\rho:=\kappa\int_{K}|z|^\gamma \rho(z)dz.
\end{eqnarray}
The same analysis can be performed for $ \sigma_{\varepsilon,S} $, so that $\sigma_{\varepsilon,S} $ satisfies \A{H} and $\sup_{(t,x)\in [0,T]\times \R^d}|(\sigma- \sigma_{\varepsilon,S})(t,x)|\le C_\rho \varepsilon^{\gamma} $.  From \A{H}, we also have that $b_{\varepsilon,S}, \sigma_{\varepsilon,S} $ are both $\gamma/2 $-H\"older continuous in time uniformly in $\varepsilon>0 $. Repeating the previous arguments replacing 
$\rho_{\varepsilon} $ by $\zeta_{\varepsilon^2} $, we deduce
$\sup_{(t,x)\in [0,T]\times \R^d}|(b_{\varepsilon,S}-b_\varepsilon)(t,x)|
+\sup_{(t,x)\in [0,T]\times \R^d} |(\sigma_{\varepsilon,S}-\sigma_\varepsilon)(t,x)|
\le C_\zeta \varepsilon^\gamma$,
which eventually yields:
$$\sup_{(t,x)\in [0,T]\times \R^d}|(b-b_\varepsilon)(t,x)|+\sup_{(t,x)\in [0,T]\times \R^d}|(\sigma-\sigma_\varepsilon)(t,x)|\le C\varepsilon^\gamma.$$
This gives the controls concerning the sup norms in \eqref{THE_FIRST_CTR_MC}.

Let us now turn to the H\"older norm. 
Observe first that, for all $t\in \R^+,\ (x,y)\in (\R^d)^2 $:
\begin{eqnarray*}
\{\sigma(t,x)-\sigma_{\varepsilon,S}(t,x)\}-\{\sigma(t,y)-\sigma_{\varepsilon,S}(t,y)\}\\
=\int_{\R^d}  \Big\{[\sigma(t,x)-\sigma(t,x-z\varepsilon)]-[\sigma(t,y)-\sigma(t,y-z\varepsilon)]\Big\} \rho(z)dz,\\
\{ \sigma_{ \varepsilon}(t,x)-\sigma_{\varepsilon,S}(t,x)\}-\{ \sigma_{\varepsilon}(t,y)-\sigma_{ \varepsilon,S}(t,y)\}\\
=\int_{\R}  \Big\{[\sigma_{ \varepsilon,S}(t-\varepsilon^2u,x)-\sigma_{ \varepsilon,S}(t,x)]-[\sigma_{\varepsilon,S}(t-\varepsilon^2u,y)-\sigma_{\varepsilon,S}(t,y)]\Big\} \zeta(u)du.
\end{eqnarray*}
It readily follows from the $\gamma$-H\"older continuity in space of $\sigma $ (resp. the $\gamma$-H\"older continuity in space and the $\gamma/2 $-H\"older continuity in time of $ \sigma_{ \varepsilon,S}$) that the following controls hold:
\begin{eqnarray}
|[\sigma(t,x)-\sigma_{\varepsilon}(t,x)]-[\sigma(t,y)-\sigma_{ \varepsilon}(t,y)]|&\le& 
C(|x-y|^\gamma \wedge  \varepsilon^{\gamma})\nonumber \le  
C|x-y|^{\eta} \varepsilon^{\gamma-\eta},\nonumber\\
 | (\sigma-\sigma_{ \varepsilon})(t,.)|_{\eta}&\le& C \varepsilon^{\gamma-\eta},\ \eta \in (0,\gamma).
\label{CTR_HOLDER_NORM}
\end{eqnarray}
This completes the proof. \hfill $\square $

We will need as well some controls on the derivatives of the mollified coefficients.
\begin{PROP}[Controls on the Derivatives of the Mollified Coefficients]
\label{PROP_CDER_MC} Under the assumptions of Proposition \ref{PROP_FCMC}, we have that there exists $C\ge 1$ s.t. for all $\varepsilon\in (0,1) $ and for all multi-index $\alpha, \ |\alpha|\in \leftB 1,4\rightB $:
\begin{equation}
\label{CTR_DER_MOLLI}
\begin{split}
\sup_{(t,x)\in [0,T]\times \R^d}|D_x^\alpha b_{\varepsilon}(t,x)|+\sup_{(t,x)\in [0,T]\times \R^d}|D_x^\alpha \sigma_{\varepsilon}(t,x)| \le C{\varepsilon}^{-|\alpha|+\gamma},
\sup_{t\in [0,T]}|D_x^\alpha \sigma_{\varepsilon}(t,.)|_{\gamma}\le C{\varepsilon}^{-|\alpha|}.
\end{split}
\end{equation}
Also, there exists a constant $C$ s.t.:
\begin{equation}
\label{CTR_DER_MOLLI_TIME}
\sup_{(t,x)\in [0,T]\times \R^d}|\partial_t
\sigma_\varepsilon(t,x)|\le C \varepsilon^{-2
+\gamma},\ \sup_{t\in [0,T]}|\partial_t
\sigma_\varepsilon(t,.)|_\eta\le C\varepsilon^{-2+\gamma-\eta},\ \forall \eta\in (0,\gamma].
\end{equation}
\end{PROP}
\textit{Proof.}
For all multi-index $\alpha ,|\alpha|\in \leftB 1, 4\rightB $ and $(t,x)\in [0,T]\times \R^d$ and all $\varepsilon>0 $:
\begin{eqnarray*}
D_x^\alpha \sigma_{\varepsilon,S }(t,x)&=&\int_{\R^d} \sigma(t,z)D_x^\alpha \rho_{\varepsilon}(x-z)dz
                  =\int_{\R^d} [\sigma(t,z)-\sigma(t,x)]D_x^\alpha \rho_{\varepsilon}(x-z)dz.
\end{eqnarray*}
Indeed, setting for all $x\in \R^d,\ g_{\varepsilon}(x):=\int_{\R^d} \rho_{\varepsilon}(x-z)dz=1 $ we have $D_x^\alpha g_{\varepsilon}(x):=\int_{\R^d} D_x^\alpha \rho_{\varepsilon}(x-z)dz=0$. Thus, since $|D_x^\alpha \rho_{\varepsilon}(x-z)|\le \varepsilon^{-(|\alpha|+d)} |D_w^\alpha \rho(w)||_{w=\frac{(x-z)}\varepsilon} $, we derive:
\begin{eqnarray*}
|D_x^\alpha \sigma_{\varepsilon,S}(t,x)|&\le& \int_{\R^d} |\sigma(t,z)-\sigma(t,x)| \varepsilon^{-(|\alpha|+d)}|D_w^\alpha \rho(w)|_{w=\frac{(x-z)}{\varepsilon}} dz\\
&\le &\kappa \varepsilon^{-|\alpha|+\gamma}\int_{\R^d} \left(\frac{|z-x|}{\varepsilon}\right)^{\gamma} \varepsilon^{-d}|D_w^\alpha \rho(w)|_{w=\frac{(x-z)}{\varepsilon}} dz\le c\varepsilon^{-|\alpha|+\gamma},
\end{eqnarray*}
exploiting the H\"older continuity assumption \A{H} for $\sigma$ in the last but one inequality and the assumptions on $\rho $ for the last one. Similarly, we derive for all $(t,x,y)\in [0,T]\times (\R^d)^2 $ and all $\varepsilon>0 $:
\begin{eqnarray*}
|D_x^\alpha \sigma_{\varepsilon,S}(t,x)-D_x^\alpha \sigma_{\varepsilon,S}(t,y)|&\le& \int_{\R^d}|\sigma(t,x-z)-\sigma(t,y-z)| \varepsilon^{-(|\alpha|+d)}|D_w^\alpha \rho(w)|_{w=\frac z\varepsilon} dz\\
&\le &C \kappa {\varepsilon}^{-|\alpha|}|x-y|^\gamma.
\end{eqnarray*}
The same bounds hold for $b_{\varepsilon,S} $ as well. The previous controls readily imply \eqref{CTR_DER_MOLLI} since the additional time convolution does not have any impact here.

Equation \eqref{CTR_DER_MOLLI_TIME} is derived proceeding similarly for the time convolution, exploiting as well the $\gamma/2 $-H\"older continuity in time of $\sigma_{\varepsilon,S} $. This completes the proof. \hfill
$\square $\\
\subsection{Mollification Under \textbf{(A${}_{PS} $)} (Piecewise smooth drift and Smooth Diffusion Coefficient).}
\label{SEC_BORNE_MOL}
In this case we only need to regularize the drift in a neighborhood of the discontinuities. 
Let us denote by $m\in \N^*$, the finite number of spatial discontinuity sets and write
 ${ \mathcal I}:=\cup_{i=1}^m {\mathcal S}_i$, where we recall from \A{A${}_{PS} $} that each ${\mathcal S}_i:=\{x\in \R^d: g_i(x)=0 \}$ is a smooth (at least $C^4 $) bounded submanifold of $\R^d$ of dimension $d_i$ lower or equal to $d-1 $.
For a given parameter $\varepsilon>0 $, define its  neighborhood  ${\mathcal V}_\varepsilon({\mathcal I}) :=\cup_{i=1}^m {\mathcal V}_{\varepsilon }({\mathcal S}_i)$, where for $i\in \leftB 1,m\rightB,\ {\mathcal V}_{\varepsilon}({\mathcal S}_i):=\{z\in \R^d: -\varepsilon\le d_S(z, {\mathcal S}_i)\le \varepsilon\} $. Here, $d_S(\cdot,{\mathcal S}_i) $ stands for the \textbf{signed} distance to ${\mathcal S}_i $. This function has the same smoothness as the boundary ${\mathcal S}_i  $ (see e.g. Lemma 14.16 and its proof p. 355 in \cite{gilb:trud:98}). By convention, for $d_i\ge 1$, we choose $d_S(x,{\mathcal S}_i )$ to be positive for points $x$ being in the bounded region with bounded boundary ${\mathcal S}_i $.

The fact is now that we set $b_\varepsilon(t,x)=b(t,x) $ on $\R^d\backslash {\mathcal V}_\varepsilon({\mathcal I})$ and perform a smooth mollification on the neighborhood ${\mathcal V}_\varepsilon({\mathcal I}) $ of the discontinuity sets. \textcolor{black}{A possible way to proceed is the following.
Introduce for all $i\in \leftB 1,m\rightB $, $\partial {\mathcal V}_\varepsilon^{i,1}:=\{x\in \R^d: d_S(x,{\mathcal S}_i)=-\varepsilon \},\  \partial {\mathcal V}_\varepsilon^{i,2}:=\{x\in \R^d: d_S(x,{\mathcal S}_i)=\varepsilon \}$.   Denoting by $\Big(\Pi_{\partial {\mathcal V}_\varepsilon^{i,j}}(x)\Big)_{j\in \{1,2\}}$ the projection of $x$ on the corresponding boundary $(\partial {\mathcal V}_\varepsilon^{i,j})_{j\in \{1,2\}} $ of  ${\mathcal V}_\varepsilon({\mathcal S}_i) $, which is again well defined on ${\mathcal V}_\varepsilon({\mathcal S}_i)  $,  we set for all $(t,x)\in [0,T]\times {\mathcal V}_\varepsilon({\mathcal S}_i)$:
$$  b_\varepsilon(t,x):=b\big(t,\Pi_{\partial {\mathcal V}_\varepsilon^{i,1} }(x)\big)\exp\big(\frac 14\big)\exp\left(-\frac{1}{4-\frac{d_S(x,\partial {\mathcal V}_\varepsilon^{i,1})^2}{\varepsilon^2}} \right)+b\big(t,\Pi_{\partial {\mathcal V}_\varepsilon^{i,2} }(x)\big)\exp\big(\frac 14\big)\exp\left(-\frac{1}{4-\frac{d_S(x,\partial  {\mathcal V}_{\varepsilon}^{i,2})^2}{\varepsilon^2}} \right),$$
where $d_S(x,\partial {\mathcal V}_\varepsilon^{i,j}),\ {j\in \{1,2\}}$ stands for the \textbf{signed} distance of $x$ to the corresponding boundary $\partial {\mathcal V}_\varepsilon^{i,j} $  and is again a smooth function.
Observing that for $x\in \partial {\mathcal V}_\varepsilon^{i,1}$ (resp. $x\in \partial {\mathcal V}_\varepsilon^{i,2}$) we indeed have $d_S(x,\partial {\mathcal V}_\varepsilon^{i,2})^2=4\varepsilon^2 $ (resp. $d_S(x,\partial {\mathcal V}_\varepsilon^{i,1})^2=4\varepsilon^2 $) we indeed have that for $(t,x)\in [0,T]\times \partial {\mathcal V}_\varepsilon({\mathcal S}_i) =[0,T]\times \Big(\partial {\mathcal V}_\varepsilon^{i,1}\cup \partial {\mathcal V}_\varepsilon^{i,2}\Big),\ b_\varepsilon(t,x)=b(t,x)$ and $b_\varepsilon $ is smooth (as ${\mathcal S}_i $ on ${\mathcal V}_\varepsilon({\mathcal S}_i) $). Thus $b_\varepsilon $ is at least $C^4$ in the space variable.
}

Of course we have that
$|(b-b_\varepsilon)(t,x)| \le C \I_{x\in {\mathcal V}_\varepsilon({\mathcal I})} $ which is not necessarily small. Anyhow, for all $q>1$, since the $(\mathcal S_i)_{i\in \leftB 1,m\rightB} $ are bounded, we derive as well:
\begin{equation}
\label{DIFF_LQ}
\|b-b_\varepsilon\|_{L^q([0,T]\times \R^d)}=\{\int_{0}^Tdt \int_{\R^d}|(b-b_\varepsilon)(t,x)|^q dx \}^{1/q}\le C\{ \int_0^T dt \int_{{\mathcal V_\varepsilon({\mathcal I})}} dx\}^{1/q}\le C_{\mathcal I}\varepsilon^{1/q} T^{1/q}.
\end{equation} 
Recall indeed that since the $({\mathcal S}_i)_{i\in \leftB 1,m\rightB} $ have zero Lebesgue measure and smooth boundary, for the \textit{thickened} neighborhoods $\big({\mathcal V}_\varepsilon({\mathcal S}_i)\big)_{i\in \leftB 1,m\rightB} $, we have for all $i\in \leftB 1,m\rightB, |{\mathcal V}_\varepsilon({\mathcal S}_i)|:=\int_{{\mathcal V}_\varepsilon({\mathcal S}_i)}dx\le C_{}\varepsilon $. This is clear for a bounded portion of hyperplane. The smoothness of the boundary allows to locally map ${\mathcal V}_\varepsilon({\mathcal S}_i)$ with a bounded neighborhood of a hyperplane if ${\mathcal S}_i$ has dimension $d-1$. For submanifolds of smaller dimension $d-i,\ i>1$, the straightening of the boundary can be done in the corresponding dimension $d-i+1 $ and the associated neighborhood would be smaller, namely $|{\mathcal V}_\varepsilon({\mathcal S}_i)|\le C\varepsilon^{d-(d-i)}\le C\varepsilon^i $. We take the \textit{worst} bound for simplicity.

Observe as well that the following control holds for the derivatives of the mollified coefficient. For all multi-index $\alpha, |\alpha|\le 4 $, there exists $C\ge 1
$ s.t. for all $(t,x)\in [0,T]\times \R^d $:
\begin{equation}
\label{THE_DER_BEPS_BETA}
|\partial_x^\alpha b_\varepsilon(t,x)|\le C\{ \varepsilon^{-|\alpha|} \I_{x\in {\mathcal V}_\varepsilon({\mathcal I})}+\I_{x\not \in {\mathcal V}_\varepsilon({\mathcal I})}  \}.
\end{equation}
Under the considered assumptions it is not necessary to mollify the diffusion coefficients. We thus set for all $(t,x)\in [0,T]\times \R^d,\ \sigma_\varepsilon(t,x)=\sigma(t,x) $, in order to keep homogeneous notations under our two running assumptions for the drift.

\subsection{Stability Results}
Recall now that under \textbf{(A${}_H$)} or \textbf{(A${}_{PS}$)} equation \eqref{1} admits a density (see e.g. \cite{sheu:91} under \textbf{(A${}_H $)} or Proposition 1 in \cite{kona:kozh:meno:15:1} under \textbf{(A${}_{PS}$)}), i.e. for all $0\le s<t\le T, \ x\in \R^d, B\in{\mathcal B}(\R^d), \P[X_t\in B|X_s=x]=\int_{B}p(s,t,x,y)dy $. The same holds for the Euler scheme in \eqref{EULER} (see e.g. Theorem 2.1 in \cite{lema:meno:10}), for all $0\le t_i<t_j\le T, x\in \R^d, \P[X_{t_j}^h\in B|X_{t_i}^h=x]=\int_{B}p^h(t_i,t_j,x,y)dy $. These properties remain valid for the respective perturbed diffusion and Euler scheme whose coefficients correspond to the procedures described in Section  \ref{SEC_HOLDER_MOL} and Section \ref{SEC_BORNE_MOL} depending on whether assumption \textbf{(A${}_H$)} or \textbf{(A${}_{PS}$)} is in force. We denote the densities associated with the perturbed diffusion and discretization scheme by  $p_\varepsilon $ and $p_\varepsilon^h $ respectively.

Let us now state the sensitivity result  following from Theorems 1 and 2 in \cite{kona:kozh:meno:15:1}.
\begin{THM}[Main Sensitivity Result]
 \label{MTHM_KKM1}
Define for $q\in (d,+\infty]$ and $\eta\in (0,1] $ the quantities:
 $$\Delta_{\varepsilon,b,q}:=\sup_{t\in [0,T]}\| (b-b_\varepsilon)(t,\cdot)\|_{L^q(\R^d)},\ \Delta_{\varepsilon,\sigma,\eta}:=\sup_{t\in [0,T]}\|\sigma(t,\cdot)-\sigma_\varepsilon(t,\cdot)\|_{L^\infty(\R^d)} +\sup_{t\in [0,T]}|\sigma(t,.)-\sigma_{\varepsilon}(t,.)|_{\eta}.$$
Set $\Delta_{\varepsilon,\eta,q}:=\Delta_{\varepsilon,b,q}+\Delta_{\varepsilon,\sigma,\eta}$. It holds under \A{A} that there exists $C_{ \eta,q}\ge 1 $ s.t.  for all $0\le t_i<t_j\le T$ and $(x,y)\in (\R^d)^2 $:
\begin{equation}
\label{CTR_D}
p_c(t_j-t_i,y-x)^{-1}\Big\{|(p-p_\varepsilon)(t_i,t_j,x,y)|+|(p^h-p_\varepsilon^h)(t_i,t_j,x,y)|\Big\}\le C_{\eta,q} \Delta_{\varepsilon,\eta,q}.
\end{equation}
Also, there exists $
C\ge 1$ s.t.:
\begin{equation}
\label{BD_C_ETA_Q}
\begin{split}
C_{\eta,q} 
\le C\exp(C( (\frac \eta 2\wedge \alpha(q))^{-1}+1)^{(\frac \eta 2\wedge \alpha(q))^{-1}+1 }),\ \alpha(q):=\frac 12(1-\frac dq).
\end{split}
\end{equation}
\end{THM}
\begin{REM}[Constraint on $q$]
The constraint $q>d$ in the above result is due to the fact that to establish \eqref{CTR_D} in the case $q<+\infty $, we are led to control quantities of the type
$$Q:=\int_{0}^{t_j-t_i} dt\int_{\R^d} p_c(t,x-w)|b-b_\varepsilon|(t,w)\frac{1}{((t_j-t_i)-t)^{1/2}} p_c\big((t_j-t_i)-t,y-w\big) dw,$$
through H\"older's inequality.
The constraint $q>d$ naturally appears in order to derive $Q\le \bar C \frac{\Gamma(\alpha(q))}{\Gamma(1+\alpha(q))}T^{\alpha(q)}\|b-b_\varepsilon\|_{L^q(\R^d)} $ for a finite $\bar C:=\bar C(\A{A},T)\ge 1$.
We refer to the proof of Lemma 2 in \cite{kona:kozh:meno:15:1} for additional details.
\end{REM}
\begin{proof}
Equation \eqref{CTR_D} readily follows from Theorems 1 and 2 in \cite{kona:kozh:meno:15:1}.
The point is here to specify the control \eqref{BD_C_ETA_Q} on the constant appearing in \eqref{CTR_D}. Lemma 3 in \cite{kona:kozh:meno:15:1}, quantifies the explosive contributions for each term of the parametrix series giving the difference of the densities. It holds for both the diffusion and the Euler scheme, see Section 3.2 of \cite{kona:kozh:meno:15:1} for details, and yields: 
\begin{equation*}
\begin{split}
C_{\eta,q} \le \sum_{r \ge 1} (r+1)
 \frac{\bar C^{r+1}\left[ \Gamma (\frac{\eta }{2}\wedge \alpha(q))\right] ^{r}}{\Gamma
(1+r(\frac{\eta }{2}\wedge \alpha(q)) )} T^{r(\frac{ \eta}{2}\wedge \alpha(q)) },
\end{split}
\end{equation*}
for a constant $\bar C:=\bar C($(\textbf{A})$,T)$ which does not depend on $\eta $ or $q$. 

Introduce for $\theta\in (0,\frac 12]$ the quantity:
\begin{eqnarray*}
I(T,\theta):=\sum_{r\ge 1}(r+1) \frac{\bar C^{r+1}\left[ \Gamma ( \theta)\right] ^{r}}{\Gamma
(1+r\theta )} T^{r \theta}.
\end{eqnarray*}
%
One easily gets that for a given $T>0$, there exists $\tilde C:=\tilde C(\A{A},T)\ge 1$  independent of $\theta  $ as well such that:
\begin{equation*}
I(T,\theta) \le C\sum_{r \ge 1} (r+1)
 \frac{\tilde C^{r+1}\left[ \Gamma (\theta)\right] ^{r}}{\Gamma
(1+r \theta  )}.
\end{equation*}
Set now $r_0:=\lceil \frac 1{\theta}\rceil $ and write by monotonicity of the $\Gamma $ function (see e.g. formula 8.363 (8) in Gradstein and Ryzhik \cite{grad:ryzh:14}):
\begin{eqnarray*}
I(T,\theta)&\le& C \sum_{k\ge 0} (k+1)r_0\sum_{kr_0\le  r<(k+1)r_0}\frac{\{\tilde C\Gamma(\theta)\}^r}{\Gamma(1+k)}\notag\\
&\le& C \sum_{k\ge 0} \frac{(k+1) r_0}{\Gamma(k+1)}\sum_{kr_0\le  r<(k+1)r_0}\{\tilde C(\theta^{-1}+\exp(-1))\}^r\notag\\
&\le& C r_0^2 \sum_{k\ge 0} \frac{(k+1) }{\Gamma(k+1)}[\{\tilde C(\theta^{-1}+\exp(-1))\}^{r_0}]^{k+1}
\le 
C\exp(\tilde C( \theta^{-1}+1)^{\theta^{-1}+1 }).
\end{eqnarray*}
This gives \eqref{BD_C_ETA_Q} taking $\theta=\frac \eta 2\wedge \alpha(q) $ and completes the proof.
\end{proof}

From Theorem \ref{MTHM_KKM1}, we get the following key sensitivity results.
\begin{LEMME}[Sensitivity under \textbf{(A${}_{H}$)}]\label{LEMMA_SENS_HOLD}
Under Assumption \textbf{(A${}_{H}$)}, for $\eta\in (0,\gamma)$ there exists $c\le 1$ s.t. for all $0\le t_i<t_j\le T, (x,y)\in (\R^d)^2 $:
\begin{equation}
\label{CTR_SENSI_HOLDER}
|(p-p_\varepsilon)(t_i,t_j,x,y)|+|(p^h-p_\varepsilon^h)(t_i,t_j,x,y)|\le C_\eta \varepsilon^{\gamma-\eta} p_c(t_j-t_i,y-x),
\end{equation}
where $C_\eta:=C_{\eta,\infty}\le  C\exp(C( (\frac \eta 2)^{-1}+1)^{(\frac \eta 2)^{-1}+1 })$ for $C_{\eta,\infty} $ as in \eqref{BD_C_ETA_Q}.
\end{LEMME}
\begin{proof}
The lemma derives from Theorem \ref{MTHM_KKM1} and Proposition \ref{PROP_FCMC}. The bound on $C_\eta $ follows observing as well that for $ \eta\in (0,\gamma),\ \frac \eta 2<\frac \gamma 2\le \frac 12=\alpha(\infty)$ so that $ \frac \eta 2\wedge \alpha(\infty)=\frac \eta 2$.
\end{proof}

\begin{LEMME}[Sensitivity under \textbf{(A${}_{PS}$)}] \label{LEMMA_SENS_BOUND}
Under Assumption \textbf{(A${}_{PS}$)}, for $q>d$ there exists $c\le 1$ s.t. for all $0\le t_i<t_j\le T, (x,y)\in (\R^d)^2 $:
\begin{equation}
\label{CTR_SENSI_BOUNDED}
|(p-p_\varepsilon)(t_i,t_j,x,y)|+|(p^h-p_\varepsilon^h)(t_i,t_j,x,y)|\le C_q \varepsilon^{1/q} p_c(t_j-t_i,y-x),
\end{equation}
where $C_q:=C_{1,q}\le  C\exp(C( \alpha(q)^{-1}+1)^{\alpha(q)^{-1}+1 })$ for $C_{ 1,q} $ as in \eqref{BD_C_ETA_Q}.
\end{LEMME}
\begin{proof}
Recall that under \A{A${}_{PS} $}, since the diffusion coefficient is smooth, there is no need to regularize it and $\sigma=\sigma_\varepsilon $. Thus,  $\Delta_{\varepsilon,\sigma,1}=0 $. From this observation and equation \eqref{DIFF_LQ}, Theorem \ref{MTHM_KKM1} then yields \eqref{CTR_SENSI_BOUNDED}. The bound on $C_q $ follows observing as well that for $q\in (d,+\infty), \frac 1 2\wedge \alpha(q)=\alpha(q)$.
\end{proof}

Let us mention that the constants $C_\eta, C_q$ in equations \eqref{CTR_SENSI_HOLDER} and \eqref{CTR_SENSI_BOUNDED} respectively explode when $\eta $ goes to 0 and $q $ goes to $d$, which is precisely what we want in order to have the fastest convergence rate w.r.t. $\varepsilon $. On the other hand, the explosion rates that we have emphasized in \eqref{BD_C_ETA_Q} are crucial in order to equilibrate the global errors. This step is performed in Section \ref{SEC_EUL_ANA} below.

\mysection{Error Analysis and Derivation of the Main Results}
\label{HK_ANA}

\subsection{Stream Line to the Proofs of the Main Results.}
This Section is devoted to the proof of Theorems \ref{THM_HOLDER_ES} and \ref{THM_HOLDER_ES_PSD}.

Our main results are 
those controlling the difference of the densities, i.e.  the estimates given in equations \eqref{ERR_EUL} under \textbf{(A${}_H $)} and \eqref{ERR_EUL_PS_SD}, \eqref{ERR_EUL_PS_SD_PART} under \textbf{(A${}_{PS}$)}.

To obtain these bounds, the strategy is the following. Let $0\le t_i<t_j\le T $ and $(x,y)\in (\R^d)^2 $ be given.
One writes for $\varepsilon>0$:
\begin{equation}
\label{DECOUP_ERR} 
|p(t_i,t_j,x,y)-p^h(t_i,t_j,x,y)|\le 
|p-p_{ \varepsilon}|(t_i,t_j,x,y)+|p_{\varepsilon}-p_{\varepsilon}^h|(t_i,t_j,x,y)
+ |p_{\varepsilon}^h-p^h|(t_i,t_j,x,y).
\end{equation}
Now, one derives from the sensitivity Lemma \ref{LEMMA_SENS_HOLD} that, under \textbf{(A${}_{H}$)}, for all $\eta\in (0,\gamma) $:
\begin{eqnarray}
\label{BD_PREAL_EUL_HOLD}
|p(t_i,t_j,x,y)-p^h(t_i,t_j,x,y)| \le  C_\eta \varepsilon^{\gamma-\eta}p_c(t_j-t_i,y-x)+|(p_{\varepsilon}-p_{\varepsilon}^h)|(t_i,t_j,x,y).
\end{eqnarray}
Similarly, Lemma \ref{LEMMA_SENS_BOUND} yields that, under \textbf{(A${}_{PS}$)}, for all $q>d$:
\begin{eqnarray}
\label{BD_PREAL_EUL_BOUND}
|p(t_i,t_j,x,y)-p^h(t_i,t_j,x,y)| \le  C_q 
\varepsilon^{1/q}p_c(t_j-t_i,y-x)+|(p_{\varepsilon}-p_{\varepsilon}^h)|(t_i,t_j,x,y).
\end{eqnarray}

To investigate and minimize the contributions in the error it thus remains from equations \eqref{BD_PREAL_EUL_HOLD} and \eqref{BD_PREAL_EUL_BOUND} to precisely control the difference $|p_\varepsilon-p_\varepsilon^h| $ in \eqref{DECOUP_ERR}. Let us now recall that, since the densities $p_{\varepsilon},p_{\varepsilon}^h $ are now respectively associated with a diffusion process and its Euler scheme with \textit{smooth} coefficients, they can be compared thanks to the results in \cite{kona:mamm:02} adapted to the current inhomogeneous setting.  We thus have that:
\begin{equation}
\label{ERR_F_1}
|(p_\varepsilon-p_\varepsilon^h)(t_i,t_j,x,y)|\le C_{b_\varepsilon,\sigma_\varepsilon} h p_c(t_j-t_i,y-x),
\end{equation}
where $C_{b_\varepsilon,\sigma_\varepsilon} $ depends on the derivatives of $b_\varepsilon,\sigma_\varepsilon $ and therefore explodes when $\varepsilon$ goes to 0.

The delicate and crucial point is that we must here precisely quantify this explosion. A key ingredient, to proceed is the \textit{parametrix} series representation for the densities of the diffusion and its Euler scheme. These aspects are recalled in Section \ref{PARAM} below.

Importantly, the parametrix expansion of the density of $X_{t_j}^{t_i,x}$ in \eqref{1}, i.e. for the equation without mollified coefficients, also directly allows to derive, without any sensitivity analysis, exploiting the controls on the derivatives of the density $p(t_i,t_j,x,\cdot) $ of $X_{t_j}^{t_i,x} $ w.r.t. $x$ up to order 2 under \A{A${}_H $}, the bounds in \eqref{WEAK_HOLDER} and \eqref{HOLDER_BOREL}. The arguments follow from \textit{cancellation} techniques that are also crucial to derive our main estimates. We first illustrate this approach in Section \ref{ERR_FAIBLE_INT} which is dedicated to the proof of \eqref{WEAK_HOLDER} and \eqref{HOLDER_BOREL} (integrated weak error).

The main results corresponding to the controls of the difference of the densities  are established in Section \ref{SEC_EUL_ANA}. 
As emphasized above, these results do rely on the sensitivity analysis. They also require a careful analysis of the explosions of the higher order derivatives of the involved heat kernels which need to be quantitatively controlled in terms of the corresponding regularization procedure. The main result in that direction is Proposition \ref{THE_FINAL_PROP} below whose proof, which heavily exploits cancellation techniques, is postponed to Section \ref{SEC_PROOF_ANA}. It yields a precise control of the constant $C_{b_\varepsilon,\sigma_\varepsilon} $ in \eqref{ERR_F_1}. The main results of Theorems \ref{THM_HOLDER_ES} and \ref{THM_HOLDER_ES_PSD} are then derived in Section \ref{SEC_PROOF_THM_1} and \ref{SEC_PROOF_AS} respectively,
balancing the errors appearing in \eqref{ERR_F_1}, and \eqref{BD_PREAL_EUL_HOLD} under \A{A${}_H $} or \eqref{BD_PREAL_EUL_BOUND} under \A{A${}_{PS} $}.

\subsection{Parametrix Representation of Densities.}
\label{PARAM}

From Section 2 in \cite{kona:kozh:meno:15:1}, we derive that under \textbf{(A)} (i.e. the expansions below hold under both \textbf{(A${}_H$)} and \textbf{(A${}_{PS}$)}), for all $\varepsilon\ge 0 $ (the expansion below even holds for the initial coefficients taking $\varepsilon=0 $), $0\le s<t\le T,\ (x,y)\in (\R^d)^2 $:
\begin{equation}
\label{SERIE_P_EPS}
 p_{\varepsilon}(s,t,x,y):=\sum_{r\in \N} \tilde p_{\varepsilon} \otimes H_{\varepsilon}^{(r)}(s,t,x,y),
\end{equation} 
where for $0\le u<t\le T, (z,y)\in (\R^d)^2$:
\begin{equation}
\label{PARAM_KER}
H_\varepsilon(u,t,z,y):=(L_u^\varepsilon-\tilde L_u^{\varepsilon,y})\tilde p_\varepsilon(u,t,z,y),
\end{equation}
and $L_u^\varepsilon,\tilde L_u^{\varepsilon,y} $ respectively stand for the generators at time $u$ of the processes 
\begin{eqnarray}
X_t^{(\varepsilon)}=z+\int_u^t b_\varepsilon(v,X_v^{(\varepsilon)})dv+\int_u^t \sigma_\varepsilon(v,X_v^{(\varepsilon)})dW_v,
\tilde X_t^{(\varepsilon),y}= z 
+\int_u^t \sigma_\varepsilon(v,y)dW_v,
\label{DEF_PROC_G}
\end{eqnarray}
i.e. for all $\varphi\in C^2(\R^d,\R),\ x\in \R^d $,
\begin{eqnarray*}
L_u^\varepsilon \varphi(x)=\langle b_\varepsilon(u,x),\nabla_x \varphi(x)\rangle+\frac 12 {\rm Tr}\Big(\sigma_\varepsilon\sigma_\varepsilon^*(u,x) D_x^2 \varphi(x) \Big),\
\tilde L_u^{\varepsilon,y} \varphi(x)=
\frac 12 {\rm Tr}\Big(\sigma_\varepsilon\sigma_\varepsilon^*(u,y) D_x^2 \varphi(x) \Big).
\end{eqnarray*}
Also $\tilde p_\varepsilon(u,t,z,y):=\tilde p_\varepsilon^y(u,t,z,w)|_{w=y} $ where $\tilde p_\varepsilon^y(u,t,z,.) $ stands for the density at  time $t$ of the process $\tilde X^{(\varepsilon),y} $ starting from $z$ at time $u$. We denote in \eqref{SERIE_P_EPS}, $\tilde p_\varepsilon\otimes H_\varepsilon^{(0)}(s,t,x,y)=\tilde p_\varepsilon(s,t,x,y)$ and for all $r\ge 1,\ \tilde p_\varepsilon\otimes H_\varepsilon^{(r)}(s,t,x,y)= \int_{s}^t du\int_{\R^d}\tilde p_\varepsilon(s,u,x,z)H_\varepsilon^{(r)}(u,t,z,y)dz$ where for $r\ge 2$, $ H_\varepsilon^{(r)}(u,t,z,y):=H_\varepsilon\otimes H_\varepsilon^{(r-1)}(u,t,z,y):=\int_{u}^{t} dv\int_{\R^d}H_\varepsilon(u,v,z,w)H_\varepsilon^{(r-1)}(v,t,w,y) dw$. More generally, the symbol $\otimes$ stands for the time-space convolution, i.e. for two  real valued functions $f,g$ defined on $[0,T]^2\times (\R^d)^2 $, $0\le s<t\le T, f\otimes g(s,t,x,y):=\int_s^t du\int_{\R^d} f(s,u,x,z)g(u,t,z,y)dz $. We also recall that under \A{A${}_{PS}$}, since the diffusion coefficient is smooth we do not regularize it and denote in this case $\sigma_\varepsilon=\sigma $.

To investigate the contribution $p_\varepsilon-p_\varepsilon^h $ in \eqref{DECOUP_ERR}
we will also use for $0\le t_i<t_j\le T,(x,y)\in (\R^d)^2 $ the function:
\begin{equation}
\label{SERIE_P_EPS_D}
p_{\varepsilon}^d(t_i,t_j,x,y):=\sum_{r\in \N} \tilde p_{\varepsilon} \otimes_h H_{\varepsilon}^{(r)}(t_i,t_j,x,y),
\end{equation}
where the quantities at hand are the same as above and the symbol $\otimes_h $ replacing the $\otimes $ in \eqref{SERIE_P_EPS} denotes the \textit{discrete} convolution, i.e. for all $r\ge 1$,
$$ \tilde p_\varepsilon\otimes_h H_\varepsilon^{(r)}(t_i,t_j,x,y)= h \sum_{k=0}^{j-i-1}\int_{\R^d}\tilde p_\varepsilon(t_i,t_{i+k},x,z)H_\varepsilon^{(r)}(t_{i+k},t_j,z,y)dz.$$
Even though $p_{\varepsilon}^d(t_i,t_j,x,.) $ is not \textit{a priori} a density, we will call it so with a slight abuse of terminology. An important control, under \textbf{(A)}, for the terms in the parametrix series  is the following:
\begin{eqnarray}
\forall 0\le s<t\le T,\  |\tilde{p}_\varepsilon \otimes H_\varepsilon^{(r)}(s,t,x,y)|  
\leq  \frac{((1\vee T^{(1-\gamma)/2})c_1)^{r+1}\left[ \Gamma (\frac{\gamma }{2})\right] ^{r}}{\Gamma
(1+r\frac{\gamma }{2})}p_{c}(t-s,y-x) (t-s)^{\frac{r \gamma}{2}},  \nonumber\\
\forall 0\le t_i<t_j\le T,\ |\tilde{p}_\varepsilon \otimes_h H_\varepsilon^{(r)}(t_i,t_j,x,y)|\le \frac{((1\vee T^{(1-\gamma)/2})c_1)^{r+1}\left[ \Gamma (\frac{\gamma }{2})\right] ^{r}}{\Gamma
(1+r\frac{\gamma }{2})}p_{c}(t_j-t_i,y-x) (t_j-t_i)^{\frac{r \gamma}{2}},
\label{10}
\end{eqnarray}  
taking $\gamma=1 $ under \A{A${}_{PS} $}.
We emphasize that those bounds are uniform w.r.t. $\varepsilon\ge 0 $ and refer to \cite{kona:mamm:02} or Section 2 in \cite{kona:kozh:meno:15:1} for a proof.

From the same references (see also Lemma 3.6 in \cite{kona:mamm:00}), we have that the density of the Euler scheme also admits a similar \textit{parametrix} representation.
Introduce for $0\le t_i <t_k\le T, (z,y)\in (\R^d)^2$, the schemes:
\begin{eqnarray}
\label{SCHEMES_AND_PARAM}
X_{t_{k}}^{h,(\varepsilon)}&=&z+\sum_{l=i}^{k-1}\big( b_\varepsilon(t_l,X_{t_l}^{h,(\varepsilon)}) h+\sigma_\varepsilon(t_l,X_{t_l}^{h,(\varepsilon)}) (W_{t_{l+1}}-W_{t_l}) \big),\nonumber\\
\tilde X_{t_{k}}^{h,(\varepsilon),y}&
=&
z+\sum_{l=i}^{k-1} 
\sigma_\varepsilon(t_l,y) (W_{t_{l+1}}-W_{t_l}). 
\end{eqnarray}
Viewed as Markov Chains, their generators write
for all $\varphi \in C^2(\R^d,\R),\ x\in \R^d $:
\begin{eqnarray*}
L_{t_i}^{h,\varepsilon}\varphi(x):=h^{-1}\E[\varphi(X_{t_{i+1}}^{h,(\varepsilon),t_i,x})-\varphi(x)],
\tilde L_{t_i}^{h,\varepsilon,y}\varphi(x)= h^{-1}\E[\varphi(\tilde X_{t_{i+1}}^{h,(\varepsilon),y, t_i,x})-\varphi(x)].
\end{eqnarray*}
Define now for $0\le t_i < t_j\le T, (z,y) \in (\R^d)^2$ the \textit{Markov chain} analogue of the parametrix kernel $H$ in \eqref{PARAM_KER} by:
$$H_{\varepsilon}^h (t_i,t_j,z,y):=(L_{t_i}^{h,\varepsilon}-\tilde L_{t_i}^{h,\varepsilon,y})\tilde p_\varepsilon^h(t_i+h,t_j,x,y).$$
One gets the following parametrix representation for the density of the Euler scheme:
\begin{equation}
\label{SERIE_P_EPS_EULER}  
 p_{\varepsilon}^h(t_i,t_j,x,y):=\sum_{r=0}^{j-i} \tilde p_{\varepsilon} \otimes_h H_{\varepsilon}^{h,(r)}(t_i,t_j,x,y).
\end{equation}
 Again, the subscript $\varepsilon$ is meant to explicitly express the dependence on the mollified coefficients. Also, the terms in the above series satisfy the controls of equation \eqref{10} uniformly in $\varepsilon\ge 0 $.

\subsection{Integrated Weak Error under \textbf{(A${}_H$)}.}
\label{ERR_FAIBLE_INT}

We first prove the  statements concerning the integrated weak error in \eqref{WEAK_HOLDER} and \eqref{HOLDER_BOREL}. We insist that, in that case, no regularization of the coefficients is needed. We have the following result:
\begin{PROP}[Controls of the Derivatives.]\label{PROP_HK_BASE}
Let $T>0$ be fixed.
Under \textbf{(A${}_H $)}, there exist constants $C\ge 1,\ c\in(0,1]$ s.t. for all $0\le s<t\le T, (x,y)\in(\R^d)^2 $ and all multi-index $\alpha,\ |\alpha|\le 2 $:
\begin{equation}
\label{CTR_HK_BASE}
|D_x^\alpha p(s,t,x,y)|\le \frac{C}{(t-s)^{|\alpha|/2}}p_c(t-s,x-y).
\end{equation}
As a consequence we also derive that for $t_j=jh \in [0,T] $ being fixed and setting for all $(t,x)\in [0,t_j]\times \R^d $, $v(t,x):=\E[f(X_{t_j}^{t,x})] $, as soon as $f$ is bounded, we have that for all $(t,x)\in [0,t_j)\times \R^d $:
\begin{equation}
\label{CTR_GRAD}
|\nabla _x v(t,x) |\le \frac{C}{(t_j-t)^{1/2}}
\end{equation}
and for 
$ f\in C^\beta(\R^d,\R), \ \beta \in (0,1]$ (space of globally, and possibly unbounded, H\"older continuous functions), we have for a multi-index $\alpha,\ |\alpha|\le2 $ and all $(t,x)\in [0,t_j)\times \R^d $: 
\begin{equation}
\label{CTR_GRAD_EXPLO}
|D_x^\alpha v(t,x)|\le \frac{C}{(t_j-t)^{(|\alpha|-\beta)/2}}.
\end{equation} 
\end{PROP}
\begin{proof}
Equation \eqref{CTR_HK_BASE} is a direct consequence of 
Proposition \ref{THE_FINAL_PROP} below. This estimate readily gives \eqref{CTR_GRAD}.
On the other hand, we get that for 
$ f\in C^\beta(\R^d,\R), \ \beta \in (0,1]$, we have for a multi-index $\alpha,\ |\alpha|\le 2, (t,x)\in [0,t_j)\times \R^d $:
\begin{equation*}
\label{FIRST_CANC}
D_x^\alpha v(t,x)= \int_{\R^d}D_x^\alpha p(t,t_j,x,y) f(y) dy= \int_{\R^d}D_x^\alpha p(t,t_j,x,y) (f(y)-f(x)) dy,
\end{equation*} 
recalling that $D_x^\alpha \int_{\R^d }p(t,t_j,x,y) dy =0 $ for the last identity. This is precisely what we call a \textit{cancellation} technique. It allows here to exploit the spatial H\"older continuity of $f$ to get rid of the time singularity appearing in \eqref{CTR_HK_BASE} when $|\alpha|=2 $, or to decrease the time singularity appearing in \eqref{CTR_GRAD}. Hence, from \eqref{CTR_HK_BASE}:
$$|D_x^\alpha v(t,x)|\le  \frac{C|f|_\beta}{(t_j-t)^{(|\alpha|-\beta)/2}}\int_{\R^d}p_c(t_j-t,y-x) \Big(\frac{|x-y|}{(t_j-t)^{1/2}} \Big)^{\beta}dy.$$
 Equation \eqref{CTR_GRAD_EXPLO} readily follows.
Similar operations will be recurrent in the proof of Proposition \ref{THE_FINAL_PROP}.
\end{proof}

\subsubsection{Proof of \eqref{WEAK_HOLDER}: H\"older final test function.}
Set $t_{\beta,\gamma}^h:=\sup\{(t_k)_{k\in \leftB 0,j \rightB}: t_k\le t_j-h^{\gamma/\beta}<t_{k+1} \} $ and $I_{\beta,\gamma}^h:=t_{\beta,\gamma}^h/h$. In particular, if $ \gamma\ge \beta$, $t_{\beta,\gamma}^h=t_{j-1} $ and if $\beta>\gamma $, $t_{\beta,\gamma}^h<t_{j-1} $.

Let $v$ be the function defined in Proposition \ref{PROP_HK_BASE}. It follows from Proposition \ref{THE_FINAL_PROP} that $v\in C^{\beta/2,\beta}([0,t_j]\times \R^d,\R)\cap C^{1,2}([0,t_j)\times \R^d,\R) $.
An expansion similar to \eqref{DECOUP_ERR_EDP} yields:
\begin{eqnarray}
\label{THE_EQ_HOLDER_LAST_STEP}
|\E[f(X_{t_j}^{h,t_i,x})-f(X_{t_j}^{t_i,x})]|\le  |\E[f(X_{t_j}^{h,t_i,x})-v(t_{\beta,\gamma}^h,X_{t_{\beta,\gamma}^h}^{h,t_i,x})]|\nonumber\\
+C \sum_{k=i}^{I_{\beta,\gamma}^h-1}\int_{t_k}^{t_{k+1}} ds \E[\{|\nabla_x v(s,X_s^{h,t_i,x})|+|D_x^2 v(s,X_s^{h,t_i,x})|\} \{ |s-t_k|^{\gamma/2}+|X_s^{h,t_i,x}-X_{t_k}^{h,t_i,x}|^\gamma\}]\nonumber\\
=:(T_L+T_M)(h,t_i,t_j,x),
\end{eqnarray}
where $T_L $ stands for the contribution associated with the last step(s) and $T_M $ for the other \textit{main} steps.

From equation \eqref{CTR_GRAD_EXPLO} in Proposition \ref{PROP_HK_BASE}, one readily gets:
\begin{equation}
\label{CTR_TM}
T_M(h,t_i,t_j,x)\le Ch^{\gamma/2}\sum_{k=i}^{I_{\beta,\gamma}^h-1}\int_{t_k}^{t_{k+1}} (1+\frac{1}{(t_j-s)^{1-\beta/2}})ds\le Ch^{\gamma/2}.
\end{equation}
The contribution $T_L $ requires a more careful treatment. Let us write:
\begin{eqnarray}
T_L(h,t_i,t_j,x) &\le&   \E[|f(X_{t_j}^{h,t_i,x})-f(X_{t_j-h^{\gamma/\beta}}^{h,t_i,x}) |]+\E[|v(t_{j}, X_{t_j-h^{\gamma/\beta}}^{h,t_i,x})-v(t_{j}-h^{\gamma/\beta}, X_{t_j-h^{\gamma/\beta}}^{h,t_i,x})|]\nonumber\\
&&+\big|\E[v(t_{j}-h^{\gamma/\beta}, X_{t_j-h^{\gamma/\beta}}^{h,t_i,x})-v(t_{\beta,\gamma}^h,X_{t_{\beta,\gamma}^h}^{h,t_i,x})]\big|\nonumber\\
&\le&   \E[|f(X_{t_j}^{h,t_i,x})-f(X_{t_j-h^{\gamma/\beta}}^{h,t_i,x}) |]+\E[|v(t_{j}, X_{t_j-h^{\gamma/\beta}}^{h,t_i,x})-v(t_{j}-h^{\gamma/\beta}, X_{t_j-h^{\gamma/\beta}}^{h,t_i,x})|]\nonumber\\
&&+\int_{t_{\beta,\gamma}^h}^{t_j-h^{\gamma/\beta}}\Big\{\E[ |\nabla_x v(s,X_s^{h,t_i,x})||b(s,X_s^{h,t_i,x})-b(\phi(s),X_{\phi(s)}^{h,t_i,x})|\nonumber\\
&& +\frac 12|D_x^2 v(s,X_s^{h,t_i,x})||a(s,X_s^{h,t_i,x})-a(\phi(s),X_{\phi(s)}^{h,t_i,x})|]\Big\}ds\nonumber\\
&\le&  C \E[|X_{t_j}^{h,t_i,x}-X_{t_j-h^{\gamma/\beta}}^{h,t_i,x} |^\beta]+Ch^{\gamma/2}\int_{t_{\beta,\gamma}^h}^{t_{j}-h^{\gamma/\beta}} (1+\frac{1}{(t_j-s)^{1-\beta/2}})ds
\le Ch^{\gamma/2},\nonumber\\
\label{CTR_TL}
\end{eqnarray}
expanding as in \eqref{DECOUP_ERR_EDP} the term $|\E[v(t_{j}-h^{\gamma/\beta}, X_{t_j-h^{\gamma/\beta}}^{h,t_i,x})-v(t_{\beta,\gamma}^h,X_{t_{\beta,\gamma}^h}^{h,t_i,x})]| $ with It\^o's formula and using \eqref{FLAVOUR_STRONG} for the last two inequalities. Plugging  \eqref{CTR_TL} and \eqref{CTR_TM} into \eqref{THE_EQ_HOLDER_LAST_STEP} yields the required control in  \eqref{WEAK_HOLDER}. 
 
 \begin{REM}[Extensions to functions $f$ with subquadratic exponential growth]\label{RQ_UNBOUNDED}
 We stated \eqref{WEAK_HOLDER} for $f\in C^\beta (\R^d,\R)$ for simplicity. Observe anyhow that the above arguments can be adapted to derive the expected convergence rate as soon as $f$ is locally $\beta $-H\"older and satisfies the growth condition:
 \begin{equation}
 \label{GR_LOC}
\begin{split}
 \exists C_0>0, \forall x\in \R^d,\ |f(x)|\le C_0\exp(c_0|x|^2), \  c_0\le \frac{c}{4T},\\
 \forall (x,y)\in (\R^d)^2,\  |x-y|\le 1, |f(x)-f(y)|\le C_0|x-y|^\beta \exp(c_0|x|^2),
 \end{split}
 \end{equation} 
 where $c$ is as in equation \eqref{CTR_HK_BASE}. In that case, the controls of equations \eqref{CTR_GRAD} and \eqref{CTR_GRAD_EXPLO} would write in the following way. There exists a constant $C\ge 1$ s.t. for all $(t,x)\in [0,t_j)\times \R^d $:
 \begin{eqnarray}
\label{BD_GR_EXPLO}
|\nabla v(t,x)|&\le& \frac{C}{(t_j-t)^{1/2}}\int_{\R^d}p_c(t_j-t,y-x)\exp(c_0|y|^2) dy, \nonumber\\
&\le &\frac{C\exp(2c_0|x|^2)}{(t_j-t)^{1/2}}\int_{\R^d}p_c(t_j-t,y-x)\exp(2c_0|y-x|^2) dy\le \frac{C\exp(2 c_0 |x|^2)}{(t_j-t)^{1/2}}, \nonumber\\
\forall \alpha,\ |\alpha|=2,\ |D_x^\alpha v(t,x)|&\le& \frac{C\exp(2 c_0 |x|^2)}{(t_j-t)^{1- \beta/2}} .
 \end{eqnarray}
 Plugging \eqref{BD_GR_EXPLO} into \eqref{THE_EQ_HOLDER_LAST_STEP} and \eqref{CTR_TL} still yields, thanks to the condition on $ c_0$ in \eqref{GR_LOC} and \eqref{Growth}, an integrable contribution.
\end{REM}
 
 \subsubsection{Proof of \eqref{HOLDER_BOREL}: Indicator of a Domain as Test Function.}
 \label{SUBSEC_HODLER_BOREL}
 We have assumed $ A$ to be  $ C^2$ domain and $\partial A $ bounded. Let us denote by $d_S(\cdot,\partial A)$ the \textbf{signed} distance to the boundary, i.e. $d(x,\partial A) > 0$ for $x\in  A$ and $d(x,\partial A) \le 0$ for $x\not \in  A$.
 
  It is known (see e.g. Lemma 14.16 and its proof p. 355 in \cite{gilb:trud:98}) that for $\delta>0$ small enough, on $V_\delta(A):=\{y\in \R^d: |d_S(y,\partial A)|\le \delta\} $, the \ function $d_S(\cdot,\partial A)$ is $C^2$ and both the exterior and interior sphere conditions hold.
 The interior sphere condition writes that for $y \in A_\delta:=V_\delta(A)\cap A:=\{ y\in \R^d: 0< d_S(y,\partial A)\le \delta\}$ (interior points of $A$ whose distance to the boundary is lower or equal than $\delta$), its orthogonal projection on the boundary $\Pi_{\partial A}(y) $ is also the unique point s.t. defining $B(y,d_S(y,\partial A)):=\{z\in \R^d: \|z-y\|\le d_S(y,\partial A) \} $, $B(y,d_S(y,\partial A))\cap \partial A=\Pi_{\partial A}(y) $. 
 The exterior sphere condition writes similarly for the points $y\in V_\delta(A)\backslash \bar A:=\{y\in \R^d: -\delta\le  d_S(y,\partial A)< 0\} $ (strictly exterior points of $A$ whose distance to the boundary is lower or equal than $\delta$).
 
 For such a $\delta$, let us now write for $0\le t_i<t_j\le T,\ x\in \R^d $:
 \begin{equation}
\label{DECOUP_IND_HOLDER}
 \begin{split}
\E[\I_{X_{t_j}^{h,t_i,x}\in A}]-\E[\I_{X_{t_j}^{t_i,x}\in A}]\\
=\{\E[\I_{X_{t_j}^{h,t_i,x}\in A}]-\E[f_\delta(X_{t_j}^{h,t_i,x})]\}+\{\E[f_\delta(X_{t_j}^{h,t_i,x})]-\E[f_\delta(X_{t_j}^{t_i,x})]\}+\{\E[f_\delta(X_{t_j}^{t_i,x})]-\E[\I_{X_{t_j}^{t_i,x}\in A}]\}=:\sum_{i=1}^3 T_i^\delta,
\end{split}
 \end{equation}
 where 
 \begin{equation}
\label{DEF_F_DELTA}
 f_\delta(x)=\begin{cases} 1,\ {\rm if}\ x\in A, \\
  \exp(1)\exp(-1/(1 -\frac{d_S(x,\partial A)^2}{\delta^2})),\ {\rm if}\ x\in V_\delta(A)\backslash A,\\
  0 \ {\rm if}\ x\not \in A\cup V_\delta(A).
  \end{cases}
\end{equation}
  Namely, $f_\delta$ stands for a smooth approximation (at least $C^2$) of the mapping $x\mapsto \I_{x\in A} $. 
 
Recalling again from the proof of Lemma 14.16 in \cite{gilb:trud:98} that for  $x\in V_\delta(A) $, $\nabla_x d_S(x,\partial A)=n(\Pi_{\partial A}(x)) $, where $n(\Pi_{\partial A}(x)) $ stands for the inner unit normal associated with the projection on the boundary, we get for $x\in V_\delta(A)\backslash A $:
\begin{equation}
\label{GRAD}
\nabla f_\delta(x)=-2\frac{d_S(x,\partial A) n(\Pi_{\partial A}(x))}{\delta^2}\Bigg(1-\frac{d_S(x,\partial A)^2}{\delta^2}\Bigg)^{-2}f_\delta(x).
\end{equation}
This yields in particular  that  $|\nabla f_\delta|_\infty=\sup_{x\in V_\delta(A)}|\nabla f_\delta(x)| \le C\delta^{-1} $. This last bound in particular yields that there exists $C\ge 1$ s.t. for all $\eta\in (0,\gamma] $, 
\begin{equation}
\label{MOD_HOLD_REG}
\sup_{x,y \in V_{2\delta}(A)}\frac{|f_\delta(x)-f_\delta(y)|}{|x-y|^\eta}\le C\delta^{-\eta}.
\end{equation}
Indeed, from the control on $|\nabla f_\delta|_\infty $ and the smoothness of $f_\delta $, we get for all $x,y\in V_{2\delta}(A)$, either $|x-y|\le \delta $ and $|f_\delta(x)-f_\delta(y)| \le C\delta^{-1}|x-y|\le C\delta^{-\eta}|x-y|^\eta$, 
or $|x-y|\ge \delta $ and $|f_\delta(x)-f_\delta(y)| \le C\le C\delta^{-\eta}|x-y|^\eta $.

Now,  the terms $T_1^\delta$ and $T_3^\delta$ in \eqref{DECOUP_IND_HOLDER} can be handled similarly thanks to the Gaussian upper bound that is satisfied, under \textbf{(A${}_H$)}, by the density of both the diffusion and its Euler scheme, see 
 Proposition \ref{THE_FINAL_PROP} or again \cite{sheu:91}, Theorem 2.1 in\cite{lema:meno:10}. Precisely, with the notations of \eqref{DECOUP_IND_HOLDER} and provided that $\delta \le (t_j-t_i)^{1/2} $:
 \begin{equation}
\label{CTR_DIST_DENS}
| T_1^\delta+T_3^\delta|\le \E[\I_{X_{t_j}^{t_i,x}\in V_\delta(A) }]+\E[\I_{X_{t_j}^{h,t_i,x}\in V_\delta(A) }]\le \frac{C \delta}{(t_j-t_i)^{1/2}}\exp\Big(-c\frac{d(x,\partial A)^2}{t_j-t_i}\Big),
 \end{equation}
where $d(x,\partial A)=|d_S(x,\partial A)| $ stands for the \textbf{nonnegative} distance to the boundary. 
 Indeed, we have that locally, up to a change of coordinate, only one variable is orthogonal to the straightened image of the hypersurface $\partial A $. We can thus integrate the Gaussian bounds w.r.t. the other ones. This yields the above control.
 
 Observe that to find the indicated convergence rate this imposes $\delta \le (t_j-t_i)^{1/2}h^{\gamma/2} $ which specifies the admissible magnitude for the parameter $\delta $.
 On the other hand, to analyze $T_2^\delta $ we recall from \eqref{MOD_HOLD_REG} that setting for all $(t,x)\in [0,t_j)\times \R^d $, $v_\delta(t,x):=\E[f_\delta(X_{t_j}^{t,x})]$ the terminal function $f_\delta $ is $\eta $-H\"older continuous, for all $\eta\in (0,\gamma] $, with H\"older modulus of continuity bounded by $\delta^{-\eta} $ on $V_{2\delta}( A) $. We will now establish, similarly to \eqref{CTR_GRAD_EXPLO}, that for all multi-index $\alpha, \ |\alpha|\le 2 $, $(t,x)\in [0,t_j)\times \R^d $:
 \begin{equation}
\label{CTR_GRAD_DELTA}
|D_x^\alpha v_\delta(t,x)|\le \frac{C}{(\delta \vee d(x,\partial A))^\eta}\frac{1}{(t_j-t)^{(|\alpha|-\eta)/2}}.
 \end{equation}
Recall indeed that
\begin{equation}
\label{EXP_GRAD}
|D_x^\alpha v_\delta(t,x)|
=\left|\int_{\R^d}D_x^\alpha p(t,t_j,x,y)(f_\delta(y)-f_\delta(x))dy\right|
\le \frac{C}{(t_j-t)^{|\alpha|/2}}\int_{\R^d} p_c(t_j-t,y-x)|f_\delta(y)-f_\delta(x)|dy,
\end{equation}
exploiting Proposition \ref{THE_FINAL_PROP} for the last inequality.
Thus, fromÊ \eqref{DEF_F_DELTA}: 
\begin{trivlist}
\item[-] if both $x,y\not \in V_\delta(A) $, then $f_\delta(x)=\I_{x\in A},f_\delta(y)=\I_{y\in A}  $. If $x\in (A\cup V_\delta(A))^C, y\in A\backslash V_\delta(A) $, or by symmetry $y\in (A\cup V_\delta(A))^C, x\in A\backslash V_\delta(A)  $, then $|x-y|\ge \delta \vee d(x,\partial A) $. If now $x,y\in (A\cup V_\delta(A))^C$ or $x,y\in A\backslash V_\delta(A)$ then $f_\delta(x)=f_\delta(y) $ yielding a trivial contribution in \eqref{EXP_GRAD}.
\item[-] if $x,y \in V_\delta(A) $, then the control of the H\"older modulus gives: $|f_\delta(x)-f_\delta(y)|\le C\delta^{-\eta }|x-y|^\eta= C(\delta \vee d(x,\partial A))^{-\eta}|x-y|^\eta$.
\item[-] if $x\in V_\delta(A),y\not \in V_\delta(A)  $ (resp. $y\in V_\delta(A),x\not \in V_\delta(A)  $) we can exploit the H\"older continuity for $y\in V_{2\delta}(A) $ (resp. $x\in V_{2\delta}(A) $) and the fact that $|x-y|\ge \delta \vee d(x,\partial A) $ for $y\not \in V_{2\delta}(A) $ (resp. $x\not \in V_{2\delta}(A) $).
\end{trivlist}
In all cases, we have established that $|f_\delta(x)-f_\delta(y)|\le C(\delta\vee d(x,\partial A))^{-\eta}|x-y|^\eta $, which plugged into \eqref{EXP_GRAD} yields the control \eqref{CTR_GRAD_DELTA}. Recall now that, again from Proposition \ref{THE_FINAL_PROP}, we have 
$v_\delta \in C^{\eta/2,\eta}([0,t_j]\times \R^d)\cap C^{1,2}([0,t_j)\times \R^d)$. In particular, $v_\delta$ has the same H\"older continuity modulus as $f_\delta $. We can as well assume w.l.o.g. that $\gamma/\eta\ge 1 $ so that $h^{\gamma/\eta}\le h\le 1 $.
 
Exploiting now \eqref{CTR_GRAD_DELTA} in an expansion similar to \eqref{DECOUP_ERR_EDP} and \eqref{THE_EQ_HOLDER_LAST_STEP}, we get:
\begin{eqnarray}
|T_2^\delta|&\le & |\E[f_\delta(X_{t_j}^{h,t_i,x})-f_\delta(X_{t_j-h^{\gamma/\eta}}^{h,t_i,x})|+|\E[v_\delta(t_j,X_{t_{j}-h^{\gamma/\eta}}^{h,t_i,x})-v_\delta(t_j-h^{\gamma/\eta},X_{t_{j}-h^{\gamma/\eta}}^{h,t_i,x})]|+\nonumber\\
&&+|\E[v_\delta(t_j-h^{\gamma/\eta},X_{t_{j}-h^{\gamma/\eta}}^{h,t_i,x})- v_\delta(t_{j-1},X_{t_{j-1}}^{h,t_i,x})]|\nonumber\\
&&+C \sum_{k=i}^{j-2}\int_{t_k}^{t_{k+1}} ds \E[\{|\nabla_x v_\delta(s,X_s^{h,t_i,x})|+|D_x^2 v_\delta(s,X_s^{h,t_i,x})|\} \{ |s-t_k|^{\gamma/2}+|X_s^{h,t_i,x}-X_{t_k}^{h,t_i,x}|^\gamma\}]\nonumber\\
&\le &Ch^{\gamma/2}\Big\{\E[(\delta \vee d(X_{t_j}^{h,t_i,x},\partial A))^{-2\eta}]^{1/2} +1 +\int_{t_i}^{t_{j}-h^{\gamma/\eta}} (1+\frac{1}{(t_j-s)^{1-\eta/2}}\E[\{ \delta \vee d(X_{s}^{h,t_i,x},\partial A)\}^{-2\eta}]^{1/2})ds\Big\},\nonumber\\\label{PREAL_T2_DELTA}
\end{eqnarray}
where the term $|\E[v_\delta(t_j-h^{\gamma/\eta},X_{t_{j}-h^{\gamma/\eta}}^{h,t_i,x})- v_\delta(t_{j-1},X_{t_{j-1}}^{h,t_i,x})]|$ is again expanded with It\^o's formula which yields bounds similar to those appearing for the contributions associated with the indexes $k\in \leftB i,j-2\rightB $.

Recalling as well that the Euler scheme satisfies the Aronson Gaussian bounds (see Proposition \ref{THE_FINAL_PROP} and Theorem 2.1 in \cite{lema:meno:10} for details) we obtain for all $s\in (t_i,t_j] $:
\begin{eqnarray*}
\E[\{ \delta \vee d(X_{s}^{h,t_i,x},\partial A)\}^{-2\eta}]\le C\Big\{ (\delta \vee d(x,\partial A))^{-2\eta}
+\int_{ \frac 12 d(x,\partial A )\ge d(y,\partial A)} \frac{\exp(-c\frac{|x-y|^{2}}{s-t_i})}{(\delta \vee d(y,\partial A))^{2\eta}}   \frac{dy}{(s-t_i)^{d/2}}\Big\}.
\end{eqnarray*}
Since on $\{ \frac 12 d(x,\partial A )\ge d(y,\partial A)\} $ we have $|x-y|\ge |x-\Pi_{\partial A}(y)|-|\Pi_{\partial A}(y)-y|\ge \frac {d(x,\partial A)}{2}\ge d(y,\partial A) $, where $\Pi_{\partial A}(y) $ again denotes the projection of $y$ on the boundary $\partial A $, we get:
\begin{eqnarray*}
\E[\{ \delta \vee d(X_{s}^{h,t_i,x},\partial A)\}^{-2\eta}]&\le& C\Big\{ (\delta \vee d(x,\partial A))^{-2\eta}
+\int_{ \frac 12 d(x,\partial A )\ge d(y,\partial A)} \frac{\exp(-c\frac{d(y,\partial A)^{2}}{s-t_i})}{(\delta \vee d(y,\partial A))^{2\eta}}   \frac{dy}{(s-t_i)^{d/2}}\Big\}\\
&\le& C\Big\{ (\delta \vee d(x,\partial A))^{-2\eta}
+1\Big\}.
\end{eqnarray*}
Hence, since  $d(x,\partial A)\ge  (t_j-t_i)^{1/2}h^{\gamma/2}\ge \delta $, we get from \eqref{PREAL_T2_DELTA}
\begin{eqnarray}
\label{EST_T_2}
|T_2^\delta| &\le & Ch^{\gamma/2}\Big\{\frac{1}{\eta d(x,\partial A)^\eta}+1\Big\}.
\end{eqnarray}
The point is now to find the $\eta\in (0,\gamma] $ maximizing $J_{x,A}:\eta\in (0,\gamma]\mapsto \eta d(x,\partial A)^\eta$ in order to minimize the associated contribution in $\frac{1}{\eta d(x,\partial A)^\eta} $ for $T_2^\delta $.  Two cases occur:
\begin{trivlist}
\item[-] $d(x,\partial A)\ge \exp(-\frac 1\gamma) $. In that case for $\eta\in (0,\gamma] $, $J_{x,A}'(\eta)=d(x,\partial A)^\eta(1+\eta \ln(d(x,\partial A))\ge 0 $ and the maximum over the constraint set is attained for $\eta=\gamma $ and $J_{x,A}(\eta)=\gamma d(x,\partial A)^\gamma $.
\item[-] $0<d(x,\partial A)< \exp(-\frac 1\gamma) $. The optimum is then attained for $\eta=-\frac{1}{\ln(d(x,\partial A))}\in (0,\gamma) $. This choice then yields: $J_{x,A}(\eta):=\frac{1}{|\ln(d(x,\partial A))|} \exp(\eta  \ln(d(x,\partial A)))=\frac{1}{|\ln(d(x,\partial A))|} e^{-1}$. 
\end{trivlist}
This gives from \eqref{EST_T_2} the global bound:
\begin{equation}
\label{EST_T_2_GLOB}
|T_2^\delta|\le Ch^{\gamma/2}\Big( 1+\frac{1}{\gamma d(x,\partial A)^\gamma}\I_{d(x,\partial A)\ge \exp(-\frac 1\gamma)}+|\ln(d(x,\partial A))|\I_{d(x,\partial A)< \exp(-\frac 1\gamma)}\Big).
\end{equation}

It is of course the last term above that becomes significant 
when the distance of the starting point comes closer to the boundary. 
The global error estimate deriving from \eqref{CTR_DIST_DENS}, the previous computations  on $J_{x,A} $ and \eqref{EST_T_2_GLOB} is then better, up to a multiplicative constant, than the one deriving from \eqref{ERR_EUL} as soon as:
 \begin{equation}
\label{BD_J}
 \frac{1}{J_{x,A}(\eta)}=|\ln(d(x,\partial A))| e\le  eh^{-C\psi(h)} \iff  |\ln(d(x,\partial A))| \le  
  h^{-C\psi(h)} \iff d(x,\partial A)\ge \exp(- 
  h^{-C\psi(h)}).
  \end{equation}
  Since to apply the Aronson's estimates for $T_1^\delta,T_3^\delta$ (see again eq. (3.20)) we had already assumed $ d(x,\partial A)\ge (t_j-t_i)^{1/2}h^{\gamma/2}\ge h^{(1+\gamma)/2}$, we derive that the condition in \eqref{BD_J} is always fulfilled. It can indeed be easily checked that $h^{(1+\gamma)/2}\ge  \exp(-h^{-C\psi(h)})$ for $h$ small enough.
Equation \eqref{HOLDER_BOREL} now follows from \eqref{CTR_DIST_DENS} and \eqref{EST_T_2_GLOB}.

\begin{REM}[Extension to piecewise smooth domains.]\label{PW_SMOOTH}
Let us mention that results similar to \eqref{HOLDER_BOREL} could also be derived for domains $A:=\cap_{i=1}^n A_i$ that write as finite intersections of smooth domains $(A_i)_{i\in \leftB 1,n\rightB} $ with bounded boundaries, and therefore have piecewise smooth boundary. In that case, $d(x,\partial A):=\inf_{i\in \{1,\cdots,n\}}d(x,\partial A_i)$ is well defined, but the corresponding signed distance can fail to be smooth, precisely close to the resulting \textit{corners}. Hence, $f_\delta$ cannot be directly defined as above. Namely,  some additional mollification of the corresponding distance would be necessary as well. 
\end{REM}

\subsection{Error Expansion for The Euler Scheme: Controls on the Densities.}
\label{SEC_EUL_ANA}
From Theorem 1.1, Theorem 2.1 and their proofs in \cite{kona:mamm:02} we have with the notations of the previous paragraph:
\begin{equation}
(p_{\varepsilon}-p_{\varepsilon}^h)(t_i,t_j,x,y)=
(p_{\varepsilon}-p_{\varepsilon}^d)(t_i,t_j,x,y)+ h\int_0^1 (1-\tau) \Big\{p_{\varepsilon}^d\otimes_h(\tilde L_{.,*}^{\varepsilon}-\tilde L.^{*,\varepsilon})^2 p_{\varepsilon}^{\tau,h}(t_i,t_j,x,y)\Big\}d\tau,
\label{ERR_EXP_EUL}
\end{equation}
where we denote for $0\le t_i<t_j\le T, \tau \in [0,1] $:
\begin{eqnarray*}
p_\varepsilon^{\tau,h}(t_i,t_j,x,y)&:=&\sum_{r=0}^{j-i}\tilde p_\varepsilon^{\tau} \otimes_h H_\varepsilon^{h,(r)}(t_i,t_j,x,y),\\ \forall (k,z)\in (i,j\rightB\times \R^d,\ \tilde p_\varepsilon^\tau(t_i,t_k,x,z)&:=&\int_{\R^d}\tilde p_\varepsilon^x(t_i,t_i+\tau h,x,w)\tilde p_\varepsilon^z(t_i+\tau h,t_k,w,z)dw.
\end{eqnarray*}
Also, for $k\in \{1,2\},t=t_{i+k},\ k\in \leftB 0,j-i-1\rightB$, $(\tilde L_{t,*}^{\varepsilon}) ^{k} \phi(x,y):=(L_{t,\xi}^{\varepsilon})^k \phi(x,y))|_{\xi=x}$,
$(\tilde L_{t}^{*,\varepsilon}) ^{k} \phi(x,y):=( \bar L_{t,\xi}^{\varepsilon})^k \phi(x,y))|_{\xi=y}$ for 
\begin{eqnarray*}
L_{t,\xi}^{\varepsilon}\phi(x,y)&=&\langle b_{\varepsilon}(t,\xi),\nabla_x \phi(x,y)\rangle +\frac 12{\rm Tr}(a_{\varepsilon}(t,\xi)D_x^2\phi(x,y)),\\
\bar L_{t,\xi}^{\varepsilon}\phi(x,y)&=&\frac 12{\rm Tr}(a_{\varepsilon}(t,\xi)D_x^2\phi(x,y)).
\end{eqnarray*}
Observe that $L_t^{\varepsilon}\phi(x,y)=L_{t,*}^{\varepsilon}\phi(x,y) $, but more generally the operators do not coincide anymore when iterated. Also, we indicate that the operators involved slightly differ from \cite{kona:mamm:02} since we chose to use a Gaussian process without drift as proxy, see \eqref{DEF_PROC_G} and \eqref{SCHEMES_AND_PARAM}. Another difference is the fact that we deal with inhomogeneous coefficients, and the notations $\tilde L_{.,*}^{\varepsilon},\tilde L.^{*,\varepsilon} $ in \eqref{ERR_EXP_EUL} are used to emphasize the time dependence of the operators in the discrete convolution $\otimes_h $. Anyhow, reproducing the proof of \cite{kona:mamm:02} taking into account the indicated differences  leads to the expression in \eqref{ERR_EXP_EUL}.

%

 We mention carefully that in order to analyze the contribution of the last term in the r.h.s. of \eqref{ERR_EXP_EUL}  no  smoothness in time of the coefficients is needed. On the other hand, such smoothness is clearly required to derive some convergence rates, since to control $ p_{\varepsilon}-p_{\varepsilon}^d$ we need to investigate the difference between time integrals and Riemann sums  (see Proposition \ref{PROP_C_D} and its proof below).

The term $\int_0^1(1-\tau) \{p_{\varepsilon}^d\otimes_h(\tilde L_{.,*}^{\varepsilon}-\tilde L.^{*,\varepsilon})^2 p_{\varepsilon}^{\tau,h}(t_i,t_j,x,y)\} d\tau$ involves derivatives of the coefficients and \textit{heat kernels} up to order 4. The point is again that the derivatives of the coefficients and kernels explode with $\varepsilon$ going to 0 (see equation \eqref{CTR_DER_MOLLI}). It is precisely this aspect that deteriorates the convergence rate w.r.t. the usual \textit{smooth} case. We carefully mention that if $\sigma(t,x)=\sigma$, the previous contributions involve lower derivatives of the heat kernel (up to order 2).

The key elements are now the following Propositions. 
The first one gives bounds for the derivatives of the densities involved in the parametrix series \eqref{SERIE_P_EPS}, \eqref{SERIE_P_EPS_D}.
The second one controls the difference between the discrete and continuous convolutions in 
\eqref{ERR_EXP_EUL}.
\begin{PROP}[Controls for the derivatives of the densities]
\label{THE_FINAL_PROP}
Let $\alpha,\ |\alpha|\le 4 $ be a multi-derivation index. 

Under \textbf{(A${}_H $)}, there exist constants $C\ge 1,c\in (0,1]$ s.t. for all $0\le s<t\le T,\ (x,y)\in (\R^d)^2$: 
\begin{equation}
\label{CTR_DER_PERT}
\begin{split}
|D_x^\alpha \bar p_{\varepsilon}(s,t,x,y)|\le \frac{C}{(t-s)^{|\alpha|/2}}p_c(t-s,y-x),|\alpha|\le 2,\\
|D_x^\alpha  \bar p_{\varepsilon}(s,t,x,y)|\le \frac{C}{(t-s)^{|\alpha|/2}}p_c(t-s,y-x)(1+\varepsilon^{-|\alpha|+2}(t-s)^{\gamma/2
}), |\alpha|\in \leftB 3,4\rightB,\\
|D_y^\alpha  \bar p_{\varepsilon}(s,t,x,y)|\le \frac{C\varepsilon^{-|\alpha|+\gamma}}{(t-s)^{|\alpha|/2}}p_c(t-s,y-x),\ |\alpha|\ge 1.
\end{split}
\end{equation}

Under \textbf{(A${}_{PS} $)}, for all $q>d,\ \eta\in (0,\alpha(q)),\ \alpha(q)=\frac 12(1-\frac dq)$,  there exist constants $C\ge 1,c\in (0,1]$ s.t. for all $0\le s<t\le T,\ (x,y)\in (\R^d)^2$: 
\begin{equation}
\label{CTR_DER_PERT_PS}
\begin{split}
|D_x \bar p_{\varepsilon}(s,t,x,y)|\le \frac{C}{(t-s)^{1/2}}p_c(t-s,y-x),\\
|D_x^\alpha  \bar p_{\varepsilon}(s,t,x,y)|\le \frac{C}{(t-s)^{|\alpha|/2}}p_c(t-s,y-x)(1+\bar C_{\eta,q}\varepsilon^{-|\alpha|+2-\eta+(1/q)\I_{|\alpha|\ge 3}}(t-s)^{\eta/2
}), |\alpha|\in \leftB 2,4\rightB,\\
|D_y  \bar p_{\varepsilon}(s,t,x,y)|\le \frac{C}{(t-s)^{1/2}}(1+\varepsilon^{-\eta}C_\eta(t-s)^{\eta/2})p_c(t-s,y-x),\\
|D_y^\alpha  \bar p_{\varepsilon}(s,t,x,y)|\le\frac{C(1+\bar C_{\eta,q}\varepsilon^{-|\alpha|+1-\eta}(t-s)^{\eta/2})}{(t-s)^{|\alpha|/2}}p_c(t-s,y-x),\ |\alpha|\in \leftB 2,4\rightB,
\end{split}
\end{equation}
where $\bar C_{\eta,q}=C_\eta\times C_q$ with $C_q$  as in Lemma \ref{LEMMA_SENS_BOUND} and $C_\eta$ as in Lemma \ref{LEMMA_SENS_HOLD}. 


In the above expressions $\bar p_{\varepsilon} $ can be any of the densities $p_{\varepsilon},p_{\varepsilon}^d,p_{\varepsilon}^{\tau,h} $ uniformly in $\tau\in [0,1] $. For $p_{\varepsilon}^d, p_{\varepsilon}^{\tau,h} $, the time variables $s,t $ are taken on the time grid.
\end{PROP}
\begin{REM}[Spatial H\"older continuity and heat-kernel bounds]\label{RM_ILIN}
We point out that the previous controls \eqref{CTR_DER_PERT} for $ \bar p_{\varepsilon}=p_{\varepsilon}$ would also hold under the sole spatial H\"older continuity of the coefficients $b,\sigma $. This improves in some sense those of \cite{ilin:kala:olei:62} which require smoothness in time of the coefficients. We get here the same pointwise controls for the derivatives of the non degenerate heat-kernel with spatial H\"older coefficients up to order 2, uniformly in $\varepsilon\in [0,1] $.
\end{REM}
\begin{REM}[Constants in \eqref{CTR_DER_PERT_PS}]\label{REM_PROP_A_PS}
Even though we are currently considering \A{A${}_{PS} $}, the associated small smoothing effect deriving from the regularization of the drift
is the same as for the sensitivities of densities under \A{A${}_H $}, for which it was induced by the small H\"older parameter for the difference of the diffusion coefficient and its regularization. In both cases the constant $C_\eta$ appears through the control of the corresponding parametrix series, see the proofs of Theorem \ref{MTHM_KKM1}, Lemma \ref{LEMMA_SENS_HOLD} and Proposition \ref{THE_FINAL_PROP} below.
\end{REM}

\begin{PROP}[Bounds for the difference between continuous and discrete time convolutions]\label{PROP_C_D}
Under \A{A${}_H$}, there exist $C\ge 1,c\in (0,1]$ s.t. for all $0\le t_i<t_j\le T,\ (x,y)\in (\R^d)^2 $, $\eta\in (0,\gamma)$:
\begin{equation}
\label{CTR_PN_PND}
|(p_{\varepsilon}-p_{\varepsilon}^d)(t_i,t_j,x,y)|\le C_\eta
h^{(\gamma-\eta)/2}
p_c(t_j-t_i,y-x).
\end{equation}
Under \A{A${}_{PS}$}, there exist $C\ge 1,c\in (0,1]$ s.t. for all $0\le t_i<t_j\le T,\ (x,y)\in (\R^d)^2 $, $q>d,\ \eta\in (0,\alpha(q))$:
\begin{equation}
\label{CTR_PN_PND_PS}
|(p_{\varepsilon}-p_{\varepsilon}^d)(t_i,t_j,x,y)|\le 
 \bar C_{\eta,q}\Big( h|\ln(h)|\varepsilon^{-(1+\eta)}
+h^{1-\eta/2}\varepsilon^{-(1+\eta)}+h\varepsilon^{-2+1/q}\Big) p_c(t_j-t_i,y-x).
\end{equation}
with $\alpha(q), \bar C_{\eta,q}$ as in Proposition \ref{THE_FINAL_PROP}. 

If now $d(y,{\mathcal V}_\varepsilon({\mathcal I}))\ge 2\varepsilon$ the previous bound improves to 
\begin{equation}
\label{CTR_PN_PND_PS_DIST}
|p_\varepsilon-p_\varepsilon^d|(t_i,t_j,x,y)
\le \bar C_{\eta,q}\Big( h|\ln(h)|\varepsilon^{-(1+\eta)}
+h^{1-\eta/2}\varepsilon^{-(1+\eta)}+\frac{h^{1-\eta/2}}{d(y,{\mathcal V}_\varepsilon({\mathcal I}))}\Big) p_c(t_j-t_i,y-x).
\end{equation}

If $d(y,{\mathcal V}_\varepsilon({\mathcal I}))\ge 2\varepsilon$ and additionally $\sigma(t,x)=\sigma$, i.e. constant diffusion term, then
\begin{equation}
\label{CTR_PN_PND_PS_DRIFT}
|p_\varepsilon-p_\varepsilon^d|(t_i,t_j,x,y)
\le  \bar C_{\eta,q}\Big( h|\ln(h)|\varepsilon^{-\eta}
+h\varepsilon^{-(1+\eta)+1/q}+\frac{h^{1-\eta/2}}{d(y,{\mathcal V}_\varepsilon({\mathcal I}))}\Big) p_c(t_j-t_i,y-x).
%
%
%
\end{equation}

\end{PROP}

We postpone the proof of Propositions \ref{THE_FINAL_PROP} and \ref{PROP_C_D}  to Section \ref{SEC_PROOF_ANA} for clarity.
It now remains to exploit Propositions \ref{THE_FINAL_PROP}, \ref{PROP_C_D} and \eqref{ERR_EXP_EUL} to 
specifically control how the weak error for the densities depends on the explosive norms of the mollified coefficients.

\subsubsection{Proof of The Main Results for H\"older Coefficients (Theorem \ref{THM_HOLDER_ES} under \A{A${}_H$})}
\label{SEC_PROOF_THM_1}

Observe from Proposition \ref{THE_FINAL_PROP} that, for all $k\in $ $\leftB i,j-1\rightB
,(z,y)\in \left( \R^{d}\right) ^{2},\tau \in [ 0,1]$,%
\[
\left\vert \left( \widetilde{L}_{t_{k},\ast }^{\varepsilon }-\widetilde{L}%
_{t_{k}}^{\ast ,\varepsilon }\right) p_{\varepsilon }^{\tau
,h}(t_{k},t_{j},z,y)\right\vert =\left\vert \left\langle b_{\varepsilon
}(t_{k},z),D_{z}p_{\varepsilon }^{\tau ,h}(t_{k},t_{j},z,y)\right\rangle
+\right. 
\]%
\[
\left. \frac{1}{2}{\rm Tr}\left( (a_{\varepsilon }(t_{k},z)-a_{\varepsilon
}(t_{k},y))D_{z}^{2}p_{\varepsilon }^{\tau ,h}(t_{k},t_{j},z,y)\right)
\right\vert \leq \frac{C}{\left( t_{j}-t_{k}\right) ^{1-\gamma /2}}%
p_{c}(t_{j}-t_{k},y-z). 
\]

Iterating the frozen operator, we obtain that $\left( \widetilde{L}_{t_{k},\ast }^{\varepsilon }-%
\widetilde{L}_{t_{k}}^{\ast ,\varepsilon }\right) ^{2}p_{\varepsilon }^{\tau
,h}(t_{k},t_{j},z,y)$ is a fourth order differential operator which is the
sum of the following typical terms:%
\begin{eqnarray}
b_{\varepsilon }^{l}(t_{k},z)b_{\varepsilon
}^{m}(t_{k},z)D_{z_{l}z_{m}}^{2}p_{\varepsilon }^{\tau ,h}(t_{k},t_{j},z,y)&=:&\Psi_{l,m}^{\varepsilon,\tau,h}(t_{k},t_{j},z,y), \nonumber\\
b_{\varepsilon }^{l}(t_{k},z)\left( (a_{\varepsilon
}^{mq}(t_{k},z)-a_{\varepsilon }^{mq}(t_{k},y))D_{z_{l}z_{m}z_{q}}^{3}\text{ 
}p_{\varepsilon }^{\tau ,h}(t_{k},t_{j},z,y)\right)&=:&\Psi_{l,m,q}^{\varepsilon,\tau,h}(t_{k},t_{j},z,y) , \nonumber\\
\label{DECOUP_OPERATEUR_IT}
(a_{\varepsilon }^{lm}(t_{k},z)-a_{\varepsilon
}^{lm}(t_{k},y))(a_{\varepsilon }^{qr}(t_{k},z)-a_{\varepsilon
}^{qr}(t_{k},y))D_{z_{l}z_{m}z_{q}z_{r}}^{4}p_{\varepsilon }^{\tau
,h}(t_{k},t_{j},z,y)&=:&\Psi_{l,m,q,r}^{\varepsilon,\tau,h}(t_{k},t_{j},z,y),
\end{eqnarray}%
for $l,m,q,r\in \leftB 1,d\rightB  $.
It is easy to see that the terms with fourth derivatives are the most
singular. Hence,  to evaluate 
$
p_{\varepsilon }^{d}\otimes _{h}\left( \widetilde{L}_{\cdot,\ast
}^{\varepsilon }-\widetilde{L}_{\cdot}^{\ast ,\varepsilon }\right)
^{2}p_{\varepsilon }^{\tau ,h}(t_{i},t_{j},x,y)$,
it is enough to concentrate on: 
\begin{eqnarray}
&&p_{\varepsilon }^{d}\otimes _{h}\Psi_{l,m,q,r}^{\varepsilon,\tau,h}(t_{i},t_{j},x,y)\nonumber\\
&&= 
h\big(a_{\varepsilon
}^{lm}(t_{i},x)-a_{\varepsilon }^{lm}(t_{i},y)\big)\big(a_{\varepsilon
}^{qr}(t_{i},x)-a_{\varepsilon }^{qr}(t_{i},y)\big)
D_{z_{l}z_{m}z_{q}z_{r}}^{4}p_{\varepsilon }^{\tau ,h}(t_{i},t_{j},x,y)
\nonumber\\
&&+h\sum_{k\in \leftB i+1,\left\lceil \frac{i+j}{2}\right\rceil \rightB
}\int_{\R^{d}}p_{\varepsilon }^{d}(t_{i},t_{k},x,z)\big(a_{\varepsilon
}^{lm}(t_{k},z)-a_{\varepsilon }^{lm}(t_{k},y)\big)\big(a_{\varepsilon
}^{qr}(t_{k},z)-a_{\varepsilon }^{qr}(t_{k},y)\big)
D_{z_{l}z_{m}z_{q}z_{r}}^{4}p_{\varepsilon }^{\tau ,h}(t_{k},t_{j},z,y)dz\nonumber\\
&&+h\!\!\!\!\!\!\sum_{k\in \leftB \left\lceil \frac{i+j}{2}\right\rceil +1
,j-1\rightB }\!
\int_{\R^{d}}p_{\varepsilon }^{d}(t_{i},t_{k},x,z)
\big(a_{\varepsilon }^{lm}(t_{k},z)-a_{\varepsilon
}^{lm}(t_{k},y)\big)\big(a_{\varepsilon }^{qr}(t_{k},z)-a_{\varepsilon
}^{qr}(t_{k},y)\big)D_{z_{l}z_{m}z_{q}z_{r}}^{4}p_{\varepsilon }^{\tau
,h}(t_{k},t_{j},z,y)dz\nonumber\\
&&=:(T_1+T_2+T_3)(t_i,t_j,x,y).  \label{1b}
\end{eqnarray}%
The tools to control the above terms are \eqref{CTR_DER_PERT} in Proposition \ref{THE_FINAL_PROP} and the H\"older continuity of the mollified
coefficients under \A{A${}_H $}. We readily derive:
\begin{equation}
\label{CTR_FIRST_STEP}
\left\vert T_1(t_i,t_j,x,y)\right\vert \leq \frac{Ch(1+\varepsilon ^{-2}(t_j-t_i)^{\gamma/2})}{(t_{j}-t_{i})^{2-\gamma}}%
p_{c}(t_{j}-t_{i},y-x).
\end{equation}
For the term $T_2$ in (\ref{1b}), integrating once by
parts, we obtain from \eqref{CTR_DER_PERT} and \eqref{CTR_DER_MOLLI} that:
\begin{equation}
\left\vert T_2(t_i,t_j,x,y)\right\vert \leq \frac{C\varepsilon ^{-1+\gamma }}{(t_{j}-t_{i})^{1-\gamma}}%
p_{c}(t_{j}-t_{i},y-x)\Big(1+\varepsilon ^{-1}(t_{j}-t_{i})^{\gamma /2}
\Big).
\label{2}
\end{equation}%
The term $T_3$ in (\ref{1b}) can be handled using the same arguments and  two integrations by
parts in order to get rid of the time singularities. After integrations by parts, the most singular terms w.r.t. $\varepsilon $ have the
following form:
\begin{eqnarray}
T_{31}(t_i,t_j,x,y)&:=&\text{ }h\sum_{k\in \leftB \left\lceil \frac{i+j}{2}\right\rceil +1,j-1\rightB
}\int_{\R^{d}}p_{\varepsilon }^{d}(t_{i},t_{k},x,z)D_{z_{l}}a_{\varepsilon
}^{lm}(t_{k},z)D_{z_{m}}a_{\varepsilon
}^{qr}(t_{k},z)D_{z_{q}z_{r}}^{2}p_{\varepsilon }^{\tau
,h}(t_{k},t_{j},z,y)dz,\nonumber \\
T_{32}(t_i,t_j,x,y)&:=& 
h\sum_{k\in \leftB \left\lceil \frac{i+j}{2}\right\rceil +1,j-1\rightB
}\int_{\R^{d}}D_{z_lz_m}^2p_{\varepsilon }^{d}(t_{i},t_{k},x,z)\left[ (a_{%
\varepsilon }^{lm}(t_{k},z)- a_{%
\varepsilon }^{lm}(t_{k},y))(a_{\varepsilon }^{qr}(t_{k},z)
-a_{\varepsilon }^{qr}(t_{k},y))
\right] \nonumber\\
&&\hspace{1.5cm}\times D_{z_{q}z_{r}}^{2}p_{\varepsilon }^{\tau
,h}(t_{k},t_{j},z,y)dz.  \label{3}
\end{eqnarray}%

For $T_{31}$, we obtain from inequality \eqref{CTR_DER_PERT} in Proposition \ref{THE_FINAL_PROP} and \eqref{CTR_DER_MOLLI} that: 
\begin{eqnarray}
\left\vert T_{31}(t_i,t_j,x,y)\right\vert &\leq &C\varepsilon ^{-2+2\gamma
}p_{c}(t_{j}-t_{i},y-x)\sum_{k\in \leftB \left\lceil \frac{i+j}{2}%
\right\rceil +1,j-1\rightB}\frac{h}{(t_{j}-t_{k})}\nonumber\\
&\leq& 
C\varepsilon ^{-2+2\gamma }p_{c}(t_{j}-t_{i},y-x)\int_{\frac{t_{i}+t_{j}}{%
2}}^{t_{j}-h}\frac{du}{t_{j}-u}\leq C\varepsilon ^{-2+2\gamma
}p_{c}(t_{j}-t_{i},y-x)\left\vert \ln h\right\vert . 
 \label{5b}
\end{eqnarray}%
For $T_{32}$, 
Proposition \ref{THE_FINAL_PROP} and the spatial H\"older continuity of $a_\varepsilon$ yield:%
\begin{equation}
\left\vert T_{32}(t_i,t_j,x,y)\right\vert \leq \frac{C\varepsilon ^{-2+\gamma }}{%
(t_{j}-t_{i})^{1-\gamma}}p_{c}(t_{j}-t_{i},y-x).  \label{4}
\end{equation}%

An upper-bound for $T_3$ then follows summing \eqref{4} and \eqref{5b}. We then derive from \eqref{CTR_FIRST_STEP}, \eqref{2} and \eqref{1b} that:
\begin{eqnarray}
\left\vert p_{\varepsilon }^{d}\otimes _{h}\left( \widetilde{L}_{\cdot,\ast
}^{\varepsilon }-\widetilde{L}_{\cdot}^{\ast ,\varepsilon }\right)
^{2}p_{\varepsilon }^{\tau ,h}(t_{i},t_{j},x,y)\right\vert \nonumber\\
\leq 
C
\Big(\frac{\varepsilon
^{-2+\gamma}}{(t_j-t_i)^{1-\gamma}}+\frac{h(1+\varepsilon^{-2}(t_j-t_i)^{\gamma/2})}{(t_j-t_i)^{2-\gamma}}+\varepsilon ^{-2+2\gamma }\left\vert \ln
h\right\vert \Big)p_{c}(t_{j}-t_{i},y-x). 
\label{6}
\end{eqnarray}%
We thus eventually get from \eqref{ERR_EXP_EUL}, \eqref{CTR_PN_PND} and (\ref{6}):
\begin{eqnarray}
\left\vert \left( p_{\varepsilon }-p_{\varepsilon }^{h}\right)
(t_{i},t_{j},x,y)\right\vert \nonumber\\
\leq 
C\Bigg\{ 
C_\eta h^{( \gamma -\eta ) /2}
+h
\left(\frac{h}{(t_j-t_i)^{2-\gamma}}+\frac{\varepsilon
^{-2+\gamma}}{(t_{j}-t_{i})^{1-\gamma }}\big(1+
\frac{\varepsilon^{-\gamma}}{(t_j-t_i)^{1-\gamma/2}}h\big)+\varepsilon ^{-2+2\gamma }\left\vert \ln
h\right\vert \right)\Bigg\} p_{c}(t_{j}-t_{i},y-x).  \label{7}
\end{eqnarray}%
Without loss of generality we assume now that $0\leq t_{j}-t_{i}\leq T\leq
1$. We also suppose that:
\begin{equation}%
\label{CTR_EPSILON_TI_TJ}
\left(\frac{h}{(t_j-t_i)^{1-\gamma/2}}\right)^{1/\gamma} \le \varepsilon.
\end{equation}
We will check that \eqref{CTR_EPSILON_TI_TJ} holds for the specific choice of the parameters $\varepsilon, \eta $ which is performed below.

We derive from
equations (\ref{7}), \eqref{CTR_EPSILON_TI_TJ}, together with \eqref{DECOUP_ERR}, \eqref{BD_PREAL_EUL_HOLD}  that:%
\begin{equation}
\left\vert \left( p-p^{h}\right) (t_{i},t_{j},x,y)\right\vert \leq C\left\{
C_{\eta }\left( \varepsilon ^{\gamma -\eta }+h^{\left( \gamma -\eta \right)
/2}\right) +\frac{h\varepsilon ^{-2+\gamma }}{(t_{j}-t_{i})^{1-\gamma }}+%
h\varepsilon ^{-2+2\gamma }\left\vert \ln h\right\vert
\right\} 
p_{c}(t_{j}-t_{i},y-x).  \label{8b}
\end{equation}%
Take now 
\[
C_{\eta }\varepsilon ^{\gamma -\eta }=\frac{h\varepsilon ^{-2+\gamma }}{%
(t_{j}-t_{i})^{1-\gamma }}\Longleftrightarrow \varepsilon =\left( \frac{h}{%
(t_{j}-t_{i})^{1-\gamma }}\right) ^{1/(2-\eta )}C_{\eta }^{-1/(2-\eta )}. 
\]%
For such a choice of a mollifying parameter we have for $(t_{j}-t_{i})\geq
h^{1/(2-\gamma )}$:
\begin{eqnarray*}
\varepsilon ^{\gamma }\left\vert \ln h\right\vert &=& \left( \frac{h}{%
(t_{j}-t_{i})^{1-\gamma }}\right) ^{\gamma /(2-\eta )}C_{\eta }^{-\gamma
/(2-\eta )}\left\vert \ln h\right\vert \leq 
h^{\gamma /((2-\gamma )(2-\eta ))}C_{\eta }^{-\gamma /(2-\eta )}\left\vert
\ln h\right\vert.
\end{eqnarray*}
Assume for a while that $\eta$ can be taken so that:
\begin{equation}
\label{CHOICE_LOG}
h^{\gamma /((2-\gamma )(2-\eta))}C_{\eta }^{-\gamma /(2-\eta )}
|\ln h|\le h^{\gamma /((2-\gamma )2)} \Longleftarrow \frac{2(2-\gamma)}{\gamma} \frac{\ln_2(h^{-1})}{\ln(h^{-1})}\le \frac{\eta}{2},
\end{equation}
recalling as well that $C_\eta\ge 1$ for the last assertion.
Then, for $(t_{j}-t_{i})\geq
h^{1/(2-\gamma )}$ if \eqref{CHOICE_LOG} holds:
\begin{equation*}
\varepsilon ^{\gamma }\left\vert \ln h\right\vert \le h^{\gamma /((2-\gamma )2)} \le  (t_{j}-t_{i})^{\gamma
/2} .
\end{equation*}%
Hence, from \eqref{8b}, if \eqref{CHOICE_LOG} holds:
\begin{eqnarray}
|(p-p^h)(t_i,t_j,x,y)|\le C \Big\{C_\eta h^{(\gamma-\eta)/2}+C_\eta^{\frac{2-\gamma}{2-\eta}} \left(\frac{h}{(t_j-t_i)^{1-\gamma}} \right)^{\frac{\gamma-\eta}{2- \eta }}\Big\} p_c(t_j-t_i,y-x)\nonumber\\
\le  C \exp(C(2\eta^{-1}+1)^{2\eta^{-1}+1} ) \Big\{h^{(\gamma-\eta)/2}+ \Big(\frac{h}{(t_j-t_i)^{1-\gamma}}\Big)^{\frac\gamma 2-\eta \frac{1-\gamma/2}{2-\eta}}\Big\}p_c(t_j-t_i,y-x),
\label{NEW_PREAL_BD_THM_1}
\end{eqnarray}
using the bounds of Lemma \ref{LEMMA_SENS_HOLD} for the last inequality.
The point is now to carefully choose $\eta:=\eta(h) $.
Let us consider the specific sequence $\eta=\eta(h):=2\frac{\log_3(h^{-1})}{\log_2(h^{-1})} $, where we recall that for $k\in \N$, $\log_k(x)$ stands for the $k^{\rm th} $ iterated logarithm of $x$. Observe that this $\eta(h) $ satisfies the condition \eqref{CHOICE_LOG} for $h$ small enough. 
 Setting $\beta_h:=h^{-\eta}$ and $\alpha_h:=\exp\left( C(2\eta^{-1}+1)^{2\eta^{-1}+1}\right) $, we get that:
 \begin{eqnarray*}
\log_2(\beta_h)&=&\log(\eta \log(h^{-1}))= \log(2)+\log_4(h^{-1})-\log_3(h^{-1})  +\log_2(h^{-1}),\\
\log_2(\alpha_h)&=& \log(C(2\eta^{-1}+1)^{2\eta^{-1}+1})=\log(C)+(2\eta^{-1}+1)\log(2\eta^{-1}+1)\\
&=&\log(C)+(\frac{\log_2(h^{-1})}{\log_3(h^{-1})}+1)\log(2\eta^{-1}(1+\frac{\eta}2))\\
&=&\log(C)+(\frac{\log_2(h^{-1})}{\log_3(h^{-1})}+1)\{\log(2\eta^{-1})+\log(1+\frac{\eta}2)\}\\
&=&\log(C)+(\frac{\log_2(h^{-1})}{\log_3(h^{-1})}+1)\{\log_3(h^{-1})-\log_4(h^{-1})+\log(1+\frac{\eta}2)\}\\
&=&\log_2(h^{-1})-\frac{\log_2(h^{-1})\log_4(h^{-1})}{\log_3(h^{-1})}+\log_3(h^{-1})-\log_4(h^{-1})+R_h,\\
R_h&:=&\log(C)+
\log(1+\frac{\log_3(h^{-1})}{\log_2(h^{-1})})\left\{ \frac{\log_2(h^{-1})}{\log_3(h^{-1})}+1\right\}.
 \end{eqnarray*}
It is easily seen that there exists a finite constant $\bar C>0 $ s.t. for all $h$ small enough, $R_h\le \bar C $ and that $\log_2(\beta_h)\ge \log_2(\alpha_h) -\bar C$. By monotonicity of the exponential, recalling as well that $\eta\in (0,\gamma) $, we thus derive:
\begin{equation}
\label{ERROR_SENSI_EUL}
\big(\beta_h^{\frac12}+\beta_h^{\frac{1-\gamma/2}{2-\eta}} \big)\alpha_h=\big(h^{-\frac{\eta}{2}}+h^{-\eta \frac{1-\gamma/2}{2-\eta}}\big)\exp(C (2\eta^{-1}+1)^{2 \eta^{-1}+1})\le  2h^{-\eta(1/2+ \exp(\bar C))}.
\end{equation}
The previous choice of $\eta$ yields that, since $C_\eta=C\alpha_h$, \eqref{CTR_EPSILON_TI_TJ} is satisfied as well. 
Plugging \eqref{ERROR_SENSI_EUL} into \eqref{NEW_PREAL_BD_THM_1}
we complete the proof of equation \eqref{ERR_EUL} in Theorem \ref{THM_HOLDER_ES}.

\subsubsection{Proof of The Main Results for piecewise smooth coefficients (Theorem \ref{THM_HOLDER_ES_PSD} under \A{A${}_{PS}$})} 
\label{SEC_PROOF_AS}
Keeping the definitions of \eqref{DECOUP_OPERATEUR_IT}, the idea is to proceed as in the previous section from equations \eqref{ERR_EXP_EUL},  and \eqref{1b}. 
To emphasize the specificity of Assumptions \A{A${}_{PS} $}, due to the approximation of the piecewise smooth drift, we begin with the special case $\sigma(t,x)=\sigma $. In that framework, the only terms appearing in $(\tilde L_{.,*}^{\varepsilon}-\tilde L_{.}^{*,\varepsilon})^2p_{\varepsilon}^{\tau,h} $ are the $\Psi_{l,m}^{\varepsilon,\tau,h} $ introduced in \eqref{DECOUP_OPERATEUR_IT}.
From equation \eqref{CTR_DER_PERT_PS} in Proposition \ref{THE_FINAL_PROP}, using
 a direct control for the index $k=i$ and a global integration by part for $k>i$, associated with the bound of \eqref{THE_DER_BEPS_BETA}, we derive:
\begin{eqnarray*}
|[p_{\varepsilon}^d\otimes_h (\tilde L_{.,*}^{\varepsilon}-\tilde L_{.}^{*,\varepsilon})^2p_{\varepsilon}^{\tau,h}](t_i,t_j,x,y)|\\
\le C\Big(\frac{\bar C_{\eta,q}\varepsilon^{-\eta}}{(t_j-t_i)}hp_c(t_j-t_i,y-x)+ h\!\!\!\!\! 
\sum_{k\in \leftB i+1,j-1 \rightB}  \Big|\int_{\R^d } {\rm div}_z  \Big( p_\varepsilon^d(t_i,t_k,x,z) b_\varepsilon(t_k,z)\Big) \langle b_\varepsilon(t_k,z), \nabla_z p_{\varepsilon}^{\tau,h}(t_k,t_j,z,y)\rangle
dz\Big|\Big)\nonumber\\
\le C\Big( \bar C_{\eta,q}\varepsilon^{-\eta}p_c(t_j-t_i,y-x)\\
+h \!\!\!\! \sum_{k\in \leftB i+1,j-1 \rightB}  \int_{\R^d } \Big(\frac{\bar C_{\eta,q}\varepsilon^{-\eta}}{(t_k-t_i)^{1/2}}+(1+\varepsilon^{-1}\I_{z\in {\mathcal V}_\varepsilon}({\mathcal I})) \Big) p_c(t_k-t_i,z-x) \frac{p_c(t_j-t_k,y-z)}{(t_j-t_k)^{1/2}}dz\Big).
\end{eqnarray*}
The point is now to use the H\"older inequality to exploit that the set on which $\nabla_z b_\varepsilon $ gives an explosive bound is \textit{small}. We get:
\begin{eqnarray*}
|[p_{\varepsilon}^d\otimes_h (\tilde L_{.,*}^{\varepsilon}-\tilde L_{.}^{*,\varepsilon})^2p_{\varepsilon}^{\tau,h}](t_i,t_j,x,y)|
\le C\Big(\bar C_{\eta,q}\varepsilon^{-\eta}p_c(t_j-t_i,y-x)\\
+h\sum_{k\in \leftB i+1,j-1 \rightB} \frac{1}{(t_j-t_k)^{1/2}}\Big\{\frac{\bar C_{\eta,q}\varepsilon^{-\eta}}{(t_k-t_i)^{1/2}}p_c(t_j-t_i,y-x)
\\+\varepsilon^{-1+1/q}
 \Big(\int_{\R^d} p_c(t_k-t_i,z-x)^{\bar q}p_c(t_j-t_k,y-z)^{\bar q}dz\Big)^{1/\bar q}\Big\}\Big),
\end{eqnarray*}
denoting by $\bar q>1$ the conjugate of $q$, $q^{-1}+\bar q^{-1}=1 $. Recall now that:
\begin{eqnarray*}
\Big(\int_{\R^d} p_c(t_k-t_i,z-x)^{\bar q}p_c(t_j-t_k,y-z)^{\bar q}dz\Big)^{1/\bar q}&=& \Big(\frac{c(t_j-t_i)}{(2\pi)(t_k-t_i)(t_j-t_k)}\Big)^{d/(2q)
}\bar q^{-d/(2\bar q)}  p_c(t_j-t_i,y-x).
\end{eqnarray*}
This yields:
\begin{eqnarray}
\label{THE_EQ_ERR_1T_PS_NO_SIG}
|[p_{\varepsilon}^d\otimes_h (\tilde L_{.,*}^{\varepsilon}-\tilde L_{.}^{*,\varepsilon})^2p_{\varepsilon}^{\tau,h}](t_i,t_j,x,y)|
&\le &C\Big(\bar C_{\eta,q}\varepsilon^{-\eta}+\frac{1}{\alpha(q)}\varepsilon^{-1+1/q}(t_j-t_i)^{\alpha(q)}\Big)p_c(t_j-t_i,y-x)\nonumber\\
&\le& C\bar C_{\eta,q}\varepsilon^{-1+1/q}p_c(t_j-t_i,y-x),
\end{eqnarray}
as soon as $\varepsilon^{1-\frac 1q-\eta}\le 1 $ which holds true for $\eta $ small enough (remember $q>d $).

Performing now in the general case, involving derivatives of the heat kernel up to order 4, an integration by part similar to the one described for \eqref{1b} and using the H\"older inequality as above for the terms involving derivatives of $b_\varepsilon $, we derive from \eqref{CTR_DER_PERT_PS} in Proposition \ref{THE_FINAL_PROP}, that for all $ q>d,\ \eta\in (0,\alpha(q)) $:
 \begin{equation}
\label{THE_EQ_ERR_1T_PS}
|[p_{\varepsilon}^d\otimes_h (\tilde L_{.,*}^{\varepsilon}-\tilde L_{.}^{*,\varepsilon})^2p_{\varepsilon}^{\tau,h}](t_i,t_j,x,y)|
\le 
C
\left\{1+\bar C_{\eta,q}\varepsilon^{-(1+\eta)}(1+\varepsilon^{\eta/2}|\ln(h)|) 
\right\}p_c(t_j-t_i,y-x).
\end{equation}

We thus get in whole generality, from \eqref{BD_PREAL_EUL_BOUND}, \eqref{ERR_EXP_EUL}, \eqref{THE_EQ_ERR_1T_PS} and \eqref{CTR_PN_PND_PS}  in Proposition \ref{PROP_C_D}:
\begin{eqnarray*}
|p-p^h(t_i,t_j,x,y)|&\le& C\Big[C_q
\varepsilon^{1/q} + 
\bar C_{\eta,q}
h\varepsilon^{-(1+\eta)}(1+\varepsilon^{\eta/2}|\ln(h)|)
\\
&&
+
 \bar C_{\eta,q}\Big( h|\ln(h)|\varepsilon^{-(1+\eta)}
+h^{1-\eta/2}\varepsilon^{-(1+\eta)}+h\varepsilon^{-2+1/q}\Big) \Big]p_c(t_j-t_i,y-x).
\end{eqnarray*}
If now $d(y,{\mathcal V}_\varepsilon({\mathcal I}))\ge 2 \varepsilon $, then, from \eqref{CTR_PN_PND_PS_DIST} in Proposition \ref{PROP_C_D}:
\begin{eqnarray*}
|p-p^h(t_i,t_j,x,y)|&\le& C\Big[C_q
\varepsilon^{1/q} + 
\bar C_{\eta,q}
h\varepsilon^{-(1+\eta)}(1+\varepsilon^{\eta/2}|\ln(h)|)
\\
&&
+
 \bar C_{\eta,q}\Big( h|\ln(h)|\varepsilon^{-(1+\eta)}
+h^{1-\eta/2}\varepsilon^{-(1+\eta)}+\frac{h^{1-\eta}}{d(y,{\mathcal V}_\varepsilon({\mathcal I}))}\Big) \Big]p_c(t_j-t_i,y-x).
\end{eqnarray*} 
Eventually, if we additionally have that $\sigma(t,x)=\sigma $, \eqref{CTR_PN_PND_PS_DRIFT} in Proposition \ref{PROP_C_D} and  \eqref{THE_EQ_ERR_1T_PS_NO_SIG}  yield:
\begin{eqnarray*}
|p-p^h(t_i,t_j,x,y)|&\le& C\Big[C_q
\varepsilon^{1/q} +
 \bar C_{\eta,q}h\varepsilon^{-1+1/q}\\
&&
+\bar C_{\eta,q}\Big( h|\ln(h)|\varepsilon^{-\eta}
+h\varepsilon^{-(1+\eta)+1/q}+\frac{h^{1-\eta}}{d(y,{\mathcal V}_\varepsilon({\mathcal I}))}\Big) \Big]p_c(t_j-t_i,y-x).
\end{eqnarray*}
We then set $ C_q
\varepsilon^{1/q}=\bar C_{\eta,q}h\varepsilon^{-2+1/q}$ in the general case, i.e. for $b,\sigma $ depending both on the spatial variable and without any distance condition for the final point $y$.
If $ d(y,{\mathcal V}_\varepsilon({\mathcal I}))\ge 2\varepsilon$, we take
$ C_q
\varepsilon^{1/q}=\bar C_{\eta,q}
h\varepsilon^{-(1+\eta)}
$ for a general $\sigma $
 and $C_q
\varepsilon^{1/q}= \bar C_{\eta,q}h\varepsilon^{-(1+\eta)+1/q}$ if $\sigma(t,x)=\sigma $. The results can be derived  as in the previous section choosing $\eta:=\eta(h)=\psi(h)$, $q:=q(h)$ s.t. $\alpha(q)=\psi(h)$. For \eqref{ERR_EUL_PS_SD_GLOB} and \eqref{ERR_EUL_PS_SD}, we recall as well that if $d(y,{\mathcal I})\ge h^{1/2-\epsilon} $ for a fixed given $\epsilon>0 $ for a general $\sigma $ and $d(y,{\mathcal I})\ge h^{1-\epsilon} $
for $\sigma(t,x)=\sigma $, the condition $d(y,{\mathcal V}_\varepsilon({\mathcal I}))\ge 2\varepsilon $ is met.

\mysection{Proof of the Technical Results from Section \ref{HK_ANA}.}
\label{PROOF_HK}

\label{SEC_PROOF_ANA}

\subsection{Proof of Proposition \ref{THE_FINAL_PROP}.}

\subsubsection{Proof under \textbf{(A${}_H $)}.}
Let us establish the result for $p_{\varepsilon}$. 
We start from the parametrix representation of $p_{\varepsilon} $ obtained in \eqref{SERIE_P_EPS}.
In all cases, we can readily derive from \eqref{DEF_PROC_G} (recall that $\tilde X^{\varepsilon,y}  $ is a non degenerate Gaussian process) and \eqref{CTR_DER_MOLLI} in Proposition \ref{PROP_CDER_MC} that for the \textit{main} term in the expansion for all multi-index $\alpha,\ |\alpha|\in \leftB 1,4\rightB $:
\begin{equation}
|D_x^\alpha \tilde p_{\varepsilon}(s,t,x,y)|\le \frac{C}{(t-s)^{|\alpha|/2}}p_c(t-s,y-x),\
|D_y^\alpha \tilde p_{\varepsilon}(s,t,x,y)|\le \frac{C\varepsilon^{-|\alpha|+\gamma}}{(t-s)^{|\alpha|/2}}p_c(t-s,y-x). \label{5}
\end{equation}
Let us now concentrate on the remainder term:
\begin{eqnarray*}
R_{\varepsilon}(s,t,x,y)&:=&\sum_{i\ge 1} \tilde p_{\varepsilon}\otimes H_{\varepsilon}^{(i)}(s,t,x,y)=\tilde p_{\varepsilon} \otimes \Phi_{\varepsilon}(s,t,x,y),
\Phi_{\varepsilon}(s,t,x,y):=\sum_{i\ge 1}H_{\varepsilon}^{(i)}(s,t,x,y).
\end{eqnarray*}
We focus on the first two inequalities in \eqref{CTR_DER_PERT}, the last one can be proved similarly.
The ideas are close to those in \cite{ilin:kala:olei:62}, but we need to adapt them since they considered the ``forward" version of the parametrix expansions.
The key point is that, for H\"older coefficients we have bounded controls for the derivatives of the remainder in the backward variable up to order two. 
It is first easily seen for the first derivatives, since the first order derivation gives an integrable singularity in time in the previous expansions.
Indeed, from \eqref{5} and \eqref{10}, one readily gets the statement if $|\alpha|=1$. The case $|\alpha|\ge 2$ is much more subtle and needs to be discussed thoroughly. Write indeed:
\begin{eqnarray}
D_{x}^\alpha R_{\varepsilon}(s,t,x,y)&=&\lim_{\tau\rightarrow 0} \int_{s+\tau}^{(t+s)/2} du \int_{\R^d} D_x^\alpha \tilde p_{\varepsilon}(s,u,x,z) \Phi_{\varepsilon}(u,t,z,y) dz+\notag\\
&&
\int_{(t+s)/2}^{t} du \int_{\R^d} D_x^\alpha \tilde p_{\varepsilon}(s,u,x,z) \Phi_{\varepsilon}(u,t,z,y) dz\notag\\
&=:&\lim_{\tau \rightarrow 0}D_{x}^\alpha R_{\varepsilon}^\tau(s,t,x,y)+D_x^\alpha R_{\varepsilon}^f(s,t,x,y).\label{RESTES_N}
\end{eqnarray}
The contribution $D_x^\alpha R_{\varepsilon}^f(s,t,x,y) $ does not exhibit time singularities in the integral, since on the considered integration set $u-s\ge \frac12 (t-s) $. 
Let us now recall the usual control on the parametrix kernel under \textbf{(A${}_H $)}, see e.g. Section 2 in \cite{kona:kozh:meno:15:1}. There exist $c,c_1$ s.t. for all $0\le u<t\le T, (z,y)\in (\R^d)^2 $:
\begin{equation}
\left\vert H_\varepsilon(u,t,z,y)\right\vert \leq \frac{c_1(1\vee T^{(1-\gamma)/2})}{%
(t-u)^{1-\gamma /2}}p_{c}(t-u,z-y).  \label{8}
\end{equation}%
Inequality \eqref{8} for $H_{\varepsilon}$ then yields for all $r\in \N^*, 0\le s<t\le T, (x,y)\in (\R^d)^2 $:
\begin{equation}
\label{CTR_H_ITER}
|H_{\varepsilon}^{(r)}(s,t,x,y)|
\leq ((1\vee T^{(1-\gamma)/2})c_1)^{r}\prod_{i=1}^{r-1} B(\frac{\gamma }{2},1+(i-1)\frac{\gamma }{2})
p_{c}(t-s,y-x) (t-s)^{-1+\frac{r \gamma}{2}},
\end{equation}
with the convention $\prod_{i=1}^0=1 $. We thus derive that for all $0\le s<t\le T, (x,y)\in (\R^d)^2 $:
\begin{equation}
\label{BOUND_PHI}
|\Phi_{\varepsilon}(s,t,x,y)|\le \frac{C}{(t-s)^{1-\gamma/2}}p_c(t-s,y-x).
\end{equation}

Thus, from inequalities \eqref{5} and \eqref{BOUND_PHI}:
\begin{equation}
\label{RESTES_FAR}
|D_x^\alpha R_{\varepsilon}^f(s,t,x,y)|\le \frac{C}{(t-s)^{(|\alpha|-\gamma)/2}}p_c(t-s,y-x).
\end{equation}
The delicate contribution is indeed $D_{x}^\alpha R_{\varepsilon}^\tau(s,t,x,y)$ for which we need to be more careful. If $|\alpha|=2 $ we exploit some cancellation properties of the derivatives of the Gaussian kernels.
Recall now that for an arbitrary $w\in \R^d $, setting for $0\le s<u\le T,\ \Sigma_\varepsilon(s,u,w) :=\int_s^u \sigma_\varepsilon\sigma_\varepsilon^*(v,w)dv$,
\begin{eqnarray}
\ \tilde p_{\varepsilon}^w(s,u,x,z)&=&\frac{1}{(2\pi)^{d/2}\det(\Sigma_{\varepsilon}(s,u,w))^{1/2}}
\exp\Bigg(-\frac{1}2\langle \Sigma_{\varepsilon}(s,u,w)^{-1}(z-x),z-x\rangle \Bigg),\notag\\
D_{x_ix_j}^2\tilde p_{\varepsilon}^w(s,u,x,z)&=&\Big\{\big(\Sigma_{\varepsilon}^{-1}(s,u,w)(z-x)\big)_i\big(\Sigma_{\varepsilon}^{-1}(s,u,w)(z-x )\big)_j\notag\\
&& -\delta_{ij}(\Sigma_{\varepsilon}^{-1}(s,u,w))_{ii}\Big\}\tilde p_{\varepsilon}^w(s,u,x,z),\ \forall (i,j)\in\leftB 1,d\rightB^2,
\label{SEC_DER}
\end{eqnarray}
where for $q \in \R^d$, we denote for $i\in \leftB 1,d\rightB $ by $q_i $ its $i^{{\rm th}} $ entry.
Hence, for all multi-index $\alpha,\ |\alpha|=2 $:
\begin{equation}
\label{CENTERING_COND}
\int_{\R^d} D_{x}^\alpha\tilde p_{\varepsilon}^w(s,u,x,z) dz=0. 
\end{equation}
Introducing the centering function $c_{\varepsilon}^\alpha(s,u,x,z):=\left(D_x^\alpha  \tilde p_{\varepsilon}^w(s,u,x,z)\right)|_{w=x}$, we rewrite:
\begin{eqnarray}
D_x^\alpha R_{\varepsilon}^\tau(s,t,x,y)&=&\int_{s+\tau}^{(s+t)/2}du \int_{\R^d}(D_x^\alpha \tilde p_{\varepsilon}-c_{\varepsilon}^\alpha)(s,u,x,z)\Phi_{\varepsilon}(u,t,z,y)dz\nonumber \\
&&+\int_{s+\tau}^{(s+t)/2}du \int_{\R^d}c_{\varepsilon}^\alpha(s,u,x,z)(\Phi_{\varepsilon}(u,t,z,y)-\Phi_{\varepsilon}(u,t,x,y) )dz\nonumber\\
&:=&(R_{\varepsilon}^{\tau,1}+R_{\varepsilon}^{\tau,2})(s,t,x,y),\label{DECOUP_RNEPS}
\end{eqnarray}
exploiting the centering condition \eqref{CENTERING_COND} to introduce the last term of the first equality.
On the one hand, the terms $ D_x^\alpha \tilde p_{\varepsilon}(s,u,x,z), c_{\varepsilon}^\alpha(s,u,x,z)$ only differ in their frozen coefficients (respectively at point $z$ and $x$). Exploiting the H\"older property in space of the mollified coefficients, it is then easily seen that:
\begin{eqnarray*}
|(D_x^\alpha \tilde p_{\varepsilon}-c_{\varepsilon}^\alpha)(s,u,x,z)|&\le& \frac{C|x-z|^\gamma}{(u-s)}p_c(u-s,z-x)
\le \frac{C}{(u-s)^{1-\gamma/2}}p_c(u-s,z-x),
\end{eqnarray*}
yielding an integrable singularity in time so that, from \eqref{BOUND_PHI}:
\begin{equation}
\label{CTR_R1NEPS}
|R_{\varepsilon}^{\tau,1}(s,t,x,y)|\le \frac{C}{(t-s)^{1-\gamma}}p_c(t-s,y-x).
\end{equation}
Let us now control the other contribution. The key idea is now to exploit the \textit{smoothing property}
 of the kernel $\Phi_{\varepsilon} $. Assume indeed that for $A:=\{ z\in \R^d:|x-z| \le c(t-s)^{1/2}\}$ (recall as well that $u\in [s,\frac{s+t}{2}] $) one has:
 \begin{equation}
\label{DIAG_HOLDER_KERNEL}
| \Phi_{\varepsilon}(u,t,x,y)-\Phi_{\varepsilon}(u,t,z,y)|\le C\frac{|x-z|^{\gamma/2}}{(t-u)^{1-\gamma/4}}p_c(t-u,y-z).
 \end{equation} 
Then, we can derive from \eqref{5}, \eqref{DECOUP_RNEPS} and \eqref{DIAG_HOLDER_KERNEL}:
\begin{eqnarray}
\label{LE_CTR_DU_RESTE_EN_EPS_H}
|R_{\varepsilon}^{\tau,2}(s,t,x,y)|
\le C^2 \int_{s+\tau}^{(s+t)/2}du\int_{A}\frac{|x-z|^{\gamma/2}}{(u-s)}p_c(u-s,z-x)\frac{1}{(t-u)^{1-\gamma/4}}p_c(t-u,y-z)dz\nonumber\\
+\frac{C}{(t-s)^{\gamma/4}} \int_{s+\tau}^{(s+t)/2}du\int_{A^C}\frac{|x-z|^{\gamma/2}}{(u-s)}p_c(u-s,z-x)\{|\Phi_{\varepsilon}(u,t,z,y)|+|\Phi_{\varepsilon}(u,t,x,y)|\}dz.
\end{eqnarray}
From \eqref{BOUND_PHI}, we finally get on the considered time set:
\begin{equation*}
\begin{split}
|R_{\varepsilon}^{\tau,2}(s,t,x,y)|&\le C p_c(t-s,y-x)\int_{s+\tau}^{(s+t)/2}du \frac{1}{(u-s)^{1-\gamma/4}}\frac{1}{(t-u)^{1-\gamma/4}}\\
&\le \frac{C}{(t-s)^{1-\gamma/2}}p_c(t-s,y-x),
\end{split}
\end{equation*}
which together with \eqref{CTR_R1NEPS}, \eqref{DECOUP_RNEPS}, \eqref{RESTES_FAR} and \eqref{RESTES_N} gives the statement. It remains to establish \eqref{DIAG_HOLDER_KERNEL}. From the definition of $\Phi_{\varepsilon}$ and the smoothing effect of the kernel $H_{\varepsilon}$ in \eqref{CTR_H_ITER}, it suffices to prove that on the set $\bar A:=\{ z\in\R^d:|x-z|\le c(u'-u)^{1/2}\}$:
\begin{equation}
\label{HOLD_KER}
|H_{\varepsilon}(u,u',x,w)-H_{\varepsilon}(u,u',z,w)|\le C\frac{|x-z|^{\gamma/2}}{(u'-u)^{1-\gamma/4}}p_c(u'-u,w-z), 
\end{equation}
for $u'\in (u,t],\ u\in [s,(s+t)/2]$. Observe that $\bar A\subset A $. Indeed, recalling that we want to establish \eqref{DIAG_HOLDER_KERNEL} on $A$ if $z\not\in \bar A $, we get from \eqref{CTR_H_ITER}:
\begin{eqnarray*}
\int_u^t du' \int_{\bar A^c} |H_{\varepsilon}(u,u',x,w)-H_{\varepsilon}(u,u',z,w)| |(\sum_{i\ge 2} H_{\varepsilon}^{(i)} )(u',t,w,y)| dw\\
\le \int_u^t du' \int_{\bar A^c} \frac{C}{(u'-u)^{1-\gamma/2}} (p_c(u'-u,w-x)+ p_c(u'-u,w-z))\\
\times \frac{|x-z|^{\gamma/2}}{(u'-u)^{\gamma/4}}\frac{C}{(t-u')^{1-\gamma}}p_c(t-u',y-w)dw
\le C \frac{|x-z|^{\gamma/2}}{(t-u)^{1-5\gamma/4}}p_c(t-u,y-z)
\le C \frac{|x-z|^{\gamma/2}}{(t-u)^{1-\gamma/4}}p_c(t-u,y-z),
\end{eqnarray*}
exploiting that $z\in A$, $t-u\ge \frac12(t-s) $, 
and the usual convexity inequality $\frac{|y-x|^2}{t-u}\ge \frac{|y-z|^2}{2(t-u)}-\frac{|z-x|^2}{t-u}\ge \frac{|y-z|^2}{2(t-u)}-2c^2  $ for the last but one inequality. 
On the other hand, on $\bar A $ we get \eqref{DIAG_HOLDER_KERNEL} from \eqref{HOLD_KER} and \eqref{CTR_H_ITER}.

Let us turn to the proof of \eqref{HOLD_KER}. We concentrate on the second derivatives in $H_{\varepsilon}$ which yield the most singular contributions:
\begin{eqnarray}
&&\Tr((a_{\varepsilon}(u,x)-a_{\varepsilon}(u,w))D_x^2 \tilde p_{\varepsilon}(u,u',x,w))-\Tr((a_{\varepsilon}(u,z)-a_{\varepsilon}(u,w))D_x^2 \tilde p_{\varepsilon}(u,u',z,w))\nonumber\\
&=&\Tr((a_{\varepsilon}(u,x)-a_{\varepsilon}(u,z))D_x^2 \tilde p_{\varepsilon}(u,u',x,w))
-\Tr((a_{\varepsilon}(u,z)-a_{\varepsilon}(u,w))(D_x^2 \tilde p_{\varepsilon}(u,u',z,w)-D_x^2 \tilde p_{\varepsilon}(u,u',x,w)))\nonumber\\
&=:&I+II.\label{DECOUP_E_EPS}
\end{eqnarray}
Then, from \eqref{5},
\begin{eqnarray}
|I|\le C\frac{|x-z|^\gamma}{(u-u')}p_c(u'-u,w-x)\le \frac{C|x-z|^{\gamma/2}}{(u-u')^{1-\gamma/4}}p_c(u'-u,w-x)
\le 
\frac{C|x-z|^{\gamma/2}}{(u-u')^{1-\gamma/4}}p_c(u'-u,w-z), \label{DECOUP_E_EPS_I}
\end{eqnarray}
using that $z\in \bar A $ for the second inequality, again combined with the  convexity inequality $\frac{|x-w|^2}{u'-u}\ge \frac{|z-w|^2}{2(u'-u)}-\frac{|x-z|^2}{u'-u}\ge  \frac{|z-w|^2}{2(u'-u)}-c^2$ for the last one. Now, from the explicit expression of the second order derivatives in \eqref{SEC_DER}, \A{A2} and usual computations we also derive:
\begin{eqnarray}
|II|\le \frac{C|z-w|^\gamma}{(u'-u)}\frac{|z-x|^{\gamma/2}}{(u'-u)^{\gamma/4}}p_c(u'-u,w-z)\le\frac{C|z-x|^{\gamma/2}}{(u'-u)^{1-\gamma/4}}p_c(u'-u,w-z).\label{DECOUP_E_EPS_II}
\end{eqnarray}

This gives \eqref{HOLD_KER} and completes the proof for $|\alpha| \le 2 $.

Let us now turn to $|\alpha|\ge 3 $. In those cases, the singularities induced by the derivatives are not integrable in short time, even if we exploit cancellations. We are thus led to perform integration by parts, deteriorating the bounds since these operations make the derivatives of the mollified coefficients appear.

Recalling $\alpha \in \N^d$, denote by $l $ a multi-index s.t. $|l|=2$ and $\alpha-l\ge 0 $ (where the inequality is to be understood componentwise)
. From equations \eqref{RESTES_N}, \eqref{RESTES_FAR}, we only have to consider the contribution $D_{x}^\alpha R_{\varepsilon}^\tau(s,t,x,y)$. Write:
\begin{eqnarray}
D_x^\alpha R_{\varepsilon}^\tau(s,t,x,y)
&=&D_{x}^{\alpha-l}\int_{s+\tau}^{(t+s)/2} du \int_{\R^d}D_x^l \tilde p_{\varepsilon}(s,u,x,z)\Phi_{\varepsilon}(u,t,z,y)dz\notag\\
 &=&D_{x}^{\alpha-l}\int_{s+\tau}^{(t+s)/2} du \int_{\R^d} g^{l,\varepsilon}(s,u,x,z)\Phi_{\varepsilon}(u,t,z,y)dz,
\label{decomp}
\end{eqnarray}
where $g^{l,\varepsilon}(s,u,x,z):=D_x^l \tilde p_{\varepsilon}(s,u,x,z)$.
Let us write introducing the cancellation term $c_{\varepsilon}^l $ introduced after \eqref{CENTERING_COND}:
\begin{eqnarray}
D_x^\alpha R_{\varepsilon}^\tau(s,t,x,y)=D_x^{\alpha -l}\int_{s+\tau}^{(s+t)/2} du \int_{\R^d} (g^{l,\varepsilon}- c_{\varepsilon}^l)(s,u,x,z)\Phi_{\varepsilon}(u,t,z,y) dz\notag \\
+D_x^{\alpha -l}\int_{s+\tau}^{(s+t)/2} du \int_{\R^d} c_{\varepsilon}^l(s,u,x,z)(\Phi_{\varepsilon}(u,t,z,y)-\Phi_{\varepsilon}(u,t,x,y)  )dz\notag \\
=D_x^{\alpha -l}\int_{s+\tau}^{(s+t)/2} du \int_{\R^d} (g^{l,\varepsilon}-c_{\varepsilon}^l)(s,u,x,x+z)\Phi_{\varepsilon}(u,t,x+z,y) dz\notag\\
+D_x^{\alpha -l}\int_{s+\tau}^{(s+t)/2} du \int_{\R^d} c_{\varepsilon}^l(s,u,x,x+z)(\Phi_{\varepsilon}(u,t,x+z,y)-\Phi_{\varepsilon}(u,t,x,y)  )dz.\notag\\ \label{DEF_R_C}
\end{eqnarray}
The purpose of that change of variable, already performed in \cite{kona:mamm:02}, is that we get integrable time singularities in the contributions $D_x^{\alpha -l} (g^{l,\varepsilon}-c_{\varepsilon}^l)(s,u,x,x+z)$. Anyhow, the mollified coefficients $b_{\varepsilon},\sigma_{\varepsilon} $ have explosive derivatives. From the definition of $g^{l,\varepsilon} $ and \eqref{CTR_DER_MOLLI}  one easily gets that there exists $c,C$ s.t. for all $\alpha ,\ |\alpha|\le 4$:
\begin{equation}
\label{CTR_DER_VARSIGMA}
\begin{split}
|D_x^{\alpha -l} (g^{l,\varepsilon}-c_{\varepsilon}^l)(s,u,x,x+z)|\le \frac{C\varepsilon^{-|\alpha-l|}}{(u-s)^{1-\gamma/2}}p_c(u-s,z),\\
|D_x^{\alpha -l} c_{\varepsilon}^l(s,u,x,x+z)|\le \frac{C\varepsilon^{-|\alpha-l|+\gamma}}{(u-s)}p_c(u-s,z).
\end{split}
\end{equation}
From \eqref{DEF_R_C} and \eqref{CTR_DER_VARSIGMA} it thus
remains to control the terms $D_x^{\alpha-l}\Phi_{\varepsilon}(u,t,z+x,y), D_x^{\alpha-l}(\Phi_{\varepsilon}(u,t,x+z,y)-\Phi_{\varepsilon}(u,t,x,y)) $ which are the most singular ones in $ D_x^\alpha R_{\varepsilon}^\tau(s,t,x,y)$. To this end, we will establish by induction that the following control holds:
\begin{eqnarray}
\label{CTR_IND}
\exists c,C,\ \forall 0\le s<t\le T,\ (x,y)\in (\R^d)^2,\ \forall \beta,\ |\beta|\le 3,\
|D_x^{\beta}H_{\varepsilon}^{(i)}(s,t,x,y)|\le\notag \\
\frac{C^{i}{\varepsilon}^{-|\beta|}}{(t-s)^{|\beta|/2}}(t-s)^{-1+i\gamma/2}\prod_{j=1}^{i-1}B(\frac \gamma 2,1+ (j-1)\frac \gamma 2) p_c(t-s,y-x),
\end{eqnarray}
with the convention that $\prod_{j=1}^0 =1 $.
Observe first that for $|\beta|=0 $ (no derivation), estimate \eqref{CTR_IND} readily follows from \eqref{CTR_H_ITER}.
Let us now suppose $|\beta|>0$. Observe from the definition of $H_{\varepsilon}$ that \eqref{CTR_IND} is satisfied for $i=1 $. Let us assume it holds for a given $i$ and let us prove it for $i+1$.
Write again:
\begin{eqnarray*}
D_x^\beta H_{\varepsilon}^{(i+1)}(s,t,x,y)=\int_{(s+t)/2}^t du \int_{\R^d} D_x^\beta H_{\varepsilon}(s,u,x,z) H_{\varepsilon}^{(i)}(u,t,z,y)dz\\
+D_x^\beta \int_s^{(s+t)/2} du \int_{\R^d} H_{\varepsilon}(s,u,x,x+z)H_{\varepsilon}^{(i)}(u,t,x+z,y)dz=:(R_1^{i,\beta}+R_2^{i,\beta})(s,t,x,y).
\end{eqnarray*}
The term $R_1^{i,\beta}$ is easily controlled by \eqref{CTR_IND} for $\beta=0 $ and the induction hypothesis. Observe also that, from Proposition \ref{PROP_CDER_MC} one derives similarly to \eqref{CTR_DER_VARSIGMA} that:
\begin{eqnarray*}
|D_x^\beta H_{\varepsilon}(s,u,x,x+z)|\le \frac{C\varepsilon^{-|\beta|}}{(u-s)^{1-\gamma/2}}p_c(u-s,z).
\end{eqnarray*}
Together with the induction hypothesis and the Leibniz rule for differentiation, this allows to control $R_2^{i,\beta}$. The controls on $\{R_j^{i,\beta}\}_{j\in \{1,2\}} $ give \eqref{CTR_IND} for $i+1$. We eventually derive (reminding that $|l|=2 $):
\begin{equation}
\label{CTR_DER_NOY_M}
|D_x^{\alpha-l} \Phi_{\varepsilon}(u,t,x+z,y)|\le \frac{C}{(t-u)^{(|\alpha|-2)/2}}\frac{\varepsilon^{-|\alpha|+2}}{(t-u)^{1-\gamma/2}}p_c(t-u,y-(x+z)).
\end{equation}
The spatial H\"older continuity of the derivatives of the kernel $\Phi_{\varepsilon} $ could be checked following the previous steps performed respectively to get the spatial H\"older continuity 
of the kernel and the controls on its derivatives. One gets, on $|z|\le c(t-u)^{1/2} $:
\begin{equation*}
\begin{split}
|D_x^{\alpha-l} \Phi_{\varepsilon}(u,t,x+z,y)-D_x^{\alpha-l} \Phi_{\varepsilon}(u,t,x,y)|
\le \frac{C|z|^{\gamma/2}}{(t-u)^{(|\alpha|-2)/2}}\frac{\varepsilon^{-|\alpha|+2}}{(t-u)^{1-\gamma/4}}p_c(t-u,y-(x+z)),
\end{split}
\end{equation*}
which together with \eqref{CTR_DER_NOY_M}, \eqref{CTR_DER_VARSIGMA}, \eqref{DEF_R_C} gives (proceeding as above for $|z|\ge c(t-u)^{1/2} $):
\begin{eqnarray*}
|D_x^\alpha R_{\varepsilon}^\tau(s,t,x,y)|\le \frac{C\varepsilon^{-|\alpha|+2}(t-s)^{\gamma/2}}{(t-s)^{|\alpha|/2}}p_c(t-s,y-x).
\end{eqnarray*}
The second equation of \eqref{CTR_DER_PERT} follows for $\bar p_{\varepsilon}=p_{\varepsilon} $ from 
the above control and \eqref{RESTES_FAR}, \eqref{RESTES_N}.
Observe that the control for the derivative w.r.t. $y$ has additional singularity in $\varepsilon$. This is clear since we directly differentiate the frozen mollified coefficients.
Now the statements readily hold for $p_{\varepsilon}^d$, since the \textit{integration in time} played no role in the previous computations. For $p_{\varepsilon}^h $, the only point that should be totally justified is the smoothing property and H\"older continuity of the discrete Kernel $\Phi_{\varepsilon}^{h}(t_i,t_j,x,y):=\sum_{r=1}^{j-i}H_{\varepsilon}^{h,(r)} (t_i,t_j,x,y)$. The smoothing property, equivalent of \eqref{CTR_IND}, has been investigated in \cite{lema:meno:10}. The spatial H\"older continuity can be derived as above.\\

\subsubsection{Proof under \A{A${}_{PS}$}}
Let us now turn to the proof of the heat kernel bounds for $p_\varepsilon$ under \A{A${}_{PS}$}, which almost follows the same lines. Observe first that the result for $|\alpha|=1 $ still follows from \eqref{5} and \eqref{10}. The key point is again that the derivative of the Gaussian kernel yields an integrable singularity. For $|\alpha|=2 $, we still separate the contribution $R_\varepsilon(s,t,x,y) $ as in \eqref{RESTES_N} and again focus on  $\lim_{\tau\rightarrow 0} D_x^\alpha R_\varepsilon^\tau(s,t,x,y) $ which is the only term yielding a potential singularity. With then notations of \eqref{DECOUP_RNEPS}, it is sufficient to investigate $R_\varepsilon^{\tau,2}(s,t,x,y) $. Indeed, under \A{A${}_{PS}$}, equation \eqref{CTR_R1NEPS} actually holds with $\gamma=1 $. We recall that to control $R_\varepsilon^{\tau,2}(s,t,x,y) $, the key estimate was \eqref{DIAG_HOLDER_KERNEL}. We aim at proving the different control, for all $u\in [s,t) $, for all $\eta\in (0,1] $:
\begin{equation}
\label{DIAG_HOLDER_KERNEL_BE}
| \Phi_{\varepsilon}(u,t,x,y)-\Phi_{\varepsilon}(u,t,z,y)|\le C
\varepsilon 
^{-\eta}\frac{|x-z|^{\eta}}{(t-u)^{3/4}}p_c(t-u,y-z),
 \end{equation} 
on $A:=\{ z\in \R^d: |x-z|\le (t-s)^{1/2}\wedge \varepsilon\} $.
Then, we can derive from \eqref{5}, \eqref{DECOUP_RNEPS} and \eqref{DIAG_HOLDER_KERNEL_BE}:
\begin{eqnarray*}
|R_{\varepsilon}^{\tau,2}(s,t,x,y)|
\le C 
\varepsilon
^{-\eta} \int_{s+\tau}^{(s+t)/2}du\int_{A}\frac{|x-z|^{\eta}}{(u-s)}p_c(u-s,z-x)\frac{1}{(t-u)^{3/4}}p_c(t-u,y-z)dz\\
+C((t-s)^{1/2}\wedge \varepsilon)^{-\eta} \int_{s+\tau}^{(s+t)/2}du\int_{A^C}\frac{|x-z|^{\eta}}{(u-s)}p_c(u-s,z-x)\{|\Phi_{\varepsilon}(u,t,z,y)|+|\Phi_{\varepsilon}(u,t,x,y)|\}dz.
\end{eqnarray*}
Since the drift $b_\varepsilon $ is uniformly bounded, uniformly in $\varepsilon\in [0,1] $, we have under \A{A${}_{PS}$} the following usual control on the parametrix kernel (see e.g. Section 2 in \cite{kona:kozh:meno:15:1}):
\begin{equation}
\left\vert H_\varepsilon(u,t,z,y)\right\vert \leq \frac{c_1}{%
(t-u)^{1/2}}p_{c}(t-u,y-z).  \label{8_PS}
\end{equation}%
Equation \eqref{8_PS} for $H_{\varepsilon}$ then yields 
\begin{equation}
\label{CTR_H_ITER_PS}
|H_{\varepsilon}^{(r)}(s,t,x,y)|
\leq c_1^{r}\prod_{i=1}^{r-1} B(\frac{1 }{2},1+(i-1)\frac{1 }{2})
p_{c}(t-s,y-x) (t-s)^{-1+\frac{r }{2}},
\end{equation}
again with the convention $\prod_{i=1}^0=1 $. We thus derive $|\Phi_{\varepsilon}(u,t,z,y)|\le \frac{C}{(t-u)^{1/2}}p_c(t-u,y-z) $ and $ |\Phi_{\varepsilon}(u,t,x,y)|\le \frac{C}{(t-u)^{1/2}}p_c(t-u,y-x) $. We finally get on the considered time set:
\begin{equation*}
\begin{split}
|R_{\varepsilon}^{\tau,2}(s,t,x,y)|&\le C((t-s)^{1/2} \wedge \varepsilon)^{-\eta} p_c(t-s,y-x)\int_{s+\tau}^{(s+t)/2}du \frac{1}{(u-s)^{1-\eta/2}}\frac{1}{(t-u)^{3/4}}\\
&\le \frac{C((t-s)^{1/2}\wedge \varepsilon)^{-\eta}}{\eta (t-s)^{3/4-\eta/2}}p_c(t-s,y-x).
\end{split}
\end{equation*}
It remains to establish \eqref{DIAG_HOLDER_KERNEL_BE}. From the definition of $\Phi_{\varepsilon}$ and the smoothing effect of the kernel $H_{\varepsilon}$ in \eqref{CTR_H_ITER_PS}, it suffices to prove that on $\bar A:=\{ z\in \R^d: |x-z|\le c[ (u'-u)^{1/2}\wedge \varepsilon]\} $:
\begin{equation}
\label{HOLD_KER_PS}
|H_{\varepsilon}(u,u',x,w)-H_{\varepsilon}(u,u',z,w)|\le C\{ \frac{|x-z|^{\eta}}{(u'-u)^{3/4}}((u'-u)^{1/2}\wedge \varepsilon)^{-\eta} \}p_c(u'-u,w-z), 
\end{equation}
for $u'\in (u,t],\ u\in [s,(s+t)/2]$. The contributions associated with $z\in \bar A^C $ can be handled as above. To establish the above control we focus on the first order terms involving the regularized coefficient with initial discontinuities. Indeed the second order contribution can be analyzed as in \eqref{DECOUP_E_EPS}, \eqref{DECOUP_E_EPS_I}, \eqref{DECOUP_E_EPS_II}, taking $\gamma=1 $ in those expressions. In particular, the time singularity in $(u-u')^{3/4} $ in \eqref{HOLD_KER_PS}  precisely comes from those terms. Recalling that under \A{A${}_{PS}$} the driftless proxy does not depend on $\varepsilon $ (since the diffusion is smooth, see \eqref{SCHEMES_AND_PARAM} in which one has $\sigma_\varepsilon=\sigma $ under \A{A${}_{PS}$}), we denote its density by $\tilde p $ and write:
\begin{eqnarray*}
\langle b_\varepsilon(u,x),D_x \tilde p(u,u',x,w)\rangle-\langle b_\varepsilon(u,z),D_x\tilde p(u,u',z,w)\rangle\\
=\langle b_\varepsilon(u,x)-b_\varepsilon(u,z),D_x\tilde p(u,u',x,w)\rangle+\langle b_\varepsilon(u,z),D_x\tilde p(u,u',x,w)-D_x\tilde p(u,u',z,w) \rangle:=I+II.
\end{eqnarray*}
On the one hand, from the mean value theorem and recalling that $|D_xb_\varepsilon|_\infty\le C\varepsilon^{-1}\le C ((u'-u)^{1/2}\wedge \varepsilon)^{-1} $  we get:
\begin{eqnarray*}
|I|&\le& \frac{C}{(u'-u)^{1/2}} \Big\{ 2|b|_\infty\Big(\frac{|x-z|}{(u'-u)^{1/2}\wedge \varepsilon}\Big)^\eta\I_{|x-z|>(u'-u)^{1/2}\wedge \varepsilon} +\varepsilon^{-1}|x-z|\I_{|x-z|\le (u'-u)^{1/2} \wedge \varepsilon}\Big\}  p_c(u'-u,w-x) \\
 &\le& C ((u'-u)^{1/2}\wedge \varepsilon)^{-\eta}\frac{|x-z|^\eta}{(u'-u)^{1/2}}p_c(u'-u,w-x)\le C ((u'-u)^{1/2}\wedge \varepsilon)^{-\eta}\frac{|x-z|^\eta}{(u'-u)^{1/2}}p_c(u'-u,w-z),
\end{eqnarray*}
using again a convexity inequality for the last control, recalling that $z\in \bar A$. On the other hand still from the mean value Theorem and usual controls on the derivatives of the Gaussian density:
\begin{eqnarray*}
|II|&\le& \frac{C|x-z|}{(u'-u)}\int_0^1 p_c(u'-u,w-\{z+\lambda (x-z)\})d\lambda\le \frac{C|x-z|^{\eta}}{(u'-u)^{(1+\eta)/2}}p_c(u'-u,w-z) \\
&\le& \frac{C|x-z|^{\eta}}{(u'-u)^{1/2}((u'-u)^{1/2}\wedge \varepsilon)^{\eta}}p_c(u'-u,w-z) .
\end{eqnarray*}
The above estimates give \eqref{HOLD_KER_PS} and concludes the proof for  $|\alpha|=2 $. \\

Let us turn to $|\alpha|\ge 3 $. The idea is again to proceed 	as under \textbf{(A${}_H$)}, up to a suitable modification of the key estimate \eqref{CTR_IND} which can now be localized and becomes for all $q>d$:
\begin{eqnarray}
\label{CTR_IND_BE}
\exists c,C,\ \forall 0\le s<t\le T,\ (x,y)\in (\R^d)^2,\ \forall \beta,\ |\beta|\le 3,\
|D_x^{\beta}H_{\varepsilon}^{(i)}(s,t,x,y)|\le\notag \\
\frac{C^{i}{(\varepsilon}^{-|\beta|}\I_{x\in V_\varepsilon({\mathcal I})}+\varepsilon^{-|\beta|+1/q})}{(t-s)^{|\beta|/2}}(t-s)^{-1+i\alpha(q)}\prod_{j=1}^{i-1}B(\alpha(q), \alpha(q) j) p_c(t-s,y-x),\ \alpha(q)=\frac 12(1-\frac dq),
\end{eqnarray}
with $\prod_{j=1}^0=1 $.
We again proceed by induction. Observe first that for $|\beta|=0 $ (no derivation), estimate \eqref{CTR_IND_BE} readily follows from \eqref{8_PS}.
Let us now suppose $|\beta|>0$. Observe  as well from the definition of $H_{\varepsilon}$ that \eqref{CTR_IND_BE} is satisfied for $i=1 $. Let us assume it holds for a given $i$ and let us prove it for $i+1$.
Write again:
\begin{eqnarray*}
D_x^\beta H_{\varepsilon}^{(i+1)}(s,t,x,y)=\int_{(s+t)/2}^t du \int_{\R^d} D_x^\beta H_{\varepsilon}(s,u,x,z) H_{\varepsilon}^{(i)}(u,t,z,y)dz\\
+D_x^\beta \int_s^{(s+t)/2} du \int_{\R^d} H_{\varepsilon}(s,u,x,x+z)H_{\varepsilon}^{(i)}(u,t,x+z,y)dz=:(R_1^{i,\beta}+R_2^{i,\beta})(s,t,x,y).
\end{eqnarray*}
The term $R_1^{i,\beta}$ is easily controlled by \eqref{CTR_IND_BE} that holds from the induction hypothesis for $i_0=1$ (direct differentiation of $H_\varepsilon$) and $\beta=0 $ for the considered $i$ (no differentiation of $H_{\varepsilon}^{(i)} $). Observe also that, similarly to \eqref{CTR_DER_VARSIGMA}, one has:
\begin{eqnarray}
\label{CTR_DER_VARSIGMA_BE}
|D_x^\beta H_{\varepsilon}(s,u,x,x+z)|\le \frac{C(\varepsilon^{-|\beta|}\I_{x\in V_\varepsilon({\mathcal I})}+1) }{(u-s)^{1/2}}p_c(u-s,z).
\end{eqnarray}
Now, from the Leibniz rule for differentiation, \eqref{CTR_DER_VARSIGMA_BE} and the induction hypothesis, we have:
\begin{eqnarray}
|R_2^{i,\beta}(s,t,x,y)|\le C^{i+1} \prod_{j=1}^{i-1}B(\alpha(q), \alpha(q) j)\Big\{ \sum_{\tilde \beta, |\tilde \beta|\le |\beta|}
\left(\begin{array}{c} |\beta|\\|\tilde \beta|\end{array}\right)
\int_{s}^{(s+t)/2} \frac{du (t-u)^{-1+i\alpha(q)}}{(u-s)^{1/2}(t-u)^{(|\beta|-|\tilde \beta|)/2}}\nonumber\\
\times \int_{\R^d} p_c(u-s,z)(\varepsilon^{-|\tilde \beta|}\I_{x\in V_\varepsilon({\mathcal I})}+1)(\varepsilon^{-(|\beta|-|\tilde \beta|)}\I_{x+z\in V_\varepsilon({\mathcal I})}+\varepsilon^{-(|\beta|-|\tilde \beta|)+1/q}) p_c(t-u,y-x-z)dz\Big\}\nonumber\\
\le C^{i+1} \prod_{j=1}^{i-1}B(\alpha(q), \alpha(q) j)\Big\{ \sum_{\tilde \beta, |\tilde \beta|\le |\beta|}
\left(\begin{array}{c} |\beta|\\|\tilde \beta|\end{array}\right)
\int_{s}^{(s+t)/2} \frac{du (t-u)^{-1+i\alpha(q)}}{(u-s)^{1/2}(t-u)^{(|\beta|-|\tilde \beta|)/2}}\nonumber\\
\times [p_c(t-s,y-x) (\varepsilon^{-|\beta|}\I_{x\in V_\varepsilon({\mathcal I})}+\varepsilon^{-|\beta|+1/q})+\varepsilon^{-|\beta|+1/q} (\int_{\R^d} p_c(u-s,z)^{\bar q}p_c(t-u,y-x-z)^{\bar q}dz)^{1/\bar q}\Big\},\label{HOLDER_REGU}
\end{eqnarray}
denoting by $\bar q>1$ the conjugate of $q$, $q^{-1}+\bar q^{-1}=1 $ (see also Section \ref{SEC_PROOF_AS} for similar arguments). Recall now that:
\begin{eqnarray*}
(\int_{\R^d} p_c(u-s,z)^{\bar q}p_c(t-u,y-x-z)^{\bar q}dz)^{1/\bar q}&=& \Big(\frac{c(t-s)}{(2\pi)(u-s)(t-u)}\Big)^{d/(2q)
}\bar q^{-d/(2\bar q)}  p_c(t-s,y-x)\\ 
&\le & C (u-s)^{-d/(2q)}p_c(t-s, y-x),
\end{eqnarray*}
for $u\in [s,(s+t)/2] $.
Hence, 
\begin{eqnarray*}
|R_2^{i,\beta}(s,t,x,y)|\le \frac{C^{i+1}}{(t-s)^{|\beta|/2}} \prod_{j=1}^{i-1}B(\alpha(q), \alpha(q) j)\Big\{ \int_{s}^{(s+t)/2} \frac{du (t-u)^{-1+i\alpha(q)}}{(u-s)^{1/2(1+d/q)}}\Big\}\\  
\times p_c(t-s,y-x) (\varepsilon^{-|\beta|}\I_{x\in V_\varepsilon({\mathcal I})}+\varepsilon^{-|\beta|+1/q})\\
\le \frac{C^{i+1}}{(t-s)^{|\beta|/2}} \prod_{j=1}^{i-1}B(\alpha(q), \alpha(q) j)(t-s)^{-1+(i+1)\alpha(q)}\int_{0}^{1/2}(1-u)^{-1+i\alpha(q)}u^{-1+\alpha(q)}du \\
\times p_c(t-s,y-x) (\varepsilon^{-|\beta|}\I_{x\in V_\varepsilon({\mathcal I})}+\varepsilon^{-|\beta|+1/q}).
\end{eqnarray*}

The controls on $\{R_j^{i,\beta}\}_{j\in \{1,2\}} $ give \eqref{CTR_IND_BE} for $i+1$.

Estimate \eqref{CTR_IND_BE} yields for every multi-index $l,\ |l|=2$:
\begin{equation}
\label{CTR_DER_NOY}
|D_x^{\alpha-l} \Phi_{\varepsilon}(u,t,x+z,y)|\le \frac{C_q}{(t-u)^{(|\alpha|-2)/2}}\frac{\I_{x+z\in V_\varepsilon({\mathcal I})}\varepsilon^{-|\alpha|+2}+\varepsilon^{-|\alpha|+2+1/q}}{(t-u)^{1-\alpha(q)}}p_c(t-u,y-(x+z)).
\end{equation}
The spatial H\"older continuity of the derivatives of the kernel $\Phi_{\varepsilon} $ could be checked following the previous steps performed respectively to get the spatial H\"older continuity 
of the kernel and the controls on its derivatives. One gets, on $|z|\le c\{(t-u)^{1/2}\wedge \varepsilon\} $ for all $\eta\in (0,1] $:
\begin{equation}
\label{DIFF_PHI_PS}
\begin{split}
|D_x^{\alpha-l} \Phi_{\varepsilon}(u,t,x+z,y)-D_x^{\alpha-l} \Phi_{\varepsilon}(u,t,x,y)|\\
\le \frac{C_q
\varepsilon
^{-\eta}|z|^{\eta}}{(t-u)^{(|\alpha|-2)/2}}\frac{\varepsilon^{-|\alpha|+2}\I_{x+z\in V_\varepsilon( {\mathcal I} ) }+\varepsilon^{-|\alpha|+2+1/q}}{(t-u)^{1-\alpha(q)+\eta/2}}p_c(t-u,y-(x+z)).
\end{split}
\end{equation}
Now, equation \eqref{DEF_R_C} still holds under \A{A${}_{PS} $}, with $g^{l,\varepsilon}=g^{l}, c_{\varepsilon}^l=c^l $, i.e. the driftless proxy does not depend on $\varepsilon $. Also, the smoothness assumption on $\sigma $ allows to improve \eqref{CTR_DER_VARSIGMA}. Precisely, there exist $c,C$ s.t. for all $\alpha ,\ |\alpha|\le 4 $:
\begin{equation*}
|D_x^{\alpha -l} (g^{l}-c^l)(s,u,x,x+z)|\le \frac{C}{(u-s)^{1/2}}p_c(u-s,z),\
|D_x^{\alpha -l} c^l(s,u,x,x+z)|\le \frac{C}{(u-s)}p_c(u-s,z),
\end{equation*}
which together with \eqref{CTR_DER_NOY}, \eqref{DIFF_PHI_PS},  \eqref{DEF_R_C} and choosing $\alpha(q)>\eta $ gives (proceeding as above for $|z|\ge c\{(t-u)^{1/2}\wedge \varepsilon\} $):
\begin{eqnarray*}
|D_x^\alpha R_{\varepsilon}^\tau(s,t,x,y)|\le \frac{C_{\eta,q}\varepsilon^{-(\eta+|\alpha|)+2+1/q}(t-s)^{\eta/2}}{(t-s)^{|\alpha|/2}}p_c(t-s,y-x).
\end{eqnarray*}
The controls on the derivatives w.r.t. to the forward variables are derived similarly. We here simply illustrate on the first term $\tilde p_\varepsilon\otimes H^\varepsilon(s,t,x,y) $ of the parametrix series how the derivatives must be handled. The stated controls would follow from inductions similar to the previous ones.
Write for a given multi-index $\beta $:
\begin{eqnarray*} 
D_y^\beta \Big(\tilde p_\varepsilon\otimes H_\varepsilon(s,t,x,y)\Big)\\
=\int_s^{(s+t)/2}du \int_{\R^d}\tilde p(s,u,x,z) D_y^\beta \{ \langle b_\varepsilon(u,z),D_z\tilde p(u,t,z,y)\rangle+\frac{1}{2}{\rm Tr}\{(a(u,z)-a(u,y))D_z^2\tilde p(u,t,z,y) \}\} dz+\\
\lim_{\tau\downarrow 0}\int_{(s+t)/2}^{t-\tau}du\int_{\R^d}D_y^\beta \Big(\tilde p(s,u,x,z)[\langle b_\varepsilon(u,z),D_z\tilde p(u,t,z,y)\rangle\\
+\frac{1}{2}{\rm Tr}\{(a(u,z)-a(u,y))D_z^2\tilde p(u,t,z,y) \}] \Big) dz:=(D_1^\beta+D_2^\beta)(s,t,x,y).
\end{eqnarray*} 
We readily get from the controls of \eqref{5} that:
\begin{equation}
\label{CTR_D1_B}
 |D_1^\beta(s,t,x,y)|\le \frac{C}{(t-s)^{(|\beta|-1)/2}} p_c(t-s,y-x),
 \end{equation}
which is the expected control. Since $a $ is smooth the terms involving the second derivatives w.r.t. $z$ in $D_2^\beta$ can be handled performing the change of variables $z'=z+y$ as above (see also \cite{kona:mamm:02} under the current smoothness assumption on the diffusion coefficient). Let us thus focus on the contribution:
\begin{eqnarray*}
D_{21}^\beta(s,t,x,y):=\lim_{\tau\downarrow 0}\int_{(s+t)/2}^{t-\tau}du\int_{\R^d}D_y^\beta \Big(\tilde p(s,u,x,z)\langle b_\varepsilon(u,z),D_z\tilde p(u,t,z,y)\rangle\Big) dz.
\end{eqnarray*}
Consider first the case $|\beta|=1$. Write:
\begin{eqnarray}
D_{21}^{\beta,\tau}(s,t,x,y)&:=&\int_{(s+t)/2}^{t-\tau}du\int_{\R^d}D_y^\beta \Big(\tilde p(s,u,x,z)\langle b_\varepsilon(u,z),D_z\tilde p(u,t,z,y)\rangle \Big) dz\notag\\
&=&\int_{(s+t)/2}^{t-\tau}du\int_{\R^d}\tilde p(s,u,x,z) \langle b_\varepsilon(u,z),D_y^\beta D_z\tilde p(u,t,z,y) \rangle dz\notag\\
&=&\int_{(s+t)/2}^{t-\tau} du\int_{\R^d}[\tilde p(s,u,x,z)-\tilde p(s,u,x,y)] \langle b_\varepsilon(u,z),D_y^\beta D_z\tilde p(u,t,z,y) \rangle dz\notag\\
&&+\int_{(s+t)/2}^{t-\tau} du \int_{\R^d}\tilde p(s,u,x,y) \langle b_\varepsilon(u,z)-b_\varepsilon(u,y),D_y^\beta D_z\tilde p(u,t,z,y) \rangle dz\notag\\
&=:&[D_{211}^{\beta,\tau}+D_{212}^{\beta,\tau}](s,t,x,y),\label{DEC_D_21}
\end{eqnarray}
recalling that for all $y\in \R^d $, $\int_{\R^d} D_z \tilde p(u,t,z,y)dz=0$, so that $D_y^\beta \int_{\R^d} D_z \tilde p(u,t,z,y)dz=0 $, for the last but one equality.
Still from the controls of \eqref{5}, we readily get:
\begin {eqnarray*}
|D_{211}^{\beta,\tau}|\le \frac{C}{(t-s)^{1/2}}\int_{(s+t)/2}^{t-\tau} du\int_{\R^d}
|z-y|\Big\{\int_0^1 p_c(u-s,y-x+\lambda(z-y)) d\lambda \I_{|z-y|\le (t-s)^{1/2}}\\
+(p_c(u-s,z-x)+p_c(u-s,y-x) )\I_{|z-y|>(t-s)^{1/2}} \Big\}\frac{1}{(t-u)}p_c(t-u,y-z)dz
\le Cp_c(t-s,y-x).
\end{eqnarray*}
On the other hand:
\begin{eqnarray*}
|D_{212}^{\beta,\tau}](s,t,x,y)| \le C p_c(t-s,y-x)\int_{(s+t)/2}^{t-\tau} du \int_{\R^d} \{\I_{|z-y|\le \varepsilon} \frac{|z-y|}{\varepsilon}+\I_{|z-y|>\varepsilon}\frac{|z-y|^\eta}{\varepsilon^\eta} \} \frac{1}{(t-u)}p_c(t-u,y-z)dz\\
\le \frac C\eta \varepsilon^{-\eta}(t-s)^{\eta/2}p_c(t-s,y-x),\ \eta\in (0,1].
\end{eqnarray*}
We therefore eventually derive from  the above controls, \eqref{DEC_D_21} and \eqref{CTR_D1_B} that for $|\beta|=1 $
\begin{eqnarray*}
D_y^\beta \Big( \tilde p\otimes H^\varepsilon(s,t,x,y)\Big)\le Cp_c(t-s,y-x)\{ 1+\frac C\eta\varepsilon^{-\eta}(t-s)^{\eta/2}\}.
\end{eqnarray*}
Take now $|\beta|\ge 2 $, and let $l $ be a multi-index s.t. $|l|=1$ and $\beta-l\ge 0 $. Set for all $0\le u<t\le T, (z,y)\in \R^d,\ g_l(u,t,z,y):=D_y^lD_z\tilde p(u,t,z,y) $.
Observe that there exists $C\ge 1$, $|g_l(u,t,z,y)|\le C(t-u)^{-1}p_c(t-u,y-z) $ and also, similarly to \eqref{CTR_DER_VARSIGMA_BE}, for all multi-index $\tilde \beta,\ |\tilde \beta|\le 4 $, $ |D_y^{\tilde \beta}g_l(u,t,z+y,y)|\le C(t-u)^{-1}p_c(t-u,z)$.
 Rewrite now from \eqref{DEC_D_21}:
\begin{eqnarray*}
D_{21}^{\beta,\tau}(s,t,x,y)=D_y^{\beta-l}\int_{(s+t)/2}^{t-\tau}du\int_{\R^d}(\tilde p(s,u,x,z)-\tilde p(s,u,x,y))\langle b_\varepsilon(u,z),g_l(u,t,z,y)\rangle \Big) dz\notag\\
+D_y^{\beta-l}\int_{(s+t)/2}^{t-\tau}du\int_{\R^d}\tilde p(s,u,x,y)\langle b_\varepsilon(u,z)-b_\varepsilon(u,y),g_l(u,t,z,y)\rangle  dz,
\end{eqnarray*}
recalling that $\int_{\R^d} g_l(u,t,z,y)dz=0 $ for the last equality. Now,
\begin{eqnarray*}
|D_{21}^{\beta,\tau}(s,t,x,y)|=\\
\Big|D_y^{\beta-l}\int_{(s+t)/2}^{t-\tau}du\int_{\R^d}(\tilde p(s,u,x,z+y)-\tilde p(s,u,x,y))\langle b_\varepsilon(u,z+y),g_l(u,t,z+y,y)\rangle \Big) dz\notag\\
+D_y^{\beta-l}\int_{(s+t)/2}^{t-\tau}du\int_{\R^d}\tilde p(s,u,x,y)\langle b_\varepsilon(u,z+y)-b_\varepsilon(u,y),g_l(u,t,z+y,y)\rangle  dz\Big|\\
\le C\!\!\!\!\!\sum_{{\tiny \begin{array}{c}\!\!\!\!\!\!\!\!\!\!
\beta_1,\beta_2,\beta_3,\\
 \sum_{i=1}^3 |\beta_i|=|\beta|-1
 \end{array}}}\!\!\!\!\!\!\!\!\!\!\!\!\!\!\!\!\!\!\!(|\beta_1|,|\beta_2|,|\beta_3|)!\sum_{i=1}^d\int_{(s+t)/2}^{t-\tau}du\Bigg(\int_{\R^d}|D_y^{\beta_1}\tilde p(s,u,x,z+y)-D_y^{\beta_1}\tilde p(s,u,x,y)|\\
 \times| D_y^{\beta_2}b_\varepsilon^i(u,z+y)|  |D_y^{\beta_3}g_l^i(u,t,z+y,y)| dz\notag\\
+\int_{(s+t)/2}^{t-\tau}du\int_{\R^d}|D_{y}^{\beta_1}\tilde p(s,u,x,y)| |D_y^{\beta_2} b_\varepsilon^i(u,z+y)-D_y^{\beta_2}b_\varepsilon^i(u,y)| |D_y^{\beta_3 }g_l^i(u,t,z+y,y)|  dz\Bigg),
\end{eqnarray*}
where $(|\beta_1|,|\beta_2|,|\beta_3|)!=\frac{(\sum_{i=1}^3| \beta_i|)!}{\prod_{i=1}^3 (|\beta_i|!)}$ stands for the multinomial coefficients with entries $(|\beta|_i)_{i\in \leftB 1,3\rightB} $. 
Recall as well from \eqref{THE_DER_BEPS_BETA} that we have the following control:
\begin{eqnarray*}
|D_y^{\beta_2} b_\varepsilon(u,z+y)-D_y^{\beta_2}b_\varepsilon(u,y)|\le C \Big (\big (1+\varepsilon^{-|\beta|+1}(\I_{y+z\in V_{\varepsilon}({\mathcal I})}+\I_{y\in V_{\varepsilon}({\mathcal I})} \big)\frac{|z|}{(t-s)^{1/2}}\I_{|z|> (t-s)^{1/2}}\\
+ \big(  (1+\varepsilon^{-|\beta|}\I_{y\in V_{2\varepsilon}({\mathcal I})})|z| \I_{|z|\le \varepsilon}+ \big (1+\varepsilon^{-|\beta|+1}(\I_{y+z\in V_{\varepsilon}({\mathcal I})}+\I_{y\in V_{\varepsilon}({\mathcal I})} \big)(\frac{|z|}{\varepsilon})^\eta\I_{|z|>\varepsilon}\big) \I_{|z|\le(t-s)^{1/2}}\Big)\\
\le C\Big (\big (1+\varepsilon^{-|\beta|+1}(\I_{y+z\in V_{\varepsilon}({\mathcal I})}+\I_{y\in V_{\varepsilon}({\mathcal I})} \big)\frac{|z|}{(t-s)^{1/2}}\I_{|z|> (t-s)^{1/2}}\\
+ \big(  (1+\varepsilon^{-|\beta|+1}\I_{y\in V_{2\varepsilon}({\mathcal I})}) \I_{|z|\le \varepsilon}+ \big (1+\varepsilon^{-|\beta|+1}(\I_{y+z\in V_{\varepsilon}({\mathcal I})}+\I_{y\in V_{\varepsilon}({\mathcal I})} \big)\I_{|z|>\varepsilon}\big)\big) (\frac{|z|}{\varepsilon})^\eta \I_{|z|\le(t-s)^{1/2}}\Big).
\end{eqnarray*}
Thus,
\begin{eqnarray*}
|D_{21}^{\beta,\tau}(s,t,x,y)|\\
\le C \!\!\!\!\!\sum_{{\tiny \begin{array}{c}\!\!\!\!\!\!\!\!\!\!
\beta_1,\beta_2,\beta_3,\\
 \sum_{i=1}^3 |\beta_i|=|\beta|-1
 \end{array}}}\!\!\!\!\!\!\!\!\!\!\!\!\!\!\!\!\!\!\!(|\beta_1|,|\beta_2|,|\beta_3|)!\int_{(s+t)/2}^{t-\tau}du\int_{\R^d}\int_0^1d\lambda\Big\{ \frac{|z|}{(u-s)^{(|\beta_1|+1)/2}}p_c(u-s,y-x+\lambda z)\I_{|z|\le (t-s)^{1/2}}\\
 + \frac{|z|}{(t-s)^{1/2}}(p_c(u-s,y+z-x)+p_c(u-s,y-x))\I_{|z|>(t-s)^{1/2}} \Big\} (1+\varepsilon^{-|\beta|+1}\I_{y+z\in V_{\varepsilon}({\mathcal I}) })  \frac{1}{t-u}p_c(t-u,z) dz\notag\\
+\int_{(s+t)/2}^{t-\tau}du\int_{\R^d} \frac{1}{(t-s)^{|\beta_1|/2}}p_c(t-s,y-x) \Big\{(1+\varepsilon^{-|\beta|+1}(\I_{y+z\in V_{\varepsilon}({\mathcal I})}+\I_{y\in V_{\varepsilon}({\mathcal I})} ))\frac{|z|}{(t-s)^{1/2}}\I_{|z|> (t-s)^{1/2}}\\
+\{(1+\varepsilon^{-|\beta|+1}\I_{y\in V_{2\varepsilon}({\mathcal I})}) \I_{|z|\le \varepsilon}+ (1+\varepsilon^{-|\beta|+1}(\I_{y+z\in V_{\varepsilon}({\mathcal I})}+\I_{y\in V_{\varepsilon}({\mathcal I})}))\I_{|z|> \varepsilon}\} \I_{|z|\le (t-s)^{1/2}}(\frac{|z|}{\varepsilon})^{\eta}\Big\}\\
 \frac{1}{t-u}p_c(t-u,z)|  dz\le \frac{C
 }{(t-s)^{|\beta|}}p_c(t-s,y-x)(1+\varepsilon^{-|\beta|+1}(\frac{\varepsilon^{-\eta}}{\eta}+\frac{\varepsilon^{1/q}}{\alpha(q)})),
\end{eqnarray*}
recalling that the contribution in $\frac{\varepsilon^{1/q}}{\alpha(q)} $ comes from the terms involving $\frac{\I_{y+z\in V_{\varepsilon}({\mathcal I})}}{(t-u)^{1/2}} $ that can be handled using H\"older inequalities similarly to \eqref{HOLDER_REGU}.
This gives the stated control.
\subsection{Proof of Proposition \ref{PROP_C_D}}
Write similarly to the proof of Theorem 2.1 in \cite{kona:mamm:02}:
\begin{eqnarray}
(p_{\varepsilon}-p_{\varepsilon}^d)(t_i,t_j,x,y)=(p_{\varepsilon}\otimes H_{\varepsilon}-p_{\varepsilon}\otimes_h H_{\varepsilon})(t_i,t_j,x,y)+(p_{\varepsilon}-p_{\varepsilon}^d)\otimes_h H_{\varepsilon} (t_i,t_j,x,y)\notag\\
=\sum_{r\ge 0}(p_{\varepsilon}\otimes H_{\varepsilon}-p_{\varepsilon}\otimes_h H_{\varepsilon})\otimes_h H_{\varepsilon}^{(r)}(t_i,t_j,x,y),
\label{THE_DECOUP}
\end{eqnarray}
where we apply iteratively the first equality to get the second one. From \eqref{CTR_H_ITER} under \A{A${}_H$} and \eqref{CTR_H_ITER_PS} under \A{A${}_{PS}$}, the key point is thus to control $p_{\varepsilon}\otimes H_{\varepsilon}-p_{\varepsilon}\otimes_h H_{\varepsilon}$. 
Write:
\begin{eqnarray*}
(p_{\varepsilon}\otimes H_{\varepsilon}-p_{\varepsilon}\otimes_h H_{\varepsilon})(t_i,t_j,x,y)\notag\\
=\sum_{k=0}^{j-i-1}\int_{t_{i+k}}^{t_{i+k+1}} du\int_{\R^d} \{p_{\varepsilon}(t_i,u,x,z)H_{\varepsilon}(u,t_j,z,y)
-p_{\varepsilon}(t_i,t_{i+k},x,z)H_{\varepsilon}(t_{i+k},t_j,z,y)\}dz
\end{eqnarray*}
\begin{eqnarray}
=\sum_{k=0}^{j-i-1}\Big\{\int_{t_{i+k}}^{t_{i+k+1}} du\int_{\R^d} \{[p_{\varepsilon}(t_i,u,x,z)-p_{\varepsilon}(t_i,t_{i+k},x,z)]H_{\varepsilon}(u,t_j,z,y)\}dz\notag\\
+\int_{t_{i+k}}^{t_{i+k+1}} du\int_{\R^d} \{p_{\varepsilon}(t_i,t_{i+k},x,z)[H_{\varepsilon}(u,t_j,z,y)-H_{\varepsilon}(t_{i+k},t_j,z,y)]\}dz\Big\}
=:(D_{\varepsilon}^{d,1}+D_{\varepsilon}^{d,2})(t_i,t_j,x,y).
\label{DECOUP_INT_DER}
\end{eqnarray}
\begin{trivlist}
\item[$\bullet$] Bounds for the term $D_{\varepsilon}^{d,1}$.
\begin{trivlist}
\item[-] Under \A{A${}_H$}, for $k=0$, one readily gets from \eqref{8}:
\begin{eqnarray}
\int_{t_i}^{t_{i+1}} du \Big|\int_{\R^d} \{[p_{\varepsilon}(t_i,u,x,z)-p_{\varepsilon}(t_i,t_{i},x,z)]H_{\varepsilon}(u,t_j,z,y)\}dz\Big|\le C p_c(t_j-t_i,y-x)\int_{t_i}^{t_{i+1}}\frac{du}{(t_j-u)^{1-\gamma/2}}\nonumber\\
\le \frac{Ch}{(t_j-t_i)^{1-\gamma/2}} p_c(t_j-t_i,y-x)\le Ch^{\gamma/2}p_c(t_j-t_i,y-x).\label{K0}
\end{eqnarray}

On the other hand, from the parametrix expansion of the density in \eqref{SERIE_P_EPS}, one gets that for all $\eta\in (0,\gamma) $ and $k\ge 1 $:
\begin{equation}
\label{HOLD_TEMPS}
|p_\varepsilon(t_i,u,x,z)-p_{\varepsilon}(t_i,t_{i+k},x,z)|\le \frac{C}{\eta}\frac{(u-t_{i+k})^{(\gamma-\eta)/2}}{(u-t_i)^{\gamma/2}}p_c(u-t_i,z-x).
\end{equation}
Write indeed, recalling that $u\in [t_{i+k},t_{i+k+1}] $: 
\begin{eqnarray}
p_\varepsilon(t_i,u,x,z)-p_{\varepsilon}(t_i,t_{i+k},x,z)=\tilde p_\varepsilon(t_i,u,x,z)-\tilde p_{\varepsilon}(t_i,t_{i+k},x,z)\nonumber\\
+ \sum_{l\ge 1}\int_{t_i}^{t_{i+k}}ds \int_{\R^d} \tilde p_\varepsilon\otimes H_\varepsilon^{(l-1)}(t_i,s,x,w) \big(H_\varepsilon(s,u,w,z )-H_\varepsilon(s,t_{i+k},w,z ) \big)dw\nonumber\\
+ \sum_{l\ge 1}\int_{t_{i+k}}^{u}ds\int_{\R^d}\tilde p_\varepsilon(t_i,s,x,w)H_\varepsilon^{(l)}(s,u,w,z )dw=:(T_{1,\varepsilon}+T_{2,\varepsilon}+T_{3,\varepsilon})(t_i,t_{i+k},u ,x,z), \label{PREAL_HOLDER_BOUND_HK_T}
\end{eqnarray}
with the convention $\tilde p_\varepsilon\otimes H_\varepsilon^{(0)}=\tilde p_\varepsilon $. Since $\tilde p_\varepsilon $ is a Gaussian non-degenerate kernel, and that for $k\ge 1$ we readily get from the mean value theorem:
\begin{equation}
\label{CRT_T1}
|T_{1,\varepsilon}(t_i,t_{i+k},u,x,z)|\le \frac{C(u-t_{i+k})}{(u-t_i)}p_c(u-t_i,z-x)\le \frac{C(u-t_{i+k})^{(\gamma-\eta)/2}}{(u-t_i)^{\gamma/2}}p_c(u-t_i,z-x).
\end{equation}
Also, from the definition of $\Phi_\varepsilon $ following \eqref{5} and the associated control \eqref{BOUND_PHI}, we get:
\begin{equation}
\label{CRT_T3}
\begin{split}
|T_{3,\varepsilon}(t_i,t_{i+k},u,x,z)|&\le \int_{t_{i+k}}^u ds\int_{\R^d}\tilde p_\varepsilon(t_i,s,x,w) |\Phi_\varepsilon(s,u,w,z)| dw\le Cp_c(u-t_i,z-x)\int_{t_{i+k}}^u \frac{ds}{(u-s)^{1-\gamma/2}}\\
&\le C(u-t_{i+k})^{\gamma/2}p_c(u-t_i,z-x).
\end{split}
\end{equation}
For $T_{2,\varepsilon}$, we again use some splitting in time. Write $T_{1,\varepsilon}(t_i,t_{i+k},u,x,z)=(T_{21,\varepsilon}+T_{22,\varepsilon})(t_i,t_{i+k},u,x,z) $ where:
\begin{eqnarray}
T_{21,\varepsilon}(t_i,t_{i+k},u,x,z)&:=&\int_{t_{i+k}-(u-t_{i+k})}^{t_{i+k}}ds \int_{\R^d}\tilde p_\varepsilon (t_i,s,x,w)(\Phi_\varepsilon(s,u,w,z)-\Phi_\varepsilon(s,t_{i+k},w,z))dw,\nonumber \\
|T_{21,\varepsilon}(t_i,t_{i+k},u,x,z)|&\le& \int_{t_{i+k}-(u-t_{i+k})}^{t_{i+k}}ds\Big( \frac{p_c(u-t_i,z-x)}{(u-s)^{1-\gamma/2}}+\frac{p_c(t_{i+k}-t_i,z-x)}{(t_{i+k}-s)^{1-\gamma/2}} \Big)\nonumber\\
&\le& C (u-t_{i+k})^{\gamma/2}p_c(u-t_i,z-x),\label{CRT_T_21}
\end{eqnarray}
recalling that, since $k\ge 1 $, $(u-t_i)\ge t_{i+k}-t_i\ge \frac 12 (u-t_i) $ for the last inequality.
 For
\begin{equation*}
T_{22,\varepsilon}(t_i,t_{i+k},u,x,z):= \sum_{l\ge 1}\int_{t_i}^{t_{i+k}-(u-t_{i+k})}ds \int_{\R^d} \tilde p_\varepsilon\otimes H_\varepsilon^{(l-1)}(t_i,s,x,w) \big(H_\varepsilon(s,u,w,z )-H_\varepsilon(s,t_{i+k},w,z ) \big)dw,
 \end{equation*}
we focus on the second order terms in the difference $\big(H_\varepsilon(s,u,w,z )-H_\varepsilon(s,t_{i+k},w,z ) \big) $. They are indeed the most singular. Note that on the considered time set $\frac 12(u-s)\le (t_{i+k}-s) $.
We thus get (with similar arguments than those used to handle $T_{1,\varepsilon} $):
\begin{eqnarray*}
\Big|\Tr\big((a(s,w)-a(s,z))(D_w^2 \tilde p_\varepsilon(s,u,w,z)-D_w^2 \tilde p_\varepsilon(s,t_{i+k},w,z)) \big) \Big|\le C\frac{|z-w|^{\gamma}(u-t_{i+k})}{(u-s)^2}p_c(u-s,z-w)\\
\le C\Big(\frac{|z-w|}{(u-s)^{1/2}} \Big)^\gamma \frac{(u-t_{i+k})^{(\gamma-\eta)/2}}{(u-s)^{1-\eta/2}}p_c(u-s,z-w)\le \frac{(u-t_{i+k})^{(\gamma-\eta)/2}}{(u-s)^{1-\eta/2}}p_c(u-s,z-w)
\end{eqnarray*}

The small loss on the time H\"older regularity index is here due to the fact we consider the forward time component and the $\eta $ is needed to integrate. We obtain:
\begin{eqnarray*}
|T_{22,\varepsilon}(t_i,t_{i+k},u,x,z)|\le \frac{C(u-t_{i+k})^{(\gamma-\eta)/2}}{\eta}p_c(u-t_i,z-x).
\end{eqnarray*}
Plugging this last bound and the controls of \eqref{CRT_T_21}, \eqref{CRT_T3}, \eqref{CRT_T1} into \eqref{PREAL_HOLDER_BOUND_HK_T} yields \eqref{HOLD_TEMPS}.

Now, from \eqref{K0} and using \eqref{8} and \eqref{HOLD_TEMPS} in \eqref{DECOUP_INT_DER}, we get:
\begin{equation}
\label{CTR_D_1_N}
|D_{\varepsilon}^{d,1}|(t_i,t_j,x,y)\le C_\eta
h^{(\gamma-\eta)/2}
p_c(t_j-t_i,y-x).
\end{equation}
\item[-] Under \A{A${}_{PS}$}, since we want to get higher convergence rates, we need to use the the forward Kolmogorov equation in $D_{\varepsilon}^{d,1}$. Write for $k\in \leftB 1,j-i-2\rightB,\ u\in [t_k,t_{k+1}] $:
\begin{eqnarray*}
p_{\varepsilon}(t_i,u,x,z)-p_{\varepsilon}(t_i,t_{i+k},x,z)=(u-t_{i+k}) \int_0^1 \big(\partial_v p_{\varepsilon}(t_i,v,x,z)\big)_{v=t_{i+k}+\lambda (u-t_{i+k})}d\lambda\\
=(u-t_{i+k}) \int_0^1 \big( (L_{v}^{\varepsilon})^* p_{\varepsilon}(t_i,v,x,z)\big)_{v=t_{i+k}+\lambda (u-t_{i+k})}d\lambda.
\end{eqnarray*}
\item[$\diamond $] If now $\sigma(t,x)=\sigma $, the term $ H_\varepsilon$ in \eqref{DECOUP_INT_DER} only involves a first order derivative.  We obtain:
\begin{eqnarray}
|D_{\varepsilon}^{d,1}|(t_i,t_j,x,y)&\le& Ch\Big( \sum_{k=1}^{ j-i-2}\int_0^1d\lambda \int_{t_{i+k}}^{t_{i+k+1}}du\int_{\R^d}\Big[|{\rm div}_z\big(b_\varepsilon(v,z) p_\varepsilon(t_i,v,x,z)\big)|\frac{1}{(t_j-u)^{1/2}}p_c(t_j-u,y-z)\nonumber\\
&&+  \frac 12 \big|\langle a \nabla_zp_\varepsilon(t_i,v,x,z), \nabla_z  \langle b_\varepsilon(u,z), \nabla_z \tilde p_\varepsilon(u,t_j,z,y)\rangle \rangle \big|\Big]_{v=t_{i+k}+\lambda (u-t_{i+k})} dz\Big)\nonumber\\
&&+C \int_{[t_{j-1},t_j]} du \int_{\R^d} |p_\varepsilon(t_i,u,x,z)-p_\varepsilon(t_i,t_{j-1},x,z)| \frac{1}{(t_j-u)^{1/2}}p_c(t_j-u,y-z)dz\nonumber\\
&&+C\frac{h}{(t_j-t_i)^{1/2}}p_c(t_j-t_i,y-x). 
\label{PREAL_APS}
\end{eqnarray}
From the parametrix representation \eqref{SERIE_P_EPS} of $p_\varepsilon$, it is again easily deduced similarly to \eqref{HOLD_TEMPS} that for any $\eta\in (0,1/2),\ u\in [t_{j-1},t_j] $:
$$|p_\varepsilon(t_i,u,x,z)-p_\varepsilon(t_i,t_{j-1},x,z)|\le \frac{C}{\eta}\frac{|u-t_{j-1}|^{(1-\eta)/2}}{(u-t_i)^{1/2}}p_c(u-t_i,z-x).$$
 Plugging this estimate in \eqref{PREAL_APS} and using as well \eqref{THE_DER_BEPS_BETA} and \eqref{CTR_DER_PERT_PS} yields for all $\eta\in (0,1/2) $:
\begin{eqnarray}
&&|D_{\varepsilon}^{d,1}|(t_i,t_j,x,y)\nonumber\\
&\le& C_\eta h\Big(\sum_{k=1}^{j-i-2}\int_0^1 d\lambda\int_{t_{i+k}}^{t_{i+k+1}}du \int_{\R^d}\Big[ \Big(\varepsilon^{-1}\I_{z\in{\mathcal V}_\varepsilon(\mathcal I)}+\frac{\varepsilon^{-\eta}
}{(v-t_i)^{1/2}}\Big)p_c(v-t_i,z-x)\frac{p_c(t_j-u,y-z)}{(t_j-u)^{1/2}}\nonumber\\
&&+\varepsilon^{-\eta}\frac{p_c(v-t_i,z-x)}{(v-t_i)^{1/2}}\Big(\frac{\varepsilon^{-1}}{(t_j-u)^{1/2}}\I_{z\in{\mathcal V}_\varepsilon(\mathcal I)}+\frac{1
}{(t_j-u)}\Big)p_c(t_j-u,y-z)\Big]_{v=t_{i+k}+\lambda (u-t_{i+k})} dz\Big)
\nonumber\\
&&+\frac{Ch^{1-\eta/2}}{\eta(t_j-t_i)^{1/2}}p_c(t_j-t_i,y-x)\nonumber\\
&\le & C_\eta \Big(\frac{h\varepsilon^{-(1+\eta)+1/q}}{\alpha(q)(t_j-t_i)^{1/2-\alpha(q)}}+\frac{h|\ln(h)|\varepsilon^{-\eta}}{(t_j-t_i)^{1/2}}+\frac{h^{1-\eta/2}}{\eta(t_j-t_i)^{1/2}}\Big)p_c(t_j-t_i,y-x)\nonumber\\
&\le& \bar C_{\eta,q}\Big( \frac{h\varepsilon^{-(1+\eta)+1/q}}{(t_j-t_i)^{1/2-\alpha(q)}}+\frac{h|\ln(h)|\varepsilon^{-\eta}}{(t_j-t_i)^{1/2}}+\frac{h^{1-\eta/2}}{(t_j-t_i)^{1/2}}\Big)p_c(t_j-t_i,y-x),\label{APS_SIG_CONST}
\end{eqnarray}
recalling as well for the previous computations that $u-v= (1-\lambda)(u-t_{i+k})\le h$.
\item[$\diamond $] For a general $\sigma $, an additional term appears in \eqref{PREAL_APS}, which corresponds to the second order terms in $H_\varepsilon$ for the indexes $k\in \leftB 1,j-i-2\rightB $. The other contributions are controlled similarly.
We have to bound:
\begin{eqnarray*}
Ch \sum_{k=1}^{ j-i-2}\int_0^1d\lambda \int_{t_{i+k}}^{t_{i+k+1}}du\int_{\R^d} \Big[
\big|\sum_{l,m=1}^d D_{z_lz_m}^2 \big(a_{lm}(t_i,z) p_\varepsilon(t_i,v,x,z)\big)  \\
 \Tr\big( (a_\varepsilon(u,z)-a_\varepsilon(u,y)) D_z^2 \tilde p_\varepsilon(u,t_j,z,y)\big) \big|\Big]_{v=t_{i+k}+\lambda (u-t_{i+k})} dz\\
 \le \bar C_{\eta,q} h \sum_{k=1}^{ j-i-2}\int_0^1d\lambda \int_{t_{i+k}}^{t_{i+k+1}}du\int_{\R^d} \frac{\varepsilon^{-(1+\eta)}}{u-t_i} p_c(u-t_i,z-x)\frac{1}{(t_j-u)^{1/2}}p_x(t_j-u,y-z)dz
 \le \bar C_{\eta,q} \frac{h |\ln(h)|\varepsilon^{-(1+\eta)}}{(t_j-t_i)^{1/2}}.
\end{eqnarray*}
This yields in the considered case:
\begin{eqnarray}
|D_{\varepsilon}^{d,1}|(t_i,t_j,x,y)\le \bar C_{\eta,q}\Big( \frac{h|\ln(h)|\varepsilon^{-(1+\eta)}}{(t_j-t_i)^{1/2}}+\frac{h^{1-\eta/2}}{(t_j-t_i)^{1/2}}\Big)p_c(t_j-t_i,y-x).\label{APS_SIG_GEN}
\end{eqnarray}
\end{trivlist}

\item[$\bullet$]To control the term $D_\varepsilon^{d,2} $ appearing in \eqref{DECOUP_INT_DER}, let us first introduce:
\begin{eqnarray*}
(\bar D_{\varepsilon,\sigma}^{d,21}+\bar D_{\varepsilon,\sigma}^{d,22})(t_i,t_{i+k},u,t_j,x,y):=\\
C \int_{\R^d} p_c(t_{i+k}-t_i,z-x)|a_{\varepsilon}(u,z)-a_{\varepsilon}(u,y)-( a_{\varepsilon}(t_{i+k},z)-a_{\varepsilon}(t_{i+k},y))|
\frac{1}{t_j-t_{i+k}}p_c(t_j-t_{i+k},y-z)dz\\
+\Big|\int_{\R^d} p_{\varepsilon}(t_i,t_{i+k},x,z) \Tr\Big((a_{\varepsilon}(u,z)-a_{\varepsilon}(u,y))[D_z^2 \tilde p_{\varepsilon}(u,t_j,z,y)-D_z^2 \tilde p_{\varepsilon}(t_{i+k},t_j,z,y)]\Big)\Big|dz,\\
\end{eqnarray*}
that correspond to the most singular contributions in $D_{\varepsilon}^{d,2} $ as far as the time singularity is concerned when the diffusion coefficient varies, i.e. $ \sigma(t,x)\neq \sigma$. 
\item[-] Under \A{A${}_H $}.
For $\bar D_{\varepsilon,\sigma}^{d,22}$ we can exploit the H\"older continuity in time of the Gaussian kernel $\tilde p_\varepsilon $ to derive, similarly to the computations performed above to investigate $T_{22,\varepsilon}$, that for all $\eta\in (0,\gamma) $, $k\in \leftB 0,j-i-2\rightB,\ u\in [t_{i+k},t_{i+k+1}]$ :
$$ |D_z^2 \tilde p_{\varepsilon}(u,t_j,z,y)-D_z^2 \tilde p_{\varepsilon}(t_{i+k},t_j,z,y)|\le C\frac{(u-t_{i+k})^{(\gamma-\eta)/2}}{(t_j-u)^{1+(\gamma-\eta)/2}}p_c(t_j-t_{i+k},y-z).$$
From the spatial H\"older continuity of $a_\varepsilon(u,\cdot)$, we get:
\begin{equation}
\label{DND22}
\sum_{k=0}^{j-i-2}\int_{t_{i+k}}^{t_{i+k+1}} du |\bar D_{\varepsilon,\sigma}^{d,22}|(t_i,t_{i+k},u,t_j,x,y)\le \frac{C}{\eta}h^{(\gamma-\eta)/2}p_c(t_j-t_i,y-x).
\end{equation}
On the other hand, for $k=j-i-1 $, \eqref{8} readily yields:
\begin{equation}
\label{AH_DND22_LS}
\begin{split}
\int_{t_{j-1}}^{t_{j}} du |\bar D_{\varepsilon,\sigma}^{d,22}|(t_i,t_{i+k},u,t_j,x,y)\le C\int_{t_{j-1}}^{t_j}du\int_{\R^d}p_c(t_{j-1}-t_i,z-x)\Big(\frac{p_c(t_j-u,y-z)}{(t_j-u)^{1-\gamma/2}}+\frac{p_c(t_j-t_{j-1},y-z)}{(t_j-t_{j-1})^{1-\gamma/2}}\Big)dz \\
\le Ch^{\gamma/2}p_c(t_j-t_i,y-x).
\end{split}
\end{equation}

Also, using the uniform $\gamma/2 $-H\"older continuity in time of $a$ we get:
\begin{eqnarray*}
|\bar D_{\varepsilon,\sigma}^{d,21} (t_i,t_{i+k},u,t_j,x,y)|\\
\le C \int_{\R^d} p_c(t_{i+k}-t_i,z-x)|u-t_{i+k}|^{\gamma/2} \frac{1}{t_j-t_{i+k}}p_c(t_j-t_{i+k},y-z)dz\\
\le Ch^{(\gamma-\eta)/2} p_c(t_j-t_i,y-x)(t_j-t_{i+k})^{-1+\eta/2},
\end{eqnarray*}
for $\eta\in (0,\gamma) $, recalling $u\in [t_{i+k},t_{i+k+1}] $ for the last inequality.
The difference of the first order terms appearing in $D_{\varepsilon}^{d,2}$ in \eqref{DECOUP_INT_DER} yields similar controls. From the above bound, \eqref{DND22} and \eqref{AH_DND22_LS}, we derive that under \A{A${}_H$}:
\begin{equation}
\label{CTR_D_2_N}
|D_\varepsilon^{d,2}(t_i,t_j,x,y)|\le C_\eta h^{(\gamma-\eta)/2}p_c(t_j-t_i,y-x).
\end{equation}

\item[-] Under \A{A${}_{PS}$}, write:
\begin{eqnarray}
\label{DND22_H}
\sum_{k=0}^{j-i-1}\int_{t_{i+k}}^{t_{i+k+1}} du |\bar D_{\varepsilon,\sigma}^{d,22}|(t_{i},t_{i+k},u,t_j,x,y)\nonumber\\
\le C\Big( h\sum_{k=0}^{\lceil \frac{j-i-1}2\rceil}\int_{t_{i+k}}^{t_{i+k+1}} du \int_{\R^d}p_c(t_{i+k}-t_i,z-x)\frac{1}{(t_j-t_i)^{3/2}}p_c(t_j-u,y-z)dz\nonumber\\
+\sum_{k=\lceil \frac{j-i-1}2\rceil+1}^{j-i-1}\int_{t_{i+k}}^{t_{i+k+1}}du \int_{\R^d}\sum_{l,m=1}^d \Big|D_{z_lz_m}^2\Big(p_\varepsilon(t_i,t_{i+k},x,z)(a_{\varepsilon}^{lm}(u,z)-a_{\varepsilon}^{lm}(u,y))\Big)\Big|\nonumber\\
\times |\tilde p_\varepsilon(u,t_j,z,y)-\tilde p_\varepsilon(t_{i+k},t_j,z,y)|dz \Big)\le \bar C_{\eta,q}\Big(\frac{h}{(t_j-t_i)^{1/2}}+\frac{h^{1-\eta/2}\varepsilon^{-(1+\eta)}}{
(t_j-t_i)^{1/2-\eta/2}}\Big) p_c(t_j-t_i,y-x),
\end{eqnarray}
where to derive the last inequality, we exploit \eqref{CTR_DER_PERT_PS}, the time H\"older continuity of the Gaussian density $\tilde p_\varepsilon$ for $k\in \leftB \lceil \frac{j-i-1}{2}\rceil+1,j-i-2\rightB $ and direct computations for $k=j-i-1$.
Also, the smoothness in time (Lipschitz continuity) of the diffusion coefficients gives
for $\eta\in (0,1] $,
\begin{equation}
\label{CTR_D_EPS_SIG_21}
|\bar D_{\varepsilon,\sigma}^{d,21} (t_i,t_{i+k},u,t_j,x,y)|\le Ch^{1-\eta/2} p_c(t_j-t_i,y-x)(t_j-t_{i+k})^{-1+\eta/2}.
\end{equation}

 Let us now carefully mention that, under \A{A${}_{PS} $}, because of the irregularity of the drift, it is very important as well to establish cautiously the bounds for the difference of the first order terms. Introduce:
\begin{eqnarray}
\label{DEPS_B}
(\bar D_{\varepsilon,b}^{d,21}+\bar D_{\varepsilon,b}^{d,22})(t_i,t_{i+k},u,t_j,x,y):=\nonumber\\
C \int_{\R^d} p_c(t_{i+k}-t_i,z-x)|b_{\varepsilon}(u,z)- b_{\varepsilon}(t_{i+k},z)|
\frac{1}{(t_j-t_{i+k})^{1/2}}p_c(t_j-t_{i+k},y-z)dz\nonumber\\
+\Big|\int_{\R^d} p_{\varepsilon}(t_i,t_{i+k},x,z) \langle b_{\varepsilon}(u,z),D_z \tilde p_{\varepsilon}(u,t_j,z,y)-D_z \tilde p_{\varepsilon}(t_{i+k},t_j,z,y)\rangle\Big|dz.
\end{eqnarray}
From the Lipschitz property in time of $b_\varepsilon(\cdot,z) $ we readily get:
\begin{equation}
\label{D21_APS_CONST}
\bar D_{\varepsilon,b}^{d,21}(t_i,t_{i+k},u,t_j,x,y)
\le \frac{Ch}{(t_j-t_{i+k})^{1/2}}p_c(t_j-t_i,y-x).
\end{equation}
Also, recalling that $\partial_u \tilde p_\varepsilon (u,t_j,z,y)+\frac 12 {\rm Tr}\big( a(u,y) D_z^2 \tilde p_\varepsilon(u,t_j,z,y)\big)=0$, one readily gets:
\begin{equation}
\label{D22_APS_CONST_LOIN}
\bar D_{\varepsilon,b}^{d,22}(t_i,t_{i+k},u,t_j,x,y)\le \frac{Ch}{(t_j-u)^{3/2}}p_c(t_j-t_i,y-x),
\end{equation}
which once integrated in time gives the expected control for $k\in \leftB 0, \lceil \frac{j-i-1}{2}\rceil\rightB $. The indexes $k\in \leftB\lceil \frac{j-i-1}{2}\rceil+1,j-i-1\rightB $ require a more careful treatment. 
Now, for such indexes 
and 
$u\in [t_{i+k},t_{i+k+1}] $, using again the Kolmogorov equation satisfied by $\tilde p_\varepsilon $ and two spatial integration by parts in $z$, one obtains from \eqref{CTR_DER_PERT_PS} the following \textit{global} control:
\begin{eqnarray}
\label{CTR_DRIFT_GLOB}
&&\bar D_{\varepsilon,b}^{d,22}(t_i,t_{i+k},u,t_j,x,y)\nonumber\\
&\le& C (u-t_{i+k})\int_0^1 d\lambda \sum_{l,m,q\in \leftB 1,d\rightB}\int_{\R^d} \Big[\big|D_{z_lz_m}^2\big(p_\varepsilon(t_i,t_{i+k},x,z)b_\varepsilon^l(u,z)\big)\big| |D_{z_q} \tilde p_{\varepsilon}(v,t_j,z,y)|\Big]_{v=t_{i+k}+\lambda(u-t_{i+k})}dz\nonumber\\
&\le& Ch\int_{0}^1d\lambda\int_{\R^d}\Big( \frac{\varepsilon^{-(1+\eta)}}{(t_{i+k}-t_i)} +\varepsilon^{-2}\I_{z\in {\mathcal V}_\varepsilon({\mathcal I}) }\Big)p_c(t_{i+k}-t_i,z-x)\frac{p_c(t_j-v,y-z)}{(t_j-u)^{1/2}}\Big|_{v=t_{i+k}+\lambda(u-t_{i+k})} dz\nonumber\\
&\le& Ch p_c(t_j-t_i,y-x)\Big(\frac{\varepsilon^{-(1+\eta)}}{(t_j-t_i)(t_j-u)^{1/2}}+\frac{\varepsilon^{-2+ 1/q}}{(t_j-u)^{1/2+d/(2q)}}\Big), q>d.
\end{eqnarray}
Plugging \eqref{CTR_DRIFT_GLOB}, \eqref{D22_APS_CONST_LOIN} and \eqref{D21_APS_CONST} into \eqref{DEPS_B} one derives:
\begin{equation}
\label{BORNE_DRIFT_APS}
\sum_{k=0}^{j-i-1}\int_{t_{i+k}}^{t_{i+k+1}} \!\!\!\!\!\!du \Big(\bar D_{\varepsilon,b}^{d,21}(t_i,t_{i+k},u,t_j,x,y)+\bar D_{\varepsilon,b}^{d,22}(t_i,t_{i+k},u,t_j,x,y)\Big)\le Ch\Big(\frac{\varepsilon^{-(1+\eta)}}{(t_j-t_i)^{1/2}} +\varepsilon^{-2+1/q}\Big) p_c(t_j-t_i,y-x).
\end{equation}

We carefully, point out that, since $q>d$, this term will dominate the error associated with the time discretization when compared to \eqref{DND22_H}.

We will now improve this bound using the (unsigned) distance of the final point to the neighborhood of the discontinuity sets $d(y,{\mathcal V}_\varepsilon({\mathcal I}))$.  We cannot hope to improve the control \eqref{D22_APS_CONST_LOIN} for $k\in \leftB0, \lceil \frac{j-i-1}2\rceil\rightB$ and therefore focus on the indexes $k\in \leftB \lceil \frac{j-i-1}2\rceil+1, j-i-1\rightB $.
For those indexes, performing one spatial integration by part in $z$ from \eqref{DEPS_B} yields:
\begin{eqnarray*}
\sum_{k=\lceil \frac{j-i-1}2\rceil+1}^{j-i-1}\int_{t_{i+k}}^{t_{i+k+1}} du \bar D_{\varepsilon,b}^{d,22}(t_i,t_{i+k},u,t_j,x,y)\\
\le C\!\!\!\sum_{k=\lceil  \frac{j-i-1}2\rceil+1}^{j-i-1}\int_{t_{i+k}}^{t_{i+k+1}} \!\!\!\!\! \! du \Big( 
\int_{\R^d}\Big(\frac{\varepsilon^{-\eta}}{(t_{i+k}-t_i)^{1/2}}+\varepsilon^{-1}\I_{z\in {\mathcal V}_\varepsilon({\mathcal I})}\Big)p_c(t_{i+k}-t_i,z-x)
|\tilde p_{\varepsilon}(u,t_j,z,y)-\tilde p_{\varepsilon}(t_{i+k},t_j,z,y)|\Big)dz\\
\le \frac{C}{\eta}
\Big(\frac{h^{1-\eta/2}\varepsilon^{-\eta}}{(t_{j}-t_i)^{(1-\eta)/2}}
p_c(t_j-t_i,y-x)+\bar R_\varepsilon^{d,22}(t_i,t_j,x,y) \Big),
\end{eqnarray*}
using the H\"older continuity in time of $\tilde p_\varepsilon $ for $k\in \leftB \lceil \frac{j-i-1}2\rceil+1,j-i-2\rightB $ and direct computations for $k=j-i-1$ in  the second inequality and where 
\begin{eqnarray*}
\bar R_\varepsilon^{d,22}(t_i,t_j,x,y)&:=&h^{1-\eta/2}\sum_{k=\lceil \frac{j-i-1}2\rceil+1}^{j-i-2}\int_{t_{i+k}}^{t_{i+k+1}} \frac{du}{(t_j-u)^{1-\eta/2}}\varepsilon^{-1}\int_{\R^d}p_c(u-t_i,z-x)\I_{z\in {\mathcal V}_{\varepsilon}({\mathcal I})}p_c(t_j-u,y-z)dz\\
&&+ \varepsilon^{-1}\int_{t_{j-1}}^{t_j}du \int_{\R^d}p_c(u-t_i,z-x)\I_{z\in {\mathcal V}_{\varepsilon}({\mathcal I})}\big(p_c(t_j-u,y-z)+p_c(t_j-t_{j-1},y-z)\big)dz.
\end{eqnarray*}
Since $|y-z|+|z-x|\ge |y-x| $ and $(u-t_i)\ge \frac 12 (t_j-t_i) $ we get that up to a modification of $c$ that for $k\in \leftB \lceil \frac{j-i-1}2\rceil+1,j-i-1\rightB $ and $u\in [t_{i+k},t_{i+k+1}],s=u $ or $u\in [t_{j-1},t_j]$, $s=t_{j-1}$:
\begin{eqnarray}
I_\varepsilon(t_i,t_{i+k},u,s,t_j,x,y)&:=&\varepsilon^{-1}\int_{\R^d}p_c(u-t_i,z-x)\I_{z\in {\mathcal V}_{\varepsilon}({\mathcal I})}p_c(t_j-s,y-z)dz\nonumber\\
&\le& C\varepsilon^{-1}
p_c(t_j-t_i,y-x)\int_{\R^d}
\I_{z\in {\mathcal V}_{\varepsilon}({\mathcal I})}p_c(t_j-s,y-z)dz.\label{SORTIE_DENS}
\end{eqnarray}
Indeed, either $|z-x|\ge \frac 12 |y-x| $ and in that case $p_c(u-t_i,z-x)\le Cp_c(t_j-t_i,y-x) $, or $|y-z|\ge \frac 12|y-x| $. In that case we use that $p_c(u-t_i,z-x)\le C/(t_j-t_i)^{d/2} $ and write as well:
 \begin{eqnarray*}
 \exp\Big(-\frac c2\frac{|y-z|^2}{t_j-s}\Big)\le \exp\Big(-\frac{c}{16}\frac{|y-x|^2}{(t_j-s)}\Big)\exp\Big(-\frac{c}4\frac{|y-z|^2}{(t_j-s)}\Big)\le \exp\Big(-\frac{c}{16}\frac{|y-x|^2}{(t_j-t_i)}\Big)\exp\Big(-\frac{c}4\frac{|y-z|^2}{(t_j-s)}\Big),
 \end{eqnarray*}
which also gives \eqref{SORTIE_DENS} modifying $c$.

Up to a change of coordinate, in order to straighten the boundary, we can write (following the arguments of Section \ref{SUBSEC_HODLER_BOREL} that led to \eqref{CTR_DIST_DENS}):
\begin{eqnarray*}
I_\varepsilon(t_i,t_{i+k},u,s,t_j,x,y)&\le& C\varepsilon^{-1}p_c(t_j-t_i,y-x)\int_{-\varepsilon}^{\varepsilon} \exp\Big(-\frac{|\bar z-d_S(y,{\mathcal V}_\varepsilon(\mathcal I))|^2}{2(t_j-s)}\Big)\frac{d\bar z}{(t_j-s)^{1/2}}\\
&\le& C\varepsilon^{-1}p_c(t_j-t_i,y-x)\int_{-\varepsilon}^{\varepsilon} \frac{d\bar z}{|\bar z-d_S(y,{\mathcal V}_\varepsilon(\mathcal I))|},
\end{eqnarray*}
where $d_S(y,{\mathcal V}_\varepsilon(\mathcal I))$ stands for the \textbf{signed} distance\footnote{Since the discontinuity sets are bounded, we can for instance choose the distance to be positive for the points inside the bounded domain associated with the boundary. Anyhow, this choice plays no role here.} of $y$ to the boundary of ${\mathcal V}_\varepsilon(\mathcal I) $. Since we have assumed that for this part of the Proposition that $|d_S(y,{\mathcal V}_\varepsilon(\mathcal I))|\ge 2\varepsilon $ we get 
$$|\bar z-d_S(y,{\mathcal V}_\varepsilon(\mathcal I))|\ge |d_S(y,{\mathcal V}_\varepsilon(\mathcal I)|-|\bar z|\ge |d_S(y,{\mathcal V}_\varepsilon(\mathcal I)|-\varepsilon \ge \frac{|d_S(y,{\mathcal V}_\varepsilon(\mathcal I))|}2=:\frac{d(y,{\mathcal V}_\varepsilon(\mathcal I))}2,$$ 
where $d(y,{\mathcal V}_\varepsilon(\mathcal I))$ is the \textbf{unsigned} distance of $y$ to the boundary of ${\mathcal V}_\varepsilon(\mathcal I) $.
We finally derive from the above computations \eqref{D21_APS_CONST} and \eqref{D22_APS_CONST_LOIN}:
\begin{equation}
\label{BD_FINAL_APS_DIST}
\sum_{k=0}^{j-i-1} \int_{t_{i+k}}^{t_{i+k+1}} du (\bar D_{\varepsilon,b}^{d,21}+\bar D_{\varepsilon,b}^{d,22})(t_i,t_{i+k},u,t_j,x,y)\le  C_\eta h^{1-\eta/2}\Big(\frac{\varepsilon^{-\eta}}{(t_{j}-t_i)^{1/2}}
+\frac{1}{d(y,{\mathcal V}_\varepsilon(\mathcal I))} \Big)p_c(t_j-t_i,y-x).
\end{equation}

\item[$\bullet $] Final derivation of the bounds.

Recall first that:
\begin{eqnarray*}
|p_\varepsilon-p_\varepsilon^d|(t_i,t_j,x,y)&\le& \sum_{r\ge 0}|(p_\varepsilon\otimes H_\varepsilon-p_\varepsilon\otimes_h H_\varepsilon)| \otimes_h |H_\varepsilon^{(r)}|(t_i,t_j,x,y).
\end{eqnarray*}
\item[-] Under \A{A${}_H$}, 
we first plug \eqref{CTR_D_2_N}, \eqref{CTR_D_1_N} into \eqref{DECOUP_INT_DER}. The bound \eqref{CTR_PN_PND} of the proposition then follows from the above inequality using \eqref{CTR_H_ITER}.

\item[-] Under \A{A${}_{PS} $}.

\item[$\diamond $] For a general $\sigma(t,x) $ (which varies), we derive from \eqref{APS_SIG_GEN}, \eqref{DND22_H}, \eqref{CTR_D_EPS_SIG_21}, \eqref{BORNE_DRIFT_APS} and \eqref{THE_DECOUP}, \eqref{CTR_H_ITER}:
\begin{eqnarray*}
|p_\varepsilon-p_\varepsilon^d|(t_i,t_j,x,y)&\le& \sum_{r\ge 0}|(p_\varepsilon\otimes H_\varepsilon-p_\varepsilon\otimes_h H_\varepsilon)| \otimes_h |H_\varepsilon^{(r)}|(t_i,t_j,x,y)\\
&\le& \bar C_{\eta,q}\Big( h|\ln(h)|\varepsilon^{-(1+\eta)}
+h^{1-\eta/2}\varepsilon^{-(1+\eta)}+h\varepsilon^{-2+1/q}\Big) p_c(t_j-t_i,y-x).
\end{eqnarray*}
Using \eqref{BD_FINAL_APS_DIST} instead of \eqref{BORNE_DRIFT_APS} when $d(y,{\mathcal V}_\varepsilon({\mathcal I}))\ge 2\varepsilon $ yields:
\begin{eqnarray*}
|p_\varepsilon-p_\varepsilon^d|(t_i,t_j,x,y)
&\le& \bar C_{\eta,q}\Big( h|\ln(h)|\varepsilon^{-(1+\eta)}
+h^{1-\eta/2}\varepsilon^{-(1+\eta)}+\frac{h^{1-\eta/2}}{d(y,{\mathcal V}_\varepsilon({\mathcal I}))}\Big) p_c(t_j-t_i,y-x).
\end{eqnarray*}

\item[$\diamond $] For $\sigma(t,x)=\sigma $ (fixed diffusion coefficient), when $d(y,{\mathcal V}_\varepsilon({\mathcal I}))\ge 2\varepsilon $, we derive from \eqref{APS_SIG_CONST}, \eqref{BD_FINAL_APS_DIST} that:
\begin{eqnarray*}
|p_\varepsilon-p_\varepsilon^d|(t_i,t_j,x,y)
&\le & \bar C_{\eta,q}\Big( h|\ln(h)|\varepsilon^{-\eta}
+h\varepsilon^{-(1+\eta)+1/q}+\frac{h^{1-\eta/2}}{d(y,{\mathcal V}_\varepsilon({\mathcal I}))}\Big) p_c(t_j-t_i,y-x).
\end{eqnarray*}
Observe that in this case the contribution $(\bar D_{\varepsilon,\sigma}^{d,2j})_{j\in \{1,2\}}$ vanish. The upper bound of \eqref{DND22_H} thus does not appear. This completes the proof.


\end{trivlist}

\section*{Acknowledgments}
The article was prepared within the framework of a subsidy granted to the HSE by the Government of the Russian Federation for the implementation of the Global Competitiveness Program.\\

We would like to thank the anonymous referees for helpful comments and suggestions.
\bibliographystyle{alpha}
\bibliography{bibli}

\end{document}